\documentclass[a4paper,11pt]{article}
\usepackage[latin1]{inputenc}
\usepackage{amsmath}
\usepackage{amsfonts}
\usepackage{amssymb}
\usepackage{epsfig}
\usepackage{amsopn}
\usepackage{amsthm}
\usepackage{graphicx}
\newtheorem{theorem}{Theorem}
\newtheorem{lemma}{Lemma}

\newtheorem{definition}{Definition}
\newtheorem{proposition}{Proposition}
\newtheorem{corollary}{Corollary}
\newtheorem*{theorem*}{Theorem}
\newtheorem*{lemma*}{Lemma}
\newtheorem*{remark*}{Remark}
\newtheorem*{definition*}{Definition}
\newtheorem*{proposition*}{Proposition}
\newtheorem*{corollary*}{Corollary}
\numberwithin{equation}{section} \numberwithin{theorem}{section}
\numberwithin{proposition}{section} \numberwithin{lemma}{section}
\numberwithin{definition}{section} \numberwithin{corollary}{section}

\newcommand{\real}{\mathbb{R}}

\def\a{\alpha}
\def\b{\beta}

\def\e{\varepsilon}        

\def\u{\upsilon}

\def\cb{{\cal B}}

\def\cd{{\cal D}}

\def\ck{{\cal K}}

\def\cm{{\cal M}}

\def\cp{{\cal P}}

\def\cs{{\cal S}}

\newcommand{\dx}{\,{\rm d}x}
\newcommand{\dy}{\,{\rm d}y}

\newcommand{\ds}{\,{\rm d}s}
\newcommand{\dt}{\,{\rm d}t}
\newcommand{\rd}{{\rm d}}

\def\LL{\mathrm{L}} 
\newcommand{\RR}{\mathbb{R}}

\newcommand{\ve}{\varepsilon}

\def\pa{\partial}
\def\dist{\mathrm{dist}} 
\def\qed{\,\unskip\kern 6pt \penalty 500
\raise -2pt\hbox{\vrule \vbox to8pt{\hrule width 6pt
\vfill\hrule}\vrule}\par}

\begin{document}

\title{\huge \bf  Local smoothing effects, positivity,\\
                  and Harnack inequalities for the \\
                  fast $p$\,-Laplacian equation\vspace{1cm}}

\author{\Large  Matteo Bonforte\,\footnote{\textit{e-mail:}
matteo.bonforte@uam.es}, \
Razvan Gabriel Iagar\,\footnote{\textit{e-mail:}razvan.iagar@uam.es},
\ Juan Luis V{\'a}zquez\,\footnote{\textit{e-mail:} juanluis.vazquez@uam.es} \\[5mm]
\normalsize{Departamento de Matem{\'a}ticas,} \\ \normalsize{Universidad
Aut{\'o}noma de Madrid,}\\ \normalsize{28049, Madrid, Spain}\vspace{5mm}}

\date{}

\maketitle

\begin{abstract}
We study qualitative and quantitative properties of local weak
solutions of the fast $p$-Laplacian equation, $\partial_t
u=\Delta_{p}u$, with $1<p<2$. Our main results are quantitative
positivity and boundedness estimates for locally defined solutions
in domains of $\RR^n\times [0,T]$. We combine these lower and upper
bounds in different forms of intrinsic Harnack inequalities, which
are new in the very fast diffusion range, that is when $1<p \le
2n/(n+1)$. The boundedness results may be also extended to the limit
case $p=1$, while the positivity estimates cannot.

We prove the existence as well as sharp
asymptotic estimates for the so-called large solutions  for any
$1<p<2$, and point out their main properties.

We also prove a new
local energy inequality for suitable norms of the gradients of the
solutions.  As a consequence, we prove
that bounded local weak solutions are indeed local strong solutions,
more precisely $\partial_t u\in L^2_{\rm loc}$.

\end{abstract}

\vspace{4.5cm}


\smallskip

\newpage

\section{Introduction}

In this paper we study the behaviour of  local weak
solutions of the parabolic $p$-Laplacian equation
\begin{equation}\label{PLE}
\pa_t u=\nabla\cdot(|\nabla u|^{p-2}\nabla u)
\end{equation}
in the range of exponents $1<p<2$, which is known as the fast
diffusion range. We consider weak solutions $u=u(x,t)$ defined in a space-time subdomain of $
{\RR}^{n+1}$ which we usually take to be, without loss of generality
in the results, a cylinder $Q_T=\Omega\times(0,T]$, where $\Omega$
is a domain in $\RR^n$, $n\ge 1$, and $0<T\le \infty$. The main goal
of the present paper is to establish \textsl{local upper and lower
bounds} for the nonnegative weak solutions of this equation. By local estimates we
mean estimates that hold in any compact subdomain of $Q_T$ with
bounds that do not depend on the possible behaviour of the  solution
$u$ near $\pa\Omega$ for $0\le t\le T$. Our estimates cover the
whole range $1<p<2$. The upper estimates extend to signed weak solutions
as estimates in $L^\infty_{loc}(Q_T)$.

It is well-known that fast diffusion equations, like the previous
one and other similar equations, admit local estimates, and there
are a number of partial results in the literature. For the closely
related fast-diffusion equation, $\pa_t u=\Delta (u^m)$ with
$0<m<1$, interesting local bounds were found recently by two of the
authors in \cite{BV}, including the subcritical case
$m<m_c:=(n-2)/n$, where these estimates were completely new. On the
other hand,  the theories of the porous medium/fast diffusion
equation and the $p$-Laplacian equation have strong similarities
both from the quantitative and the qualitative point of view. This
similarity is made explicit by the transformation described  in
\cite{ISV} that establishes complete equivalence of the classes of
radially symmetric solutions of both families of equations (note
that the transformation maps $m$ into $p=m+1$ and may change the
space dimension). However, the particular details of both theories
for general non-radial solutions can be quite different, and the
purpose of this paper is to make a complete analysis of the issue
for the $p$-Laplacian equation.

Let us mention that our parabolic $p$-Laplacian equation has been
widely researched for values of $p>2$, cf. \cite{DiB} and its
references, but the fast-diffusion range has been less studied, see
also \cite{DBH90, DUV, DGV}. However, just as it happens to the fast
diffusion equation for values of $m\sim 0$, the theory becomes
difficult for $p$ near 1, more precisely for $1<p<p_c=2n/(n+1)$, and
such a low range is almost absent from the literature. For the
natural occurrence of the exponent $p_c$ in the theory see for
instance \cite{EstebanVazquez} or the book \cite{VazquezSmoothing},
Chapter 11.

Some local estimates were established by Di Benedetto and Herrero in
\cite{DBH90}. We will establish here new upper and lower bounds of
local type, completing in this way these previous results, and
setting a new basis for the qualitative study of the equation in
that range.

A consequence of our local bounds from above and below is a number
of \textsl{Harnack inequalities}. The question of proving Harnack
inequalities for the fast $p$-Laplacian equation has been raised
first by DiBenedetto and Kwong in \cite{DiBK92}. This problem has
been studied again recently by DiBenedetto, Gianazza and Vespri in
\cite{DGV}, where they prove that the standard intrinsic Harnack
inequality holds for $p>p_c$ and is in general false for $p<p_c$,
and they leave as an open question the existence of Harnack
inequalities of some new form in that low range of $p$. We give a
positive answer to this intriguing open problem.

We also prove existence and sharp space-time asymptotic estimates
for the so-called large solutions $u_\infty$, namely, $u_\infty\sim
\,t^{1/(2-p)}\dist(x,\partial\Omega)^{p/(p-2)}$, for any $1<p<2$.
Moreover, we  prove a new local energy inequality for suitable norms
of the gradients of the solutions, which can be extended to more
general operators of $p$-Laplacian type. As a consequence, we obtain
that bounded local weak solutions are indeed local strong solutions,
more precisely $\partial_t u\in L^2_{\rm loc}$, cf. Corollary
\ref{CorEnergy}. This qualitative information adds an important item
to the general theory of the $p$-Laplacian type diffusions.

Some of the results and techniques may be also extended to more general
degenerate diffusion equations, as mentioned in the concluding remarks.

\

\

\noindent \textbf{Organization of the paper.} We begin with a
section where we state the definitions and the main results of the
present paper in a concentrated form. It contains: local upper
bounds for solutions, positivity estimates, Harnack inequalities and
local inequalities for the energy, i.\,e., for the gradients of the
solutions. The rest of the paper will be divided into several parts,
as follows:

\textsc{Local Smoothing Effect for $L^r$ norms}. In Section
\ref{sec.LSE}, we give the proof of Theorem \ref{mainupper0}, which
is the main Local Smoothing Effect. It is proved in a first step for
the class of bounded local strong solutions. The proof (Subsection
\ref{Proof.SE}) is obtained by joining a space-time local smoothing
effect (Subsection \ref{sec.space.time}) with an $L^r_{loc}$
stability estimate, i.\,e., we control the evolution in time of the
local $L^r$ norms, $r\ge 1$ (Subsection \ref{subsect.Lr.stab}). The
local smoothing result for general local strong solutions will be
postponed to Section \ref{EOP}.

Let us point out here that as a consequence of this result and known
regularity theory (cf. \cite{DiB} or Appendix A2), it follows that
the local strong solutions are H\"older continuous, whenever their
initial trace lies in $L^r_{loc}$ for suitable $r$.

\textsc{Continuous Large Solutions.} In Section \ref{sec.largesol},
we apply the boundedness result of Theorem \ref{mainupper0}, to
prove the existence of the so-called {\sl large solutions} for the
parabolic $p$-Laplacian equation for any $1<p<2$. We derive some of
their properties, in particular we prove a sharp asymptotic
behaviour for large times. We also construct the so-called extended
large solutions, in the spirit of \cite{ChV}.  These results are a
key tool in the proof of our sharp local smoothing effect, when
passing from bounded to general local strong solutions. Roughly
speaking, extended large solutions play the role of (quasi)
``absolute upper bounds'' for local solutions.

\textsc{Local lower bounds}. We devote Sections \ref{sec.MDP.pos}
and \ref{sec.mainposit} to establish lower estimates for local weak
solutions, in the form of quantitative positivity estimates for
small times, see Theorem \ref{mainposit}, and estimates which are
global in time, of the \textsl{Aronson-Caffarelli type}, see Theorem
\ref{main.AC}. In Section \ref{sec.MDP.pos} we prove all these facts
for a minimal Dirichlet problem, while in Section
\ref{sec.mainposit} we extend them to general continuous local weak
solutions via a technique of local comparison.

\textsc{Harnack inequalities}. In Section \ref{sec.Harnack}, we prove
forward, backward and elliptic Harnack inequalities in its intrinsic
form, cf. Theorem \ref{FBH}, together with some other alternative forms,
that avoid the delicate intrinsic geometry. This inequalities are sharp
and extend to the very fast diffusion range $1<p\le p_c$, the results of
\cite{DiBK92, DGV} valid only in the supercritical range $p_c< p<2$, for
which we give a different proof.

\textsc{A special energy inequality}. In Section \ref{Sec.Energy},
we prove a new estimate for gradients, Theorem \ref{inequality},
which, besides its application in the proof of the Local Smoothing
Effect, has several applications outlined in that section, such as
the fact that bounded local weak solutions are indeed local strong
solutions, cf. Corollary \ref{CorEnergy}. This inequality can be
extended to more general operators of $p$-Laplacian type. Let us
also mention that such a technical tool is not needed in developing
the corresponding theory for the fast diffusion equation.

\textsc{Panorama, open problems and existing literature.} In the
last section we draw a panorama of the obtained results, we pose
some open problems and we briefly compare our results with other
related works.

\section{Statements of the main results}

\subsection{The notion of solution}

We use the following definition of local weak solution,
found in the literature, cf. \cite{DiB, DUV}.

\begin{definition}\label{local.weak}
A ``local weak solution'' of {\rm \eqref{PLE}} in $Q_T$ is a measurable
function
\begin{equation*}
u\in C_{loc}\big(0,T;L^{2}_{loc}(\Omega)\big)\cap
L^{p}_{loc}\big(0,T;W^{1,p}_{loc}(\Omega)\big)
\end{equation*}
such that, for every  open bounded
 subset $K$ of \ $\Omega$ and for every time interval
$[t_1,t_2]\subset(0,T]$, the following equality holds true:
\begin{equation}\label{weak}
\int_{K}u(t_2)\varphi(t_2)\,\dx-\int_{K}u(t_1)\varphi(t_1)\,\dx
+\int_{t_1}^{t_2}\int_{K}\left(-u\varphi_{t}+|\nabla u|^{p-2}\nabla
u\cdot\nabla\varphi\right)\,\dx\,\dt=0,
\end{equation}
for any test function $\varphi\in
W^{1,2}_{loc}\big(0,T;L^{2}(K)\big)\cap
L^{p}_{loc}\big(0,T;W_{0}^{1,p}(K)\big)$. Under similar assumptions,
we say that $u$ is a local weak subsolution (supersolution) if we
replace in \eqref{weak} the equality by $\leq$ (resp. $\geq$) and we
restrict the class of test functions to $\varphi\geq0$.

A local weak solution $u$ is called ``local strong
solution'' if $u_t\in L^1_{loc}(Q_T)$, $\Delta_{p}u\in L^1_{loc}(Q_T)$
and equation \eqref{PLE} is satisfied for a.\,e. $(x,t)\in Q_T$. In the definition of
local strong sub- or super-solution we only add  the condition $u_t\in L^1_{loc}(Q_T)$,
while the requirement  $\Delta_{p}u\in L^1_{loc}(Q_T)$ is not imposed (and is in general not true).

\end{definition}
We will recall in the sequel known properties of the local weak or
strong solutions at the point where we need them. We just want to
stress the local (in space-time) character of the definition, since
there is no reference to any initial and/or boundary data taken by
the local weak solution $u$. However, \textit{in some statements
initial data are taken as initial traces in some space $L^r_{\rm
loc}(\Omega)$}, and then $u\in
C\big([0,T];L^{r}_{loc}(\Omega)\big)$.  This can be done in view of
the results of DiBenedetto and Herrero \cite{DBH90}. Let us point
our that the  $p$-Laplacian equation is invariant under constant
$u$-displacements (i.\,e., if $u$ is a local weak solution so is
$u+c$ for any $c\in \RR$). This is a quite convenient property not
shared by the porous medium/fast diffusion equation. The equation is
also invariant under the symmetry $u\mapsto -u$.

Throughout the paper we will use the fixed values of the constants
\begin{equation}
p_c=\frac{2n}{n+1}, \qquad r_c=\frac{n(2-p)}{p}, \qquad \vartheta_r=\frac{1}{rp+(p-2)n}.
\end{equation}
Note that $1<p_c<2$ for $n>1$, and $r_c>1$ for $1<p<p_c$. See figure
in Section~\ref{sec.panorama}.

Next, we state our main results. By local weak solution we will
always refer to the solutions of the fast $p$-Laplacian equation
introduced in  Definition \ref{local.weak}, defined in $Q_T$, and
with $1<p<2$.  At some places we denote by $|\Omega |$ the Lebesgue
volume of a measurable set $\Omega$, typically a ball.

\subsection{Local Smoothing Effects}

Our main result in terms of local upper estimates reads

\begin{theorem}\label{mainupper0}
Let $u$ be a local strong solution of the fast $p$-Laplacian
equation with $1<p<2$ corresponding to an initial datum $u_0\in
L^r_{\rm loc}(\Omega)$, where $\Omega\subseteq\real^n$ is an open
domain containing the ball $B_{R}(x_0)$. If either $1<p\le p_c$ and
$r>r_c$, or $p_c<p<2$ and $r\ge 1$, then there exists two positive
constants $C_1$ and $C_2$ such that:
\begin{equation}\label{mainsmooth0}
u(x_0,t)\leq\frac{C_1}{t^{n\vartheta_r}}
\left[\int_{B_{R}(x_0)}|u_{0}(x)|^{r}\,\dx\right]^{p\vartheta_r}
+C_2\left(\frac{t}{R^{p}}\right)^{\frac{1}{2-p}}.
\end{equation}
Here $C_1$  and $C_2$ depend only on $r$, $p$, $n$;  we recall that
$\vartheta_r>0$ under our assumptions.
\end{theorem}

\noindent \textbf{Remarks.} (i) We point out  that a natural choice
for  $R$  is  $R={\rm \dist}(x_0,\partial \Omega)$. In this way
reference to the inner ball can be avoided. We ask the reader to
write the equivalent statement.

\noindent (ii) As we have mentioned, using the results of Appendix
A2, we deduce that the local strong solutions are in fact locally
H\"older continuous.

\noindent (iii) This theorem will be a corollary of a slightly more
general theorem, namely Theorem \ref{mainupper}, where the constants
$C_i$ depend also on $R/R_0$. The two terms in the estimates are sharp
in a sense that will be explained after the statement of Theorem \ref{mainupper}.

\noindent (iv) Note that changing $u$ into $-u$ and applying the
same result we get a bound from below for $u$. Therefore, we can
replace $u(x_0,t)$ by $|u(x_0,t)|$ in the left-hand side of formula
\eqref{mainsmooth0}.

\noindent (v)  The above theorem extends to the limit case $p=1$
with the assumption $r>n$.

\noindent (vi) The proof of this Theorem can be extended ``as it is''
to local strong subsolutions.

\medskip

\noindent \textbf{Continuous large solutions and extended large
solutions.} The upper estimate \eqref{mainsmooth0} will be used to
prove the existence of continuous large solutions for the parabolic
$p$-Laplacian equation, cf. Theorem \ref{large-parabolic}. Moreover,
we prove sharp asymptotic estimates for such large solutions in
Theorem \ref{asympt.large}, of the form: $u(x,t) \sim O({\rm
dist}(x,\partial\Omega)^{\frac{p}{2-p}}t^{\frac{1}{2-p}})$. See
precise expression in  \eqref{asympt.large.eq}.

\subsection{Lower bounds for nonnegative solutions}

The next results deal with properties of nonnegative solutions. Note
that since the equation is invariant under constant
$u$-displacements, the results apply to any local weak solution that
is bounded below (and by symmetry $u\mapsto -u$ to any solution that
is bounded above). We divide our presentation of the results into
several different parts.

\noindent \textsc{A. General positivity estimates}. Let $u$ be a
nonnegative, continuous local weak solution of the fast
$p$-Laplacian equation in a cylinder $Q=\Omega\times (0,T)$, with
$1<p<2$, taking an initial datum $u_0\in L^1_{\rm loc}(\Omega)$. Let
$x_0\in\Omega$ be a fixed point, such that
$\dist(x_0,\partial\Omega)>5R$. Consider the \textit{minimal
Dirichlet problem}, which is the problem posed in $B_{3R}(x_0)$,
with initial data $u_0\chi_{B_{R}(x_0)}$ and zero boundary
conditions. The extinction time $T_m=T_m(u_0,R)$ of the solution of
this problem (which is always finite, as results in Subsection
\ref{FET} show) is called the \textit{minimal life time}, and indeed
it satisfies $T_m(u_0,R)<T(u)$, where $T(u)$ is the (finite or
infinite) extinction time of $u$. In order to pass from the estimate
in the center $x_0$ to the infimum in $B_{R}(x_0)$, we need that
$\dist(x_0,\partial\Omega)>5R$. With all these notations we have:

\begin{theorem}\label{mainposit}
Under the previous assumptions, there exists a positive constant
$C=C(n,p)$ such that
\begin{equation}\label{ineq.mainpos}
\inf_{x\in B_{R}(x_0)}u^{p-1}(x,t)\geq
CR^{p-n}t^{\frac{p-1}{2-p}}T_{m}^{-\frac{1}{2-p}}\int_{B_R(x_0)}u_{0}(x)\,\dx,
\end{equation}
for any $0<t<t^{*}$, where $t^{*}>0$ is a critical time depending on
$R$ and on $\|u_0\|_{L^{1}(B_R)}$, but not on $T_m$.
\end{theorem}
The explicit expression the critical time is $ t^{*}=k^{*}(n,p)
R^{p-n(2-p)}\|u_0\|_{L^{1}(B_R(x_0))}^{2-p}$, cf. \eqref{crittime2}.

The next result is a lower bound for continuous local weak
solutions, in the form of \textsl{Aronson-Caffarelli estimates}. The
main difference with respect to Theorem \ref{mainposit} is that this
estimate is global in time, and implies the first one.

\begin{theorem}\label{main.AC} Under the assumptions of the
last theorem, for any $t\in(0,T_m)$ we have
\begin{equation}\label{AC.estimates.main}
R^{-n}\int_{B_R(x_0)}u_0(x)\,\dx\leq
C_{1}t^{\frac{1}{2-p}}R^{-\frac{p}{2-p}}+
C_{2}\,t^{-\frac{p-1}{2-p}}T_{m}^{\frac{1}{2-p}}R^{-p}\inf\limits_{x\in
B_{R}(x_0)}u(x,t)^{p-1},
\end{equation}
where $C_1$ and $C_2$ are positive constants depending only on $n$ and $p$.
\end{theorem}

\noindent \textbf{Remark.} The presence of $T_m$ may seem awkward
since the extinction time is not a direct expression of the data. On
the other side, the above estimates hold with the same form in the
whole range $1<p<2$. We now improve the above estimates by replacing
the $T_m$ with some local information on the data, and for this
reason it is necessary to separate the results that hold in the
supercritical and in the subcritical range.

\medskip

\noindent \textsc{B. Improved estimates in the ``good'' fast diffusion
range}. Let us consider $p$ in the supercritical or ``good'' fast
diffusion range, i.\,e. $p_c<p<2$.
In this range, we can obtain both lower and upper estimates for
$T_m$ in terms of the local $L^1$ norm of $u_0$. We prove the following result:
\begin{theorem}\label{posit.goodFDE}
If $p_c<p<2$, we have the following upper and lower bounds for
the extinction time of the Dirichlet problem $T$ on any ball $B_R$:
\begin{equation}
c_1R^{p-n(2-p)}\|u_0\|_{L^1(B_{R/3})}^{2-p}\leq T\leq
c_2R^{p-n(2-p)}\|u_0\|_{L^1(B_{R})}^{2-p},
\end{equation}
for some $c_1$, $c_2>0$. Then, the lower estimate
\eqref{ineq.mainpos} reads
\begin{equation}\label{lowerH.goodFDE}
\inf_{x\in B_{R}(x_0)}u(x,t)\geq
C(n,p)\left(\frac{t}{R^p}\right)^{\frac{1}{2-p}}, \quad \mbox{for} \
\hbox{any} \ 0<t<t^{*}.
\end{equation}
\end{theorem}

This absolute lower bound is nothing but a lower Harnack inequality,
indeed when combined with the upper estimates of Theorem
\ref{mainposit}, it implies the elliptic, forward and even backward
inequalities, as in Theorem \ref{intr.Harnack}, or in \cite{DGV}.

\medskip

\noindent \textsc{C. Improved estimates in the very fast diffusion
range}. We now consider $1<p\le p_c$. In this range the results of
the above part B are no longer valid, since an upper estimate of
$T_m$ in terms of the $L^1$ norm of the data is not possible.
However, when $u_0\in L^r_{loc}(\Omega)$ with $r\geq r_c$, we can
estimate $T_m$ by $\|u_0\|_{L^{r_c}(B_{R})}$, cf. \eqref{upper.T1}
or \eqref{upper.T2}. In this way we obtain:
\begin{theorem}\label{posit.VFDE}
Under the running assumptions, let $1<p\le p_c$ and let $u_0\in
L^{r_c}(\Omega)$. Let $x_0\in\Omega$ and $R>0$ such that
$B_{3R}(x_0)\subset\Omega$. Then, the following Aronson-Caffarelli
type estimate holds true for any $t\in(0,T)$:
\begin{equation}
R^{-n}\|u_0\|_{L^{1}(B_R(x_0))}\leq
C_{1}\,t^{\frac{1}{2-p}}R^{-\frac{p}{2-p}}+C_2\,\|u_0\|_{L^{r_c}(B_R(x_0))}R^{-p}t^{-\frac{p-1}{2-p}}\inf\limits_{x\in
B_{R}(x_0)}u^{p-1}(x,t).
\end{equation}
Moreover we have
\begin{equation}\label{posit.indepT}
\inf_{B_{R}(x_0)}u^{p-1}(\cdot,t)\geq
C\,R^{p-n}t^{\frac{p-1}{2-p}}\|u_0\|_{\LL^{r_c}(B_{R})}^{-1}\|u_0\|_{\LL^{1}(B_{R})},
\end{equation}
for any $0<t<t^{*}$, with $t^{*}$ as in Theorem \ref{mainposit}.
\end{theorem}

\noindent \textbf{Sharpness of Theorem \ref{posit.VFDE}}. The
estimates of Theorem \ref{posit.VFDE} are sharp, in the sense that a
better estimate in terms of the $L^1$ norm of $u_0$ is impossible in
the range $1<p<p_c$. To show this, we produce the following
counterexample, imitating a similar one in \cite{BV}.

Consider first a radially symmetric function $\varphi\in
L^1(\real^n)$, with total mass 1 (i.\,e. $\int\varphi\,\dx=1$),
compactly supported and decreasing in $r=|x|$, and rescale it, in
order to approximate the Dirac mass $\delta_0$:
$\varphi_{\lambda}(x)=\lambda^n\varphi(\lambda x)$. Let $u(x,t)$ be
the solution of the Cauchy problem for the fast $p$-Laplacian
equation with initial data $\varphi$, and let $T_1>0$ be its finite
extinction time. From the scale invariance of the equation, it
follows that the solution corresponding to $\varphi_{\lambda}$ is
\begin{equation*}
u_{\lambda}(x,t)=\lambda^{n}u(\lambda\,x,\lambda^{np-2n+p}t), \quad
T_{\lambda}=T_1\lambda^{-(np-2n+p)}\to\infty \ \mbox{as} \
\lambda\to\infty,
\end{equation*}
while the initial data $\varphi_{\lambda}$ has always total mass 1.
Hence, estimating $T$ in terms of $\|u_0\|_{L^1}$ is impossible,
proving that our estimates are sharp also in the range $1<p<p_c$.

\noindent \textbf{The limit case $p\to1$.} The positivity result is
false in this case, indeed formulas \eqref{posit.indepT} and
\eqref{AC.estimates.main} degenerate for $p=1$. Moreover solutions
of the 1-Laplacian equation of the form below  clearly do not
satisfy none of the above positivity estimates. Indeed the function
\begin{equation*}
u(x,t)=(1-\lambda_{\Omega}t)_{+}\chi_{\Omega}(x), \quad
\lambda_{\Omega}=\frac{P(\Omega)}{|\Omega|}, \quad u_0=\chi_{\Omega}
\end{equation*}
is a weak solution to the total variation flow, i.e. the
1-Laplacian, whenever $\Omega$ is a set of finite perimeter
$P(\Omega)$, satisfying certain condition on the curvature of the
boundary, we refer to \cite{ACM, CBN} for further details.

\subsection{Harnack inequalities}\label{subs.Harnack}

Joining the lower and upper estimates obtained before, we can prove
intrinsic Harnack inequalities for any $1<p<2$. Let $u$ be a
nonnegative, continuous local weak solution of the fast
$p$-Laplacian equation in a cylinder $Q=\Omega\times (0,T)$, with
$1<p<2$, taking an initial datum $u_0\in L^r_{\rm loc}(\Omega)$,
where $r\geq\max\{1,r_c\}$. Let $x_0\in\Omega$ be a fixed point, and
let $\dist(x_0,\partial\Omega)>5R$. We have
\begin{theorem}\label{FBH}
Under the above conditions,  there exists constants $h_1\,,h_2$
depending only on $d,p,r$, such that, for any $\varepsilon\in[0,1]$  the following inequality holds

\begin{equation}\label{intr.Harnack}
\inf\limits_{x\in B_{R}(x_0)}u(x,t\pm\theta)\geq
h_1\e^{\frac{rp\vartheta_r}{2-p}}\left[\frac{\|u(t_0)\|_{L^1(B_R)}R^{\frac{n}{r}}}
    {\|u(t_0)\|_{L^r(B_R)}R^{n}}\right]^{rp\vartheta_r+\frac{1}{2-p}}u(x_0,t),
\end{equation}
for any
\begin{equation*}
t_0+\e t^{*}(t_0)<t\pm\theta<t_0+t^{*}(t_0), \quad
t^{*}(t_0)=h_2R^{p-n(2-p)}\|u(t_0)\|_{L^{1}(B_R(x_0))}^{2-p}.
\end{equation*}
\end{theorem}
The proof of this theorem is given in Section \ref{sec.Harnack},
together with an alternative form that avoids the intrinsic
geometry.

\noindent \textbf{Remarks.} (i) In the ``good fast-diffusion range''
$p>p_c$, we can let $r=1$ and we recover the intrinsic Harnack
inequality of \cite{DGV}, that is
\begin{equation*}
\inf\limits_{x\in B_{R}(x_0)}u(x,t\pm\theta)\geq
h_1\e^{\frac{rp\vartheta_r}{2-p}}u(x_0,t),
\qquad\mbox{for any }t_0+\e t^{*}(t_0)<t\pm\theta<t_0+t^{*}(t_0).
\end{equation*}
Let us notice that in this inequality, the ratio of $L^r$ norms
simplifies, and the constants $h_1,h_2$ do not depend on $u_0$. The
size of the intrinsic cylinder is given by $t^*$ as above, in
particular we observe that $t^{*}(t_0)\sim R^{p-n(2-p)}\|u(t_0)\|_{L^{1}(B_R(x_0))}^{2-p}
\sim R^p u(t_0,x_0)^{2-p}$.

\noindent (ii) In the subcritical range $p\le p_c$, the Harnack
inequality cannot have a universal constant, independent of $u_0$,
cf. \cite{DGV}. We have thus shown that, if one allows for the
constant to depend on $u_0$, we obtain an intrinsic Harnack
inequality, which is a natural continuation of the one in the good
range $p>p_c$.  The size of the intrinsic cylinders is proportional
to a ratio of local $L^r$ norms, but this ratio
simplifies only when $p>p_c$.

\noindent (iii) We also notice that we need a small waiting time
$\e\in(0,1]$. This waiting time is necessary for
the regularization to take place, and thus for the intrinsic
inequality to hold, and it can be taken as small as we wish.

\noindent (iv) The backward Harnack inequality, i.\,e., estimate
\eqref{intr.Harnack} taken at time $t-\theta$, is typical of the
fast diffusion processes, reflecting an important feature that these
processes enjoy, that is extinction in finite time, the solution
remaining positive until the finite extinction time. It is easy to
see that the backward Harnack inequality does not hold either for
the linear heat equation, i.\,e. $p=2$, or for the degenerate
$p$-Laplacian equation, i.\,e. $p>2$.

\noindent (v) \textit{The Size of Intrinsic Cylinders.} The critical
time $t^{*}(t_0)$ above represents the size of the intrinsic
cylinders. In the supercritical fast diffusion range this time can
be chosen ``a priori'' just in terms of the initial datum at
$t_0=0$, but in the subcritical range its size must change with
time; roughly speaking the diffusion is so fast that the local
information at $t_0$ is not relevant after some time, which is
represented by $t^*(t_0)$. We must bear in mind that a large class
of solutions  completely extinguish in finite time.

\subsection{Special local energy inequality}

\begin{theorem}\label{inequality}
Let $u$ be a continuous local weak solution of the fast
$p$-Laplacian equation in a cylinder  $Q=\Omega\times (0,T)$, with
$1< p<2$, in the sense of Definition \ref{local.weak}, and let $0\le
\varphi\in C^{2}_c(\Omega)$ be any admissible test function. Then
the following inequality holds:
\begin{equation}\label{mainid}
\frac{\rd}{\dt}\int_{\Omega}|\nabla
u|^{p}\varphi\,\dx+\frac{p}{n}\int_{\Omega}(\Delta_{p}u)^{2}\varphi\,\dx\leq \frac{p}{2}\int_{\Omega}
|\nabla u|^{2(p-1)}\Delta\varphi\,\dx,
\end{equation}
in the sense of distributions in $\cd^{'}(0,T)$.
\end{theorem}

Beyond the interest in itself, Theorem \ref{inequality} has the
following consequence that will be important in the sequel:
\begin{corollary}\label{CorEnergy}
Let $u$ be a continuous local weak solution. Then
$u_t=\Delta_{p}u\in L^2_{loc}(Q_T)$, in particular $u$ is a local
strong solution in the sense of Definition \ref{local.weak}.
\end{corollary}

We present here a short formal calculation that leads
to the inequality \eqref{mainid}. The complete and rigorous proof
of Theorem \ref{inequality} is longer and technical and will be given
in Section \ref{Sec.Energy}.

\noindent\textbf{Formal Proof of Theorem \ref{inequality}.} We start
by differentiating the energy, localized with an admissible test
function $\varphi\ge 0$
\begin{equation}\label{der.1}
\begin{split}
\frac{\rd}{\dt}\int_{\Omega}|\nabla
u|^{p}\varphi\,\dx&=p\int_{\Omega}(|\nabla u|^{p-2}\nabla
u\cdot\nabla u_{t})\varphi\,\dx\\
&=-p\int_{\Omega}\hbox{div}(|\nabla u|^{p-2}\nabla
u\varphi)\Delta_{p}u\,\dx\\
&=-p\int_{\Omega}(\Delta_{p}u)^{2}\varphi\,\dx-p\int_{\Omega}\Delta_{p}u\,|\nabla
u|^{p-2}\nabla u\cdot\nabla\varphi\,\dx.
\end{split}
\end{equation}
Next, we estimate the last term in the above calculation. To this
end, we use the following formula, (cf. also \cite{CLMT})
\begin{equation}\label{formula}
\left(\hbox{div}F\right)^{2}=\hbox{div}\left(F\hbox{div}F\right)
-\frac{1}{2}\Delta\left(|F|^{2}\right)+\hbox{Tr}\left[\left(\frac{\partial
F}{\partial x}\right)^{2}\right],
\end{equation}
which holds true for any vector field $F$. We combine it with the
following inequality
\begin{equation}\label{ineq}
\hbox{Tr}\left[\left(\frac{\partial F}{\partial
x}\right)^{2}\right]\geq\frac{1}{n}\left[\hbox{Tr}\left(\frac{\partial
F}{\partial
x}\right)\right]^{2}=\frac{1}{n}\left(\hbox{div}F\right)^{2}.
\end{equation}
and we then apply these to the vector field $F=|\nabla
u|^{p-2}\nabla u$. We obtain
\begin{equation*}
(\Delta_{p}u)^{2}
    \geq\hbox{div}\big[|\nabla u|^{p-2}\nabla u\,\hbox{div}(|\nabla u|^{p-2}\nabla u)\big]
        -\frac{1}{2}\Delta\left(|\nabla u|^{2(p-1)}\right)
        +\frac{1}{n}(\Delta_{p}u)^{2}.
\end{equation*}
We multiply by $\varphi$ and integrate the above inequality in
space, then we plug it into \eqref{der.1}, and thus get
\begin{equation*}
\begin{split}
\frac{1}{p}\frac{\rd}{\dt}\int_{\Omega}|\nabla
u|^{p}\varphi\,\dx&=-\int_{\Omega}(\Delta_{p}u)^{2}\varphi\,\dx-
\int_{\Omega}(\hbox{div}F)(F\cdot\nabla\varphi)\,\dx\\
&=-\int_{\Omega}(\Delta_{p}u)^{2}\varphi\,\dx+\int_{\Omega}\hbox{div}(F\hbox{div}F)\varphi\,\dx\\
&\leq-\frac{1}{n}\int_{\Omega}(\Delta_{p}u)^{2}\varphi\,\dx+\frac{1}{2}\int_{\Omega}\Delta\left(|\nabla
u|^{2(p-1)}\right)\varphi\,\dx
\end{split}
\end{equation*}
where the notation $F=|\nabla u|^{p-2}\nabla u$ is kept for sake of
simplicity. A double integration by parts in the last term gives
(\ref{mainid}). Let us finally notice that, in order to perform the
integration by parts in the last inequality step above, we need that
$\varphi=0$ and $\nabla\varphi=0$ on $\partial\Omega$. \qed

\noindent \textbf{Remarks.}  (i)  The second term in the left-hand
side can also be written as $(p/n)\int_{\Omega}u_t^{2}
\,\varphi\,\dx$ and accounts for local dissipation of the `energy
integral' of the left-hand side.  This result continues to hold and
it is well known for the linear heat equation, i.\,e., when $p=2$.

\noindent (ii) Theorem \ref{inequality} may be extended to more
general operators, the so-called $\Phi$-Laplacians, under suitable
conditions, we refer to Proposition \ref{Ineq.Phi} and to the remarks at the
end of Section \ref{Sec.Energy} for such extensions.

\noindent (iii) Inequality \eqref{mainid} is  new and holds also in
the limit $p\to 1$ at least formally. In any case, our proof relies
on some results concerning regularity that fail when $p=1$. When
$p\to 1$ our inequality reads
\begin{equation*}
\frac{\rd}{\dt}\int_{\Omega}|\nabla
u|\varphi\,\dx +\frac{1}{n}\int_{\Omega}(\Delta_{1}u)^{2}\varphi\,\dx\le 0,
\end{equation*}
in $\cd'(0,T)$, showing in particular that the local energy, in this
case the local total variation associated to the 1-Laplacian (or
total variation flow) decays in time with some rate. This inequality
can be helpful when studying the asymptotic of the total variation
flow, a difficult open problem that we do not attack here. A
slightly different version of this inequality for $p=1$ is proven in
\cite{ACM} in the framework of entropy solutions, and is the key
tool in proving the $\LL^2_{loc}$ regularity of the time derivative
of entropy solutions.

\section{Local smoothing effect for bounded strong solutions}\label{sec.LSE}

We turn now to the proof of Theorem \ref{mainupper0}, that will be
divided into two parts: first, we prove it for bounded strong
solutions, then (in Section \ref{EOP}) we prove the result in the
whole generality, for any local strong solution. The result of
Theorem \ref{mainupper0} is obtained as an immediate corollary of
the following slightly stronger form of the result.

\begin{theorem}\label{mainupper}
Let $u$ be a local strong solution of the fast $p$-Laplacian
equation, with $1<p<2$, as in Definition \ref{local.weak},
corresponding to an initial datum $u_0\in L^r_{\rm loc}(\Omega)$,
where $\Omega\subseteq\real^n$ is any open domain containing the
ball $B_{R_0}(x_0)$. If either $1<p\le p_c$ and $r>r_c$ or $p_c<p<2$
and $r\ge 1$, then there exist two positive constants $C_1$ and
$C_2$ such that for any $0<R<R_0$ we have:
\begin{equation}\label{mainsmooth}
\sup_{(x,\tau)\in B_{R}\times(\tau_0
t,t]}u(x,\tau)\leq\frac{C_1}{t^{n\vartheta_r}}
\left[\int_{B_{R_0}}|u_{0}(x)|^{r}\,\dx\right]^{p\vartheta_r}
+C_2\left(\frac{t}{R_{0}^{p}}\right)^{\frac{1}{2-p}},
\end{equation}
where $\tau_0=[(R_0-R)/(R_0+R)]^p$ and
\begin{equation*}
\begin{split}
&C_1=K_1\left[\frac{R_0+R}{R_0-R}\right]^{p(n+p)\vartheta_r}, \quad
C_2=K_2\left[\frac{R_0+R}{R_0-R}\right]^{p(n+p)\vartheta_r}
    \left[K_3\left(\frac{R_0-R}{R_0+R}\right)^{\frac{2-p-rp}{2-p}}+K_4\right]^{p\vartheta_r},
\end{split}
\end{equation*}
with $K_i$, $i=1,2,3$, depending only on $r$, $p$, $n$, and
$K_4=\omega_n$ if $r>1$, $K_4=\omega_n+L(n,p)>\omega_n$ if $r=1$.
$\omega_n$ is the measure of the unit ball of $\real^n$, and we
recall that $\vartheta_r=1/[rp+(p-2)n]>0$.
\end{theorem}

Note that Theorem \ref{mainupper0} follows immediately from Theorem
\ref{mainupper} by letting $R=0$ and considering, for $x_0\in\Omega$
fixed, the ball centered in $x_0$ and tangent to $\partial\Omega$.

\medskip

\noindent\textbf{Interpreting the two terms in the estimate. } The
right-hand side of (\ref{mainsmooth}) is the sum of two independent
terms. Let us discuss them separately.

\noindent (i) The first term concentrates the influence of the
initial data $u_0$. It has the exact form of the global smoothing
effect (i.\,e. the smoothing effect for solutions defined in the
whole space with initial data in $L^r(\real^n)$ or in the
Marcinkiewicz space $\cm^r(\RR^n)$), cf. Theorem 11.4 of
\cite{VazquezSmoothing}. Hence, if we pass to the limit in
(\ref{mainsmooth}) as $R_0\to\infty$, we recover the global
smoothing effect on $\real^n$ (however, the constant need not to be
optimal).

\noindent (ii) The second term appears as a correction term when
passing from global estimates to local upper bounds. It can be
interpreted as an {\sl absolute damping} of all external influences
due to the form of the diffusion operator, more precisely, due to
fast diffusion. Let us note that, by shrinking the ball $B_{R_0}$
(and at the same time the smaller ball $B_R$), the influence of this
term increases, while that of the first one tends to disappear.

A remarkable consequence of this absolute damping is the existence
of large solutions that we will discuss in Section
\ref{sec.largesol}. Indeed, there is an explicit large solution with
zero initial data that has precisely the form of the last term in
\eqref{mainsmooth} with $R=0$ -- or in the corresponding term in
\eqref{mainsmooth0} -- which means that such term has an optimal
form that cannot be improved without information on the boundary
data (again, the constant need not to be optimal).

\medskip

We first prove Theorem \ref{mainupper} for bounded local strong
solutions, then we will remove the assumption of local boundedness
in Section \ref{EOP}. The proof of Theorem \ref{mainupper} for
bounded local strong solutions consists of combining
$\LL_{loc}^r$-stability estimates, together with a space-time local
smoothing effect, proved via Moser-style iteration. This will be the
subject of the next subsections.

\subsection{Space-time local smoothing
effects}\label{sec.space.time}

In this section we prove a form of the Local Smoothing Effect for
the $p$-Laplacian equation, with $1<p<2$. More precisely, we are
going to prove that $L^{r}_{loc}$ regularity in space-time for some
$r\ge 1$ implies $L^{\infty}_{loc}$ estimates in space-time.
\begin{theorem}\label{localsm}
Let $u$ be a bounded local strong solution of the $p$-Laplacian
equation, $1<p<2$, and let either $1<p\le p_c$ and $r>r_c$ or
$p_c<p<2$ and $r\ge 1$. Then, for any two parabolic cylinders
$Q_1\subset Q$, where $Q=B_{R_0}\times(T_0,T]$ and
$Q_1=B_{R}\times(T_1,T]$, with $0<R<R_0$ and $0\leq T_0<T_1<T$, we
have:
\begin{equation}\label{localsmooth}
\sup_{Q_1}|u|\leq
K\,\left[\frac{1}{(R_0-R)^{p}}+\frac{1}{T_1-T_0}\right]^{\frac{p+n}{rp+n(p-2)}}\left(\int\!\!\!\int_{Q}u^r\,\dx\,\dt+|Q|\right)^{\frac{p}{rp+n(p-2)}},
\end{equation}
where $K>0$ is a constant depending only on $r$, $p$ and $n$.
\end{theorem}

\noindent \textbf{Remarks}: (i) Under the assumptions of Theorem
\ref{localsm}, the local boundedness in terms of some space-time
integrals of the solution $u$ is proved as Theorem 3.8 in
\cite{DUV}. In this section, we only give a slight, quantitative
improvement of it, which in fact appears in this form in \cite{DGV},
but only for the``good'' range $p_c<p<2$ and for $L^1_{loc}$ initial
data. We prove it here for all $1<p<2$.

The proof of our main local upper bound, in the form of Theorems \eqref{mainsmooth0}
or \eqref{mainsmooth} is given in Subsection \ref{Proof.SE} , in which we combine the above time-space
smoothing effect \eqref{mainsmooth} with the $L^r_{loc}$ stability estimates \eqref{Lr.stab} of session \ref{subsect.Lr.stab}.

\noindent (ii) This space-time Smoothing Effect holds also for the
equation with bounded variable coefficients, as well as for more
general operators such as $\Phi$-Laplacians or as in the general
framework treated in \cite{DGV}. We are not addressing this
generality since the rest of the theory is not immediate.

\medskip

We divide the proof into several steps, following the same general
program used by two of the authors in \cite{BV} for the fast
diffusion equation.

\noindent\textbf{Step 1. A space-time energy inequality}

\noindent Let us consider a bounded local strong solution $u$
defined in a parabolic cylinder $Q=B_{R_0}\times(T_0,T]$. Take
$R<R_0$, $T_1\in(T_0,T]$ and consider a smaller cylinder
$Q_1=B_{R}\times(T_1,T]\subset Q$. Under these assumptions, we
prove:
\begin{lemma}\label{space-timeenergy}
For every $1<p<2$ and $r>1$, the following inequality holds:
\begin{equation}
\begin{split}
\int_{B_R}u^{r}(x,T)\,\dx+\int\!\!\!\int_{Q_1}\left|\nabla
u^{\frac{p+r-2}{p}}\right|^{p}\,\dx\,\dt&\leq
C(p,r,n)\left[\frac{1}{(R_0-R)^p}+\frac{1}{T_1-T_0}\right]\\
&\times\left[\int\!\!\!\int_{Q}(u^r+u^{r+p-2})\,\dx\,\dt\right].
\end{split}
\end{equation}
The same holds also for local subsolutions in the sense of
Definition \ref{local.weak}.
\end{lemma}
\begin{proof}
(i) We multiply first the $p$-Laplacian equation by
$u^{r-1}\varphi^{p}$, where   $\varphi=\varphi(x,t)$ is a smooth
test function with compact support. Integrating in $Q$ we obtain:
\begin{equation*}
\begin{split}
\int\!\!\!\int_{Q}u^{r-1}u_t&\varphi^{p}\,\dx\,\dt=\int\!\!\!\int_{Q}u^{r-1}\Delta_{p}u\varphi^{p}\,\dx\,\dt
    =-\int\!\!\!\int_{Q}|\nabla u|^{p-2}\nabla u\cdot\nabla(u^{r-1}\varphi^{p})\,\dx\,\dt\\
    &=-(r-1)\int\!\!\!\int_{Q}|\nabla u|^{p}u^{r-2}\varphi^{p}\,\dx\,\dt
        -p\int\!\!\!\int_{Q}u^{r-1}|\nabla u|^{p-2}\nabla u\cdot\varphi^{p-1}\nabla\varphi\,\dx\,\dt\\
    &=-(r-1)\int\!\!\!\int_{Q}\left|u^{\frac{r-2}{p}}\nabla u\right|^{p}\varphi^{p}\,\dx\,\dt
        -p\int\!\!\!\int_{Q}u^{r-1}|\nabla u|^{p-2}\nabla u\cdot\varphi^{p-1}\nabla\varphi\,\dx\,\dt,
\end{split}
\end{equation*}
hence
\begin{equation}\label{spt}
\int\!\!\!\int_{Q}\left[u^{r-1}u_t+\frac{(r-1)p^p}{(r+p-2)^p}\left|\nabla
u^{\frac{r+p-2}{p}}\right|^{p}\right]\varphi^{p}\,\dx\,\dt=-p\int\!\!\!\int_{Q}|\nabla
u|^{p-2}u^{r-1}\varphi^{p-1}\nabla u\cdot\nabla\varphi\,\dx\,\dt.
\end{equation}
In order to estimate the term in the right-hand side we use the
inequality $ \overrightarrow{a}\cdot\overrightarrow{b}\le
\frac{|\overrightarrow{a}|^{\sigma}}{\e
\sigma}+\frac{|\overrightarrow{b}|^{\gamma}}{\gamma}\e^{\frac{\gamma}{\sigma}}$,
which holds for any vectors $\overrightarrow{a}$,
$\overrightarrow{b}$, for any $\e >0$ and for any exponents $\sigma$
and $\gamma$ with $\sigma^{-1}+\gamma^{-1}=1$. The choice of vectors
and exponents as below
\begin{equation*}
\overrightarrow{a}=u^{\frac{r+p-2}{p}}\nabla\varphi, \
\sigma=p/(p-1), \
\overrightarrow{b}=u^{\frac{(r-2)(p-1)}{p}}\,\varphi^{p-1}\,|\nabla
u|^{p-2} \nabla u, \ \gamma=p \,,
\end{equation*}
lead to
\begin{equation*}
\begin{split}
-p\int\!\!\!\int_{Q}|\nabla u|^{p-2}u^{r-1}\varphi^{p-1}\nabla
u\cdot\nabla\varphi\,\dx\,\dt&\leq\frac{(p-1)p^{p}}{(r+p-2)^{p}}\e^{\frac{1}{p-1}}
\int\!\!\!\int_{Q}\left|\nabla
u^{\frac{r+p-2}{p}}\right|^{p}\varphi^{p}\,\dx\,\dt\\&+\frac{1}{\e}\int\!\!\!\int_{Q}u^{r+p-2}|\nabla\varphi|^{p}\,\dx\,\dt.
\end{split}
\end{equation*}
On the other hand, we integrate the first term by parts with respect
to the time variable:
\begin{equation*}
\begin{split}
\frac{1}{r}\int\!\!\!\int_{Q}&\partial_{t}(u^{r})\varphi^{p}\,\dx\,\dt
=\frac{1}{r}\int_{0}^{T}\int_{B_{R_0}}\partial_{t}(u^{r})\varphi^{p}\,\dx\,\dt\\
&=\frac{1}{r}\int_{B_{R_0}}\Big[u(x,T)^{r}\varphi(x,T)^{p}-u(x,0)^{r}\varphi(x,0)^{p}\Big]\,\dx-\frac{p}{r}\int\!\!\!\int_{Q}u^{r}\varphi^{p-1}\partial_{t}\varphi\,\dx\,\dt.
\end{split}
\end{equation*}
Joining  equality (\ref{spt}) and the previous estimates, and
choosing
\begin{equation*}
\e=\left(\frac{r-1}{r+p-2}\right)^{p-1},
\end{equation*}
we obtain:
\begin{equation*}
\begin{split}
&\frac{1}{r}\int_{B_{R_0}}\Big[u(x,T)^{r}\varphi(x,T)^{p}-u(x,0)^{r}\varphi(x,0)^{p}\Big]\,\dx-\frac{p}{r}\int\!\!\!\int_{Q}u^{r}\varphi^{p-1}\partial_{t}\varphi\,\dx\,\dt\\
&+\frac{(r-1)^{2}p^p}{(r+p-2)^{p+1}}\int\!\!\!\int_{Q}\left|\nabla
u^{\frac{r+p-2}{p}}\right|^{p}\varphi^{p}\,\dx\,\dt\leq\left(\frac{r+p-2}{r-1}\right)^{p-1}\int\!\!\!\int_{Q}u^{r+p-2}|\nabla\varphi|^{p}\,\dx\,\dt.
\end{split}
\end{equation*}

\noindent (ii) We now impose some additional conditions on
$\varphi$, namely we assume that $0\leq\varphi\leq1$ in $Q$,
$\varphi\equiv0$ outside $Q$ and $\varphi\equiv1$ in
$\overline{Q_1}$. Moreover, we ask $\varphi$ to satisfy:
\begin{equation*}
|\nabla\varphi|\leq\frac{C(\varphi)}{R_0-R}, \qquad
|\partial_{t}\varphi|\leq\frac{C(\varphi)^p}{T_1-T_0}
\end{equation*}
in the annulus $B_{R_0}\setminus B_{R}$, and $\varphi(x,0)=0$ for
any $x\in B_{R}$. With these notations, we can continue the previous
estimates as follows:
\begin{equation*}
\begin{split}
\frac{1}{r}&\int_{B_{R_0}}u(x,T)^{r}\varphi(x,T)^{p}\,\dx
        +\frac{(r-1)^{2}p^p}{(r+p-2)^{p+1}}
        \int\!\!\!\int_{Q}\left|\nabla u^{\frac{r+p-2}{p}}\right|^{p}\varphi^{p}\,\dx\,\dt\\
    &\leq\frac{p}{r}\int\!\!\!\int_{Q}u^{r}\varphi^{p-1}|\partial_{t}\varphi|\,\dx\,\dt
        +\left(\frac{r+p-2}{r-1}\right)^{p-1}\int\!\!\!\int_{Q}u^{r+p-2}|\nabla\varphi|^{p}\,\dx\,\dt\\
    &\leq C^{p}\left[\frac{p}{r}\frac{1}{T_1-T_0}+\left(\frac{r+p-2}{r-1}\right)^{p-1}
        \frac{1}{(R_0-R)^p}\right]\int\!\!\!\int_{Q}\left(u^r+u^{r+p-2}\right)\,\dx\,\dt\\
    &\leq2C^{p}\max\left\{\frac{p}{r},\left(\frac{r+p-2}{r-1}\right)^{p-1}\right\}
        \left[\frac{1}{T_1-T_0}+\frac{1}{(R_0-R)^p}\right]
        \int\!\!\!\int_{Q}\left(u^r+u^{r+p-2}\right)\dx\dt\,.
\end{split}
\end{equation*}
We now estimate the left-hand side of the
last inequality, in view of the properties of $\varphi$:
\begin{equation*}
\begin{split}
\frac{1}{r}\int_{B_{R_0}}&u(x,T)^{r}\varphi(x,T)^{p}\,\dx+\frac{(r-1)^{2}p^p}{(r+p-2)^{p+1}}
        \int\!\!\!\int_{Q}\left|\nabla u^{\frac{r+p-2}{p}}\right|^{p}\varphi^{p}\,\dx\,\dt\\
    &\geq\frac{1}{r}\int_{B_R}u^{r}(x,T)\,\dx+\frac{(r-1)^{2}p^p}{(r+p-2)^{p+1}}
        \int\!\!\!\int_{Q_1}\left|\nabla u^{\frac{r+p-2}{p}}\right|^{p}\,\dx\,\dt\\
    &\geq\min\left\{\frac{1}{r},\frac{(r-1)^{2}p^p}{(r+p-2)^{p+1}}\right\}
        \left(\int_{B_R}u(x,T)^{r}\,\dx
        +\int\!\!\!\int_{Q_1}\left|\nabla u^{\frac{r+p-2}{p}}\right|^{p}\,\dx\,\dt\right).
\end{split}
\end{equation*}
Joining all the previous calculations, we arrive to the conclusion.
The same proof can be repeated also for local subsolutions as in Definition
\ref{local.weak}.
\end{proof}

\noindent \textbf{Remark}: A closer inspection of the above proof
allows us to evaluate the constant $C(r,p)$ in a more precise way.
Indeed, we observe that
\begin{equation*}
C(r,p)=2C(\varphi)\max\left\{\frac{p}{r},\left(\frac{r+p-2}{r-1}\right)^{p-1}\right\}
\min\left\{\frac{1}{r},\frac{(r-1)^{2}p^p}{(r+p-2)^{p+1}}\right\}^{-1}.
\end{equation*}
By evaluating the dependence in $r$ of the constants, we remark
that, for $r$ sufficiently large, $C(r,p)=O(r)$. Hence, $C(r,p)$ is
bounded from below by a constant independent of $r$, but from above
it is not. In any case, as we will see, the rate $C(r,p)=O(r)$ is
good for our aims. We will use the space-time energy inequality in
the following improved version.

\begin{lemma}
Under the running assumptions, we have:
\begin{equation}\label{improv1}
\begin{split}
\sup_{s\in(T_{1},T)}\int_{B_R}u^{r}(x,s)\,\dx
&+\int_{T_1}^{T}\int_{B_R}\left|\nabla
u^{\frac{r+p-2}{p}}\right|^{p}\\&\leq
C(r,p)\left[\frac{1}{T_1-T_0}+\frac{1}{(R_0-R)^{p}}\right]
\int_{T_0}^{T}\int_{B_{R_0}}\left(u^{r+p-2}+u^r\right)\,\dx\,\dt.
\end{split}
\end{equation}
Moreover, if $u$ is a weak subsolution and $u\geq1$, we have:
\begin{equation}\label{improv2}
\begin{split}
\sup_{s\in(T_1,T)}\int_{B_R}u^{r}(x,s)\,\dx&+\int_{T_1}^{T}\int_{B_R}\left|\nabla
u^{\frac{r+p-2}{p}}\right|^{p}\,\dx\,\dt\\&\leq
C(r,p)\left[\frac{1}{T_1-T_0}+\frac{1}{(R_0-R)^{p}}\right]\int_{T_0}^{T}\int_{B_{R_0}}u^r\,\dx\,\dt.
\end{split}
\end{equation}
\end{lemma}

\begin{proof}
We recall a well known property of the supremum,
namely there exists $t_0\in(T_1,T)$ such that
\begin{equation*}
\frac{1}{2}\sup_{s\in(T_{1},T)}\int_{B_R}u^{r}(x,s)\,\dx\leq\int_{B_R}u^{r}(x,t_0)\,\dx.
\end{equation*}
Since $T_0\leq T_1<t_0<T$, we can apply Lemma \ref{space-timeenergy}
and obtain
\begin{equation*}
\int_{B_R}u^{r}(x,t_0)\,\dx\leq
C(r,p)\left[\frac{1}{T_1-T_0}+\frac{1}{(R_0-R)^{p}}\right]
\int_{T_0}^{t_0}\int_{B_{R_0}}\left(u^{r+p-2}+u^r\right)\,\dx\,\dt.
\end{equation*}
On the other hand, also the second term can be estimated by Lemma
\ref{space-timeenergy}:
\begin{equation*}
\int_{T_1}^{T}\int_{B_R}|\nabla
u^{\frac{r+p-2}{p}}|^{p}\,\dx\,\dt\leq
C(r,p)\left[\frac{1}{T_1-T_0}+\frac{1}{(R_0-R)^{p}}\right]\int_{T_0}^{T}\int_{B_{R_0}}(u^{r+p-2}+u^r)\,\dx\,\dt.
\end{equation*}
Summing up the two previous inequalities, we obtain inequality
(\ref{improv1}). If $u\geq1$, is a subsolution, then $u^{r+p-2}\leq
u^{r}$, hence we immediately get (\ref{improv2}).
\end{proof}

\noindent \textbf{Step 2: An iterative form of the Sobolev
inequality}

\noindent We state the classical Sobolev inequality in a different
form, adapted for the Moser-type iteration.
\begin{lemma}
Let $f\in L^{p}(Q)$ with $\nabla f\in L^{p}(Q)$. Then, for any
$\sigma\in(1,\sigma^{*})$, for any $0\leq T_0<T_1$ and $R>0$, the
following inequality holds:
\begin{equation}\label{Sob}\begin{split}
\int_{T_0}^{T_1}\int_{B_R}f^{p\sigma}\,\dx\,\dt
&\leq2^{p-1}\cs_{p}^{p}\left[\int_{T_0}^{T_1}\int_{B_R}\big(f^{p}+R^{p}\left|\nabla
f\right|^{p}\big)\,\dx\,\dt\right]\\
&\times\sup_{t\in(T_0,T_1)}\left[\frac{1}{R^{n}}\int_{B_R}f^{p(\sigma-1)q}(t,x)\,\dx\right]^{\frac{1}{q}},
\end{split}
\end{equation}
where
\begin{equation*}
p^{*}=\frac{np}{n-p}, \qquad
\sigma^{*}=\frac{p^{*}}{p}=\frac{n}{n-p}, \qquad
q=\frac{p^{*}}{p^{*}-p}=\frac{n}{p},
\end{equation*}
and the constant $\cs_{p}$ is the constant of the classical Sobolev
inequality.
\end{lemma}

\begin{proof}
We first prove the inequality for $R=1$. We write:
\begin{equation*}
\int_{B_1}f^{p\sigma}\,\dy=\int_{B_1}f^{p}f^{p(\sigma-1)}\,\dy
\leq\left(\int_{B_1}f^{p^{*}}\,\dy\right)^{\frac{p}{p^{*}}}
\left(\int_{B_1}f^{p(\sigma-1)q}\,\dy\right)^{\frac{1}{q}},
\end{equation*}
We use now the standard Sobolev inequality in the first factor of
the right-hand side:
\begin{equation*}
\|f\|_{p^{*}}^{p}\leq\cs_{p}^{p}(\|f\|_{p}+\|\nabla f\|_{p})^{p}\leq
2^{p-1}\cs_{p}^{p}(\|f\|_{p}^{p}+\|\nabla f\|_{p}^{p}).
\end{equation*}
Passing to the supremum in time in the second factor of the
right-hand side, then integrating  the inequality in time, over
$(T_0,T_1)$, we obtain the desired form for $R=1$. Finally, the
change of variable $x=Ry$ allow to obtain (\ref{Sob}) for any $R>0$.
\end{proof}

\noindent \textbf{Step 3: Preparation of the iteration}.

\noindent Let us first define $v(x,t)=\max\{u(x,t),1\}$. We remark that $v$ is a
local weak subsolution of the $p$-Laplacian equation in the sense of
Definition \ref{local.weak}. Moreover $u\leq v\leq 1+v$ for any
$(x,t)\in Q$.

We now let $f^p=v^{r+p-2}$ in the iterative Sobolev inequality
(\ref{Sob}) and we apply it for $Q_1\subset Q$ as  in the statement
of Theorem \ref{localsm}. We then obtain:
\begin{equation}\label{start1}
\begin{split}
\int_{T_1}^{T}\int_{B_{R}}v^{\sigma(r+p-2)} \dx\,\dt&\leq
2^{p-1}\cs_{p}^{p}\left[\int\!\!\!\int_{Q}\left(v^{r+p-2}+R^{p}\left|\nabla
v^{\frac{r+p-2}{p}}\right|^{p}\right)\,\dx\,\dt\right]\\&\times\left[\sup_{t\in[T_1,T]}\frac{1}{R^n}\int_{B_R}v^{(r+p-2)(\sigma-1)q}\,\dx\right]^{\frac{1}{q}}.
\end{split}
\end{equation}
Since $v\geq1$, we can use the space-time energy inequality
(\ref{improv2}) to estimate both terms in the right-hand side:
\begin{equation}\label{est.1}
\begin{split}
\int\!\!\!\int_{Q}\Big(v^{r+p-2}
    &+\left. R^{p}\left|\nabla v^{\frac{r+p-2}{p}}\right|^{p} \right)\,\dx\,\dt\\
        &\leq R^{p}C(r,p)\left[\frac{1}{(R_0-R)^p}+\frac{1}{T_1-T_0}\right]
        \left[\int\!\!\!\int_{Q}v^{r}\,\dx\,\dt\right]
     +\int\!\!\!\int_{Q}v^{r}\,\dx\,\dt\\
        &\leq 2R^{p}C(r,p)\left[\frac{1}{(R_0-R)^p}+\frac{1}{T_1-T_0}\right]
    \left[\int\!\!\!\int_{Q}v^{r}\,\dx\,\dt\right].
\end{split}
\end{equation}
As for the other term, we use again the space-time energy inequality
(\ref{improv2}), but we replace $r>1$ by $(r+p-2)(\sigma-1)q>1$, and we get
\begin{equation}\label{est.2}
\begin{split}
\sup_{t\in[T_1,T]}\frac{1}{R^n}\int_{B_R}v^{(r+p-2)(\sigma-1)q}\,\dx
& \leq\frac{C(r,p)}{R^n}\left[\frac{1}{(R_0-R)^p}+\frac{1}{T_1-T_0}\right]\\
&\times
\left[\int\!\!\!\int_{Q_0}v^{(r+p-2)(\sigma-1)q}\,\dx\,\dt\right].
\end{split}
\end{equation}
Plugging \eqref{est.1} and \eqref{est.2}  into (\ref{start1}), we
obtain
\begin{equation}\label{start2}
\begin{split}
\int\!\!\!\int_{Q}v^{\sigma(r+p-2)}\,\dx\,\dt
&\leq2^{p-1}\cs_{p}^{p}C(r,p)^{1+\frac{1}{q}}\left[\frac{1}{(R_0-R)^p}+\frac{1}{T_1-T_0}\right]^{1+\frac{1}{q}}\\
&\times\left[\int\!\!\!\int_{Q_0}v^{r}\,\dx\,\dt\right]
\left[\int\!\!\!\int_{Q_0}v^{(\sigma-1)(r+p-2)q}\,\dx\,\dt\right]^{\frac{1}{q}}.
\end{split}
\end{equation}
notice that $R$ cancels, since $R^{p-n/q}=1$.\qed

\noindent \textbf{Step 4. Choosing the exponents}

\noindent We begin by choosing $r=q(r+p-2)(\sigma-1):=r_0$, with
$\sigma\in(1,\sigma^{*})$. This implies that $\sigma=1+rp/n(r+p-2)$.
This is always larger than 1, but it has to satisfy
$\sigma<\sigma^{*}=n/(n-p)$, hence we need that $r>n(2-p)/p:=r_c$. We
remark that $r_c>1$ if and only if $p<p_c$. We define next
\begin{equation*}
r_1=(r_0+p-2)\sigma=\left(1+\frac{1}{q}\right)r_0+p-2,
\end{equation*}
and we see that $r_1>r_0$ if and only if $r_0>r_c$. In a natural
way, we iterate this construction:
\begin{equation*}
r_{k+1}=r_{k}\left(1+\frac{1}{q}\right)+p-2,
\end{equation*}
and we note that $r_{k+1}>r_k$ if and only if $r_k>r_0>r_c$.
Moreover, we can provide an explicit formula for the exponents
\begin{equation}\label{expiter}
\begin{split}
r_{k+1}&=r_{k}(1+\frac{1}{q})+p-2=\left(1+\frac{1}{q}\right)^{k+1}r_0+(p-2)
\sum_{j=0}^{k}\left(1+\frac{1}{q}\right)^{j}\\
&=\left(1+\frac{1}{q}\right)^{k+1}\left[r_0+(p-2)\sum_{l=1}^{k+1}\left(1+\frac{1}{q}\right)^{-l}\right]
=\left(1+\frac{1}{q}\right)^{k+1}\big[r_0-(2-p)q\big]+q(2-p).
\end{split}
\end{equation}
Let us calculate two useful limits of the exponents:
\begin{equation}\label{limpropexp}
\lim_{k\to\infty}\frac{\Big(1+\frac{1}{q}\Big)^{k+1}}{r_{k+1}}=\frac{1}{r_0+(p-2)q},
\
\lim_{k\to\infty}\frac{1}{r_{k+1}}\sum_{j=0}^{k}\left(1+\frac{1}{q}\right)^{j}=\frac{q}{r_0+(p-2)q}.
\end{equation}
We are now ready to rule the iteration process.

\noindent \textbf{Step 5. The iteration}

\noindent The iteration process consists in writing the inequality
(\ref{start2}) with the exponents introduced in the previous step.
The first step then reads
\begin{equation}
\begin{split}
\left[\int\!\!\!\int_{Q}v^{r_1}\,\dx\,\dt\right]^{\frac{1}{r_1}}
&\leq\left\{2^{p-1}\cs_{p}^{p}C(r_0,p)^{1+\frac{1}{q}}
\left[\frac{1}{(R_0-R)^p}+\frac{1}{T_1-T_0}\right]^{1+\frac{1}{q}}\right\}^{\frac{1}{r_1}}\\
&\times\left[\int\!\!\!\int_{Q_0}v^{r_0}\,\dx\,\dt\right]^{\left(1+\frac{1}{q}\right)\frac{1}{r_1}}
=I_{0,1}^{\frac{1}{r_1}}\left[\int\!\!\!\int_{Q_0}v^{r_0}\,\dx\,\dt\right]^{\left(1+\frac{1}{q}\right)\frac{1}{r_1}}.
\end{split}
\end{equation}
As for the general iteration step, we have to construct a sequence
of cylinders $Q_{k}$ such that $Q_{k+1}\subset Q_{k}$, with the
convention $Q_1=Q$, and apply it to inequality (\ref{start2}). We
let $Q_{k}=B_{R_k}\times(T_k,T]$, with $R_{k+1}<R_k$ and
$T_{k}<T_{k+1}<T$. The $k$-th step then reads
\begin{equation}\label{stepk}
\left[\int\!\!\!\int_{Q_{k+1}}v^{r_{k+1}}\,\dx\,\dt\right]^{\frac{1}{r_{k+1}}}\leq
I_{k,k+1}^{\frac{1}{r_{k+1}}}\left[\int\!\!\!\int_{Q_k}v^{r_k}\,\dx\,\dt\right]^{\frac{1}{r_k}
\left(1+\frac{1}{q}\right)\frac{r_k}{r_{k+1}}},
\end{equation}
where
\begin{equation*}
I_{k,k+1}:=2^{p-1}\cs_{p}^{p}C(r_k,p)^{1+\frac{1}{q}}\left[\frac{1}{(R_{k}-R_{k+1})^p}+\frac{1}{T_{k+1}-T_k}\right]^{1+\frac{1}{q}}.
\end{equation*}
Iterating now (\ref{stepk}) we obtain
\begin{equation}\label{kprod}
\left[\int\!\!\!\int_{Q_{k+1}}v^{r_{k+1}}\,\dx\,\dt\right]^{\frac{1}{r_{k+1}}}\leq
I_{k,k+1}^{\frac{1}{r_{k+1}}}I_{k-1,k}^{\left(1+\frac{1}{q}\right)\frac{1}{r_{k+1}}}
\ ... \ I_{0,1}^{\left(1+\frac{1}{q}\right)^{k}\frac{1}{r_{k+1}}}
\left[\int\!\!\!\int_{Q_{0}}v^{r_{0}}\,\dx\,\dt\right]^{\left(1+\frac{1}{q}\right)^{k+1}\frac{1}{r_{k+1}}}\,.
\end{equation}
In order to get uniform estimates for  $I_{k,k+1}$, we have to
impose some further conditions on $R_k$ and $T_k$. More precisely,
we choose a decreasing sequence $R_k\to R_{\infty}>0$ such that
$R_k-R_{k+1}=\rho/k^2$ and an increasing sequence of times $T_{k}\to
T_{\infty}<T$ such that $T_{k+1}-T_k=\tau/k^{2p}$. Moreover, we see
that
\begin{equation*}
\tau=\frac{T_{\infty}-T_0}{\sum_{k}\frac{1}{k^{2p}}}>0, \qquad
\rho=\frac{R_0-R_{\infty}}{\sum_{k}\frac{1}{k^2}}>0.
\end{equation*}
Recall also that the constants $C(r_j,p)\leq C_{0}(p)r_j$, so that
\begin{equation*}
I_{j,j+1}\leq2^{p-1}\cs_{p}^{p}C(p)\left[j^{2p}r_{j}\left(\frac{1}{\tau}+\frac{1}{\rho^p}\right)\right]^{1+\frac{1}{q}}\leq
J_{0}J_{1}^{1+\frac{1}{q}}\big(j^{2p}r_{j}\big)^{1+\frac{1}{q}},
\end{equation*}
where $J_0=2^{p-1}\cs_{p}^{p}C(p)$, $J_{1}=\tau^{-1}+\rho^{-p}$ are
constants that do not depend on $r$. Hence, we obtain:
\begin{equation*}
\begin{split}
I_{k,k+1}^{\frac{1}{r_{k+1}}}I_{k-1,k}^{\left(1+\frac{1}{q}\right)\frac{1}{r_{k+1}}}
\ ... \ I_{0,1}^{\left(1+\frac{1}{q}\right)^{k}\frac{1}{r_{k+1}}}
\leq\left[J_{0}J_{1}^{1+\frac{1}{q}}\right]^{\frac{1}{r_{k+1}}\sum_{j=0}^{k}\left(1+\frac{1}{q}\right)^{j}}
\prod_{j=0}^{k}\left(j^{2p}r_{j}\right)^{\frac{1}{r_{k+1}}\left(1+\frac{1}{q}\right)^{k+1-j}}
\end{split}
\end{equation*}
and it remains to study the convergence of the products in the
right-hand side. To this end, we take logarithms and we write:
\begin{equation*}
\begin{split}
\log\left[\prod_{j=0}^{k}(j^{2p}r_{j})^{\frac{1}{r_{k+1}}\big(1+\frac{1}{q}\big)^{k+1-j}}\right]
&=\sum_{j=0}^{k}\frac{1}{r_{k+1}}\left(1+\frac{1}{q}\right)^{k+1-j}\big(\log(r_j)+2p\log
j\,\big)\\&=\frac{1}{r_{k+1}}\left(1+\frac{1}{q}\right)^{k+1}\sum_{j=0}^{k}\left[\frac{\log(r_j)}
{\left(1+\frac{1}{q}\right)^{j}}+2p\frac{\log
j}{\left(1+\frac{1}{q}\right)^{j}}\right]\,.
\end{split}
\end{equation*}
It is immediate to check that the series obtained in the right-hand
side are convergent.

\noindent We can pass to the limit as $k\to\infty$ in (\ref{kprod})
taking into account of (\ref{limpropexp})
\begin{equation}\label{limit}
\sup_{Q_{\infty}}|v|\leq
J_{0}^{\frac{q}{r_0+(p-2)q}}J_{1}^{\frac{q+1}{r_0+(p-2)q}}C(n,p)
\left[\int\!\!\!\int_{Q_{0}}v^{r_{0}}\,\dx\,\dt\right]^{\frac{1}{r_0+(p-2)q}}.
\end{equation}
This last estimate holds true for cylinders $Q_{\infty}\subset Q_0$
and it blows-up as $Q_{\infty}\to Q_0$, since the constant
$J_{1}=\tau^{-1}+\rho^{-p}$ blows-up in such a limit. Once we let
$Q_{\infty}=B_{R_{\infty}}\times(T_{\infty},T]$, we see that
$Q_{\infty}$ has to be strictly contained in the initial cylinder
$Q_0=Q$. We finally rewrite the estimate (\ref{limit}) in terms of
$T_{\infty}$ and $R_{\infty}$, in the following way
\begin{equation*}
\sup_{Q_{\infty}}|v|\leq
C(r_0,p,n)\left[\frac{1}{(R_0-R_{\infty})^p}+\frac{1}{T_{\infty}-T_0}\right]^{\frac{q+1}{r_0+(p-2)q}}
\left[\int\!\!\!\int_{Q_{0}}v^{r_{0}}\,\dx\,\dt\right]^{\frac{1}{r_0+(p-2)q}}.\mbox{\qed}
\end{equation*}

\noindent \textbf{Step 6. End of the proof of Theorem \ref{localsm}}

\noindent The result of Theorem \ref{localsm} is given in terms of
the local strong solution $u$. We then recall that $u\leq v\leq
1+u$, hence
\begin{equation*}
\begin{split}
\sup_{Q_{\infty}}|u|&\leq\sup_{Q_{\infty}}|v|\leq
C(r_0,p,n)\left[\frac{1}{(R_0-R_{\infty})^p}+\frac{1}{T_{\infty}-T_0}\right]^{\frac{q+1}{r_0+(p-2)q}}
\left[\int\!\!\!\int_{Q_{0}}v^{r_{0}}\,\dx\,\dt\right]^{\frac{1}{r_0+(p-2)q}}\\
&\leq
C(r_0,p,n)\left[\frac{1}{(R_0-R_{\infty})^p}+\frac{1}{T_{\infty}-T_0}\right]^{\frac{q+1}{r_0+(p-2)q}}
\left[\int\!\!\!\int_{Q_{0}}u^{r_{0}}\,\dx\,\dt+|Q_0|\right]^{\frac{1}{r_0+(p-2)q}}.
\end{split}
\end{equation*}
The proof is concluded once we go back to the original notations as
in the statement of Theorem \ref{localsm}, namely we let $r=r_0$,
$R_{\infty}=R<R_0$, $T_{\infty}=T_1\in(T_0,T)$ and $q=n/p$.\qed

\subsection{Behaviour of local $L^r$-norms. $L^r$-stability}\label{subsect.Lr.stab}

In this subsection we state and prove an $L^r_{loc}$-stability
results, namely we compare local $L^r$ norms at different times.

\begin{theorem}\label{main-normr}
Let $u\in C\big((0,T):W^{1,p}_{loc}(\Omega)\big)$ be a bounded local
strong solution of the fast $p$-Laplacian equation, with $1<p<2$.
Then, for any $r>1$ and any $0<R<R_0\le \dist(x_0,\partial\Omega)$ we
have the following upper bound for the local $L^{r}$ norm:
\begin{equation}\label{Lr.stab}
\begin{split}
\left[\int_{B_{R}(x_0)}|u|^{r}(x,t)\,\dx\right]^{\frac{2-p}{r}}
\leq&\left[\int_{B_{R_0}(x_0)}|u|^{r}(x,s)\,\dx\right]^{\frac{2-p}{r}}+C_{r}\,(t-s),
\end{split}
\end{equation}
for any $0\leq s\leq t\leq T$, where
\begin{equation}
C_{r}=\frac{C_0}{(R_{0}-R)^{p}}|B_{R_0}\setminus
B_{R}|^{\frac{2-p}{r}}\,,\qquad\mbox{if $r>1$},
\end{equation}
with $C_1$ and $C_0$ depending on $p$ and on the dimension $n$.
Moreover, $C_0$ depends also on $r$ and blows up when $r\to
+\infty$.
\end{theorem}

\noindent \textbf{Remarks}: (i) Theorem
\ref{main-normr} implies that, whenever $u(\cdot,s)\in
L^{r}_{loc}(\Omega)$, for some time $s\ge 0$ and some $r\ge 1$, then
$u(\cdot,t)\in L^{r}_{loc}(\Omega)$, for all $t>s$, and there is a quantitative estimate of
the evolution of the $L^r_{loc}$-norm. This is what we call $L^{r}_{loc}$-stability.

\noindent (ii) We remark that the result of Theorem
\ref{main-normr} is false for $p\geq2$, since any $L^{r}_{loc}$ stability result
necessarily involves the control of the boundary data; on the other hand, this
local upper bound may be extended also to the limit case $p\to 1$.

\noindent (iii) Let us examine the behaviour of the constant $C_r$. We see
that it blows-up as $R\to R_{0}$. Indeed, we can write in that limit:
\begin{equation*}
C_{r}(R,R_0,p,n) \sim(R_{0}-R)^{\frac{2-p-rp}{r}},
\end{equation*}
and in our conditions $2-p-rp<0$. On the other hand, if we choose proportional radii, say $R=R_0/2$, we get
$$C_r=C(n,p,r)R_0^{-(r-r_c)p/r}.
$$
In the limit $R_0\to\infty$, we recover the standard monotonicity of the global $L^r(\RR^n)$-norms, when $r>r_c$.

\noindent (iv) Theorem \ref{main-normr} holds true also for more
general nonlinear operators, the so-called $\Phi$-Laplacians, or for
operators with variable coefficients satisfying the standard
structure conditions of \cite{DiB}, recalled in Section 8.
The proof is similar and we leave it to the interested reader.

\begin{proof} Let us calculate
\begin{equation}\begin{split}
\frac{\rd}{\dt}\int_\Omega J(u)\varphi \dx
    &= \int_\Omega |J'(u)|\Delta_p(u)\varphi \dx
     = -\int_\Omega \left|\nabla u\right|^{p-2}\nabla u \cdot\nabla\left(J'(u)\varphi\right)\dx\\
    &=-\int_\Omega \left|\nabla u\right|^{p}J''(u)\varphi \dx
      -\int_\Omega \left|\nabla u\right|^{p-2}|J'(u)|\nabla\varphi\cdot\nabla u \dx\\
    &\le-\int_\Omega \left|\nabla u\right|^{p}J''(u)\varphi \dx
      +\int_\Omega \left|\nabla u\right|^{p-1}\,|J'(u)|\,\left|\nabla\varphi\right|\dx,
 \end{split}
\end{equation}
where $J$ is a suitable convex function that will be explicitly
chosen afterwards. All the integration by parts are justified in
view of the H\"older regularity of the solution and by Corollary
\ref{CorEnergy}. We now estimate the last integral
\begin{equation}\label{I1.HP}\begin{split}
  I_1= &\int_\Omega \left|\nabla u\right|^{p-1}\,|J'(u)|\,\left|\nabla\varphi\right|\dx\le
        \left[\int_\Omega \left|\nabla u\right|^{p}\,J''(u)\varphi\right]^{\frac{p-1}{p}}
        \left[\int_\Omega \frac{|J'(u)|^p}{[J''(u)]^{p-1}}
        \frac{\left|\nabla\varphi\right|^p}{\varphi^{p-1}}\dx\right]^{\frac{1}{p}}\\
        &\le \underbrace{\left[\int_\Omega \left|\nabla u\right|^{p}\,J''(u)\varphi\right]^{\frac{p-1}{p}}}\limits_{a^{\frac{p-1}{p}}}
        \underbrace{\left[\int_\Omega \frac{|J'(u)|^{p\delta'}}{[J''(u)]^{(p-1)\delta'}}\varphi\dx\right]^{\frac{1}{p\delta'}}
        \left[\int_\Omega \frac{\left|\nabla\varphi\right|^{p\delta}}
        {\varphi^{\gamma}}\dx\right]^{\frac{1}{p\delta}}}\limits_{b},
        \end{split}
\end{equation}
where $\gamma=\delta(p-1+1/\delta')=p\delta-1$. In the first line we have used H\"older's inequality with conjugate exponents $p/(p-1)$ and $p$. In the second line, H\"older's inequality with conjugate exponents $\delta > 1$ and $\delta'=\delta/(\delta-1)$. We now use the numerical inequality
\[
a^{\frac{p-1}{p}}b\le \frac{p-1}{\varepsilon p}a+\frac{\varepsilon^{\frac{1}{p-1}}}{p}b^p
    = a + \frac{(p-1)^{\frac{1}{p-1}}}{p^{\frac{p}{p-1}}}b^p
\]
if we choose $\varepsilon=(p-1)/p$. In this way we can write
\begin{equation}
      I_1 \le          \int_\Omega \left|\nabla u\right|^{p}J''(u)\varphi \dx
      +\frac{(p-1)^{\frac{1}{p-1}}}{p^{\frac{p}{p-1}}}
        \left[\int_\Omega \frac{|J'(u)|^{p\delta'}}{[J''(u)]^{(p-1)\delta'}}\varphi\dx\right]^{\frac{1}{\delta'}}
        \left[\int_\Omega \frac{\left|\nabla\varphi\right|^{p\delta}}
        {\varphi^{p\delta-1}}\dx\right]^{\frac{1}{\delta}}
\end{equation}
All together, we have proved that
\begin{equation}\label{HP.J.general}
\frac{\rd}{\dt}\int_\Omega J(u)\varphi \dx
    \le C_1        \left[\int_\Omega \frac{|J'(u)|^{p\delta'}}{[J''(u)]^{(p-1)\delta'}}\varphi\dx\right]^{\frac{1}{\delta'}}
        \left[\int_\Omega\frac{\left|\nabla\varphi\right|^{p\delta}}
        {\varphi^{p\delta-1}}\dx\right]^{\frac{1}{\delta}}\,,
\end{equation}
with $C_1=(p-1)^{1/(p-1)}/p^{p/(p-1)}$.
\medskip

We now specialize $J$ and $\delta$ to get the result for $r>1$. We  let $\delta'=r/(r+p-2)$ and
$\delta=r/(2-p)$ in \eqref{HP.J.general} and estimate the last
integral in \eqref{HP.J.general} using inequality
\eqref{Const.Phi.Append.1} of Lemma \ref{choice.varphi} with
$\alpha=pr/(2-p)$, and $C_{2,r}$ that depends only on $p$, $r$ and
$n$, after choosing the test function $\varphi$ as there. We obtain
\begin{equation}\begin{split}
\frac{\rd}{\dt}\int_\Omega J(u)\varphi \dx
    &= C_{2,r}
        \frac{\left|B_{R_0}\setminus B_R\right|^{\frac{2-p}{r}}}{(R_0-R)^{p}}
        \left[\int_\Omega \frac{|J'(u)|^{\frac{pr}{r+p-2}}}{[J''(u)]^{\frac{(p-1)r}{r+p-2}}}\varphi\dx\right]^{\frac{r+p-2}{r}}\\
      &:= C_{3,r}(R_0,R)\left[\int_\Omega \frac{|J'(u)|^{\frac{pr}{r+p-2}}}{[J''(u)]^{\frac{(p-1)r}{r+p-2}}}\varphi\dx\right]^{\frac{r+p-2}{r}}.
\end{split}
\end{equation} We now choose the convex function
$J:[0,+\infty)\to[0,+\infty)$ to be $J(u)=|u|^r$, so that
$$
|J'(u)|^{\frac{pr}{r+p-2}}/[J''(u)]^{\frac{(p-1)r}{r+p-2}}= \frac{r^{\frac{pr}{r+p-2}}}{(r-1)^{\frac{(p-1)r}{r+p-2}}}J(u).
$$
All together, putting $X(t)= \int_\Omega J(u(\cdot,t))\varphi \dx$,
we have proved that
\begin{equation}\label{ineq.Lr}\begin{split}
\frac{\rd X(t)}{\dt}
    &\le  \frac{r^p\,C_{3,r}(R_0,R)}{(r-1)^{p-1}}\, X(t)^{\frac{r+p-2}{r}}
    := C_{4,r}(R_0,R)\,X(t)^{1-\frac{2-p}{r}}\,.
\end{split}
\end{equation}
We integrate the closed differential inequality \eqref{ineq.Lr} over
$(s,t)$ to obtain
\[
\left[\int_\Omega J\big(u(t)\big)\varphi \dx\right]^{\frac{2-p}{r}}
    \le \left[\int_\Omega J\big(u(s)\big)\varphi \dx\right]^{\frac{2-p}{r}} + \frac{2-p}{r}C_{4,r}(R_0,R)(t-s).
\]
with $C_{4,r}$ as above. We finally estimate both integral terms, using the special form of the test function $\varphi$, see e.g. Lemma \ref{choice.varphi}, and get
\begin{equation}\label{Quasi-HP.r}
\left[\int_{B_R}|u|^r(t)\dx\right]^{2-p} \le \left[\int_{B_{R_0}} |u|^r(s)\dx\right]^{\frac{2-p}{r}} +C_r(t-s)\,,
\end{equation}
with $C_r$ as in the statement.

It only remains to remove the initial assumption $u(t)\in
L^\infty_{loc}$: consider the sequence of essentially bounded
functions $u_n(\tau)\to u(\tau)$ in $L^r_{loc}$, when $n\to\infty$,
for a.e. $\tau\in(s,t)$. It is then clear that inequality
\eqref{Lr.stab} holds for any $u_n$ and we can pass to the
limit.\end{proof}

\medskip

The reader will notice that the constant $C_{r}$ above blows up as
$r\to 1$, hence the need for a different proof in that limit case.

\begin{theorem}\label{main-norm1}
Let $u\in C\big((0,T):W^{1,p}_{loc}(\Omega)\big)$ be a nonnegative
bounded local strong solution of the fast $p$-Laplacian equation,
with $1<p<2$. Let  $0<R<R_0\le \dist(x_0,\partial\Omega)$. Then we
have
\begin{equation}\label{Lr.stab}
\int_{B_{R}(x_0)}|u(x,t)|\,\dx
\leq C_1 \int_{B_{R_0}(x_0)}|u(x,s)|\,\dx +C_2(t-s)^{1/(2-p)},
\end{equation}
for any $0\leq s\leq t\leq T$. There $C_1$ is a constant near 1 that
depends on $n,p$, while $C_2$ depends also on $R$ and $R_0$.
\end{theorem}

\begin{proof} (i)
The first part of the proof is identical to the proof of Theorem \ref{main-normr} up to formula \eqref{HP.J.general}, hence we do not repeat it. We proceed then by a different choice of $J$ and $\delta$.
We choose $\lambda$ and $\varepsilon$ small in $(0,1)$ and put for all $|u|\ge 1$
\[
J'(u)=\mbox{\rm sign}(u)\left(1-\frac{\lambda}{(1+|u|)^\varepsilon}\right)
\]
while for $|u|\le 1$ we choose a smooth curve that joins smoothly with the previous values.
Then we have for $|u|\ge 1$
$$
J''(u)=\varepsilon\lambda(1+|u|)^{-1-\varepsilon}, \qquad (1-\lambda)|u|\le J(u) \le |u|.
$$
Since $1<p<2$ we may always choose $\varepsilon$ small enough so
that $(p-1)(1+\varepsilon)<1$. We may then choose
$1/\delta'=(p-1)(1+\varepsilon)$ so that $1/\delta=2-p-\mu$ with
$\mu=\varepsilon(p-1)$ also small and positive. In view of the
behaviour of $J$, $J'$ and $J''$ for large $|u|$  we obtain the
relation
$$
|J'(u)|^{p\delta'}/[J''(u)]^{(p-1)\delta'}\le K_1 J(u)+K_2\,,
$$
for some constants $K_1$ and $K_2>0$ that depend only on $p,n,\varepsilon$ and $\lambda$. Note that $K_1$ blows up if we try to pass to the limit $\varepsilon\to 0$. Then,
\eqref{HP.J.general} implies that
\begin{equation}
\frac{\rd}{\dt}\int_\Omega J(u)\varphi \dx
    \le  C_2
        \frac{\left|B_{R_0}\setminus B_R\right|^{1/\delta}}{(R_0-R)^{p}} \left[\int_\Omega (K_1 J(u)+K_2)\varphi\dx\right]^{1/\delta'},
\end{equation}
where $C_2$ depends on $p,n$, and $p\delta$.  Therefore, if $Y(t):=\int_\Omega J(u(\cdot,t)\varphi \dx$ we get
\begin{equation}
\frac{\rd Y(t)}{\dt}
    \le  C_2
        \frac{\left|B_{R_0}\setminus B_R\right|^{1/\delta}}{(R_0-R)^{p}} \left(K_1 Y(t) +
        K_2|B_{R_0}| \right)^{1/\delta'}\le C_3\left( Y(t) +
        C_4 \right)^{1/\delta'},
        \end{equation}
where now $C_3$ and $C_4$ depend also on $R_0,R$ and $\delta$. Integration of this inequality gives
for every $0<s<t<T$:
\begin{equation}
(Y(t)+C_4)^{1/\delta}\le (Y(s)+C_4)^{1/\delta} + C_5 (t-s).
        \end{equation}
Since $(1-\lambda)|u|\le J(u) \le |u|$ we easily obtain the basic inequality
\begin{equation}
\left(\int_\Omega J(u(\cdot,t))\varphi \dx +C_4\right)^{1/\delta}\le
\left(\int_\Omega J(u(\cdot,s))\varphi \dx+C_4\right)^{1/\delta} + C_5 (t-s).
        \end{equation}

\medskip

\noindent (ii) We now translate this inequality into an $L^1$ estimate. We use the
fact that
$$
J(u)\le |u|+c_1\le c_2 J(u)+c_3.
$$
Therefore, with $a_1=1/c_2=1-\lambda$ and $a_2=(c_1-c_3)/c_2$ we have
$$
\left(\int_\Omega (a_1|u(\cdot,t)|+a_2)\varphi \dx +C_4\right)^{1/\delta}\le
\left(\int_\Omega (|u(\cdot,s)|+c_1)\varphi \dx+C_4\right)^{1/\delta} + C_5 (t-s),
$$
that we can rewrite as
$$
\left(\int_\Omega (|u(\cdot,t)|+a'_2)\varphi \dx
+C'_4\right)^{1/\delta}\le (1-\lambda)^{1/\delta}\left(\int_\Omega
(|u(\cdot,s)|
+c_1)\varphi \dx+C_4\right)^{1/\delta} + C'_5 (t-s).
$$
This means that for every $\ve>0$ we have
$$
\left(\int_\Omega |u(\cdot,t)|\,\varphi \dx \right)^{1/\delta}\le
(1+c(\ve+\lambda))\left(\int_\Omega |u(\cdot,s)|\varphi \dx\right)^{1/\delta} + C_\ve+C'_5(t-s).
$$

\noindent (iii) Let us perform a scaling step. We take a solution $u$ as in the statement and
two fixed times $t_1>t_2>0$. We put $h=t_2-t_1$. We apply now the rule to
the rescaled solution $\widehat u$ defined as  $\widehat u(x,t)= h^{-1/(2-p)} u(x,t_1+t h)$ between $s=0$ and $t=1$. Then, after raising the expression to the power $\delta$ we get
$$
\int_\Omega |\widehat u(\cdot,1)|\,\varphi \dx \le
(1+c'(\ve+\lambda))\int_\Omega |\widehat u(\cdot,0)|\varphi \dx + C_6,
$$
which implies
$$
\int_\Omega |u(\cdot,t_2)|\,\varphi \dx \le
(1+c'(\ve+\lambda))\int_\Omega |u(\cdot,t_1)|\varphi \dx +
C_6(t_2-t_1)^{1/(2-p)}.
$$
We finally eliminate the dependence on $\e$ of the constants by
fixing $\e=(2-p)/2(p-1)>0$.
\end{proof}

\noindent\textbf{Remark}. In the proofs we use and improve on a
technique introduced by Boccardo et al.~in \cite{MR1240399} to
obtain local integral estimates for the $p$-Laplacian equation in
the elliptic framework, both for  $L^r$ and  $L^1$ norms, the latter
being technically more complicated.


\subsection{Proof of Theorem \ref{mainupper} for bounded strong solutions}\label{Proof.SE}

We are now ready to prove Theorem \ref{mainupper}, by joining the
space-time smoothing effect and local $L^r$-norm estimates. We will
work with bounded strong solution, but the same proof holds for
bounded weak solutions, that are indeed are H\"older continuous,
thus strong, cf. Appendix A2 and Theorem \ref{inequality}. The
boundedness assumption will be removed by comparison with suitable
extended large solutions in Section \ref{EOP}.

\begin{proof}
Consider a bounded (hence continuous) local strong solution $u$
defined in $Q_0=B_{R_0}(x_0)\times(0,T)$, noticing that it is not
restrictive to assume $x_0=0$. Consider a smaller ball $B_{R}\subset
B_{R_0}$ and take $\rho>0$, $\e>0$ such that $R=\rho(1-\e)$ and
$R_0=\rho(1+\e)$. Then we consider the following rescaled solution
\begin{equation}\label{resc}
\tilde{u}(x,t)=Ku(\rho x,\tau t), \qquad
K=\Big(\frac{\rho^{p}}{\tau}\Big)^{\frac{1}{2-p}}, \qquad
\tau\in(0,T),
\end{equation}
and we apply the result of Theorem \ref{localsm} to the solution
$\tilde{u}$ in the cylinders $\widetilde{Q}_0=B_1\times[0,1]$ and
$\widetilde{Q}=B_{1-\e}\times[\e^{p},1]$, for some $\e\in(0,1)$.
Recalling the notation $q=n/p$, we obtain
\begin{equation}
\sup_{\widetilde{Q}}|\tilde{u}|\leq\frac{C(r,p,n)}{\e^{\frac{p(q+1)}{r+(p-2)q}}}
\left[\int\!\!\!\int_{\widetilde{Q}_{0}}\tilde{u}^r\,\dx\,\dt+
\omega_{n}\right]^{\frac{1}{r+(p-2)q}}.
\end{equation}
We then use Theorem \ref{main-normr} for $r\ge 1$ on the balls
$B_1\subset B_{1+\e}$
\begin{equation}\label{norm-r}
\int_{B_1}|\tilde{u}(x,t)|^{r}\,\dx\leq2^{\frac{r}{2-p}-1}
\left[\int_{B_{1+\e}}|\tilde{u}(x,0)|^{r}\,\dx+\big(C_{r}(1,1+\e,p,n)\,t\big)^{\frac{r}{2-p}}\right],
\end{equation}
where we use the inequality $(a+b)^l\leq2^{l-1}(a^l+b^l)$ for
$l=r/(2-p)>1$. The constant is
\begin{equation*}\begin{split}
&C_{r}(1,1+\e,p,n)=\frac{C(r,p,n)}{\e^{p}}|B_{1+\e}\setminus
B_1|^{\frac{2-p}{r}}\leq C(r,p,n)\e^{\frac{2-p}{r}-p}, \ \mbox{if} \
r>1\,,\\ &C_{r}(1,1+\e,p,n)=\frac{C(p,n)}{\e^{p}}|B_{1+\e}\setminus
B_1|^{2-p}+|B_{1+\e}|^{2-p}, \ \mbox{if} \ r=1.
\end{split}\end{equation*}
We integrate in time over $(0,1)$ and we obtain:
\begin{equation*}
\int\!\!\!\int_{\widetilde{Q}_{0}}\tilde{u}^r\,\dx\,\dt\leq2^{\frac{r}{2-p}-1}\left[\int_{B_{1+\e}}|\tilde{u}(x,0)|^{r}\,\dx+
\frac{1}{\frac{r}{2-p}+1}C_{r}(1,1+\e,p,n)^{\frac{r}{2-p}}\right].
\end{equation*}
We now join the previous estimates and we obtain:
\begin{equation*}
\begin{split}
\sup_{x\in B_{1-\e},\,t\in[\e^p,1]}\tilde{u}(x,t)
&\leq\frac{C(r,p,n)}{\e^{\frac{p(q+1)}{r+(p-2)q}}}\left\{
\left[2^{\frac{r}{2-p}-1}\int_{B_{1+\e}}\tilde{u}(x,0)^{r}\,\dx\right]^{\frac{1}{r+(p-2)q}}\right.\\
&+\left.\left[\omega_{n}+2^{\frac{r}{2-p}-1}\frac{C_{r}(1,1+\e,p,n)}{1+\frac{r}{2-p}}\right]^{\frac{1}{r+(p-2)q}}\right\}\\
&=\widetilde{C}_{1,\e}\left[\int_{B_{1+\e}}\tilde{u}(x,0)^{r}\,\dx\right]^{\frac{1}{r+(p-2)q}}+\widetilde{C}_{2,\e}.
\end{split}
\end{equation*}
Then we rescale back from $\tilde{u}$ to the initial solution $u$.
From the last estimate, we get
\begin{equation*}
\sup_{x\in B_{1-\e},\,t\in[\e^p,1]}Ku(x\rho,\tau
t)\leq\widetilde{C}_{1,\e}K^{\frac{r}{r+(p-2)q}}
\left[\int_{B_{1+\e}}\tilde{u}(x,0)^{r}\,\dx\right]^{\frac{1}{r+(p-2)q}}+\widetilde{C}_{2,\e},
\end{equation*}
or, after setting $s=\tau t$ and $y=x\rho$, we can write
equivalently:
\begin{equation*}
\sup_{y\in B_{(1-\e)\rho},\,s\in(\tau\e^{p},\tau)}u(y,s)
\leq\widetilde{C}_{1,\e}K^{\frac{r}{r+(p-2)q}-1}\rho^{-\frac{n}{r+(p-2)q}}
\left[\int_{B_{(1+\e)\rho}}u_{0}(y)^{r}\,\dy\right]^{\frac{1}{r+(p-2)q}}+\frac{\widetilde{C}_{2,\e}}{K}.
\end{equation*}
Replacing $K$ with $\rho$ and $\tau$ as in (\ref{resc}), we see that
the term in $\rho$ disappears, so that
\begin{equation*}
\sup_{y\in
B_{(1-\e)\rho},\,s\in(\tau\e^{p},\tau)}u(y,s)\leq\frac{\widetilde{C}_{1,\e}}{\tau^{\frac{n}{rp+(p-2)n}}}
\left[\int_{B_{(1+\e)\rho}}|u_{0}(x)|^{r}\,\dx\right]^{\frac{p}{rp+(p-2)n}}
+\widetilde{C}_{2,\e}\left(\frac{\tau}{\rho^p}\right)^{\frac{1}{2-p}}.
\end{equation*}
We finally let $t=\tau$, $R_0=\rho(1+\e)$, $R=\rho(1-\e)$,
$C_1=\widetilde{C}_{1,\e}$, $C_2=\widetilde{C}_{2,\e}$ and replace
$q$ by its value $n/p$, in order to get the notations of Theorem
\ref{mainupper}. The proof of the main quantitative estimate
(\ref{mainsmooth}) is concluded once we analyze the constants
\begin{equation*}
C_1=\frac{K_{1}(r,p,n)}{\e^{\frac{p(q+1)}{r+(p-2)q}}}, \ \
C_2=\frac{K_{2}(r,p,n)}{\e^{\frac{p(q+1)}{r+(p-2)q}}}\left[K_3(r,p,n)\e^{\left(\frac{2-p}{r}-p\right)\frac{r}{2-p}}+K_{4}(p,n)\right]^{\frac{1}{r+(p-2)q}},
\end{equation*}
where $K_i(r,p,n)$, $i=1,2,3$, are positive constants independent on
$\e$, and
\begin{equation*}
K_4(p,n)=\omega_n \ \mbox{if} \ r>1, \quad
K_4(p,n)=\omega_n+\frac{2-p}{3-p}\,2^{n(2-p)+\frac{p-1}{2-p}}, \
\mbox{if} \ r=1.
\end{equation*}
We conclude by letting $\e=(R_0-R)/(R_0+R)$ in the formulas above.
\end{proof}


\section{Large solutions for the parabolic $p$-Laplacian equation}\label{sec.largesol}

We call \textit{continuous large solution} of the $p$-Laplacian
equation, a function $u$ solving the following boundary problem
\begin{equation*}
\left\{\begin{array}{lll}
u_t=\Delta_{p}u\,, & \mbox{in} \ Q_T,\\
u(x,t)=+\infty\,, & \mbox{on} \ \partial\Omega\times(0,T),\\
u(x,t)<+\infty\,, & \mbox{in} \ Q_T,
\end{array}\right.
\end{equation*}
in the sense that $u$ is satisfies the local weak formulation
\eqref{weak} in the cylinder $Q_T=\Omega\times(0,T)$, where $\Omega$
is a domain in $\RR^n$, is continuous in $Q_T$, and it takes the
boundary data in the continuous sense, that is $u(x,t)\to +\infty$
as $x\to\partial\Omega$. Note that there is no reference to the
initial data in this definition. If initial data are given, they
will be taken as initial traces as mentioned before. In the sequel
we will assume that $\Omega$ is bounded and has a smooth boundary
but such requirement is not essential and is done here for the sake
of simplicity.

Using the results of Theorem \ref{mainupper0} we are ready to
establish the existence of large solutions for general bounded
domains $\Omega$. We have the following:
\begin{theorem}\label{large-parabolic}
Let either $1<p\le p_c$ and $r>r_c$, or $p_c<p<2$ and $r\ge 1$.
Given $u_0\in L^r_{loc}(\Omega)$, there exists a continuous large
solution of the $p$-Laplacian equation in $\Omega$ having $u_0$ as
initial data. Such solutions are moreover H\"older continuous in the
interior, and satisfy the local smoothing effect of Theorem
\ref{mainupper0}.
\end{theorem}

\begin{proof}
We obtain first the solution by an approximation procedure. We
consider the following Dirichlet problem:
\begin{equation}
(P_n)\left\{\begin{array}{ll}u_t=\Delta_{p}u, \ \hbox{in} \ Q_T,\\
u(x,t)=n, \ \hbox{on} \ \partial\Omega\times(0,T),\\
u(x,0)=u_{0}(x), \ \hbox{in} \ \Omega,\end{array}\right.
\end{equation}
which admits a unique continuous weak solution $u_n$, by well
established theory (see e.g. \cite{DiB}). It is easy to observe that
the unique solution $u_n$ of $(P_n)$ becomes a subsolution for the
problem $(P_{n+1})$. Since any subsolution is below any solution of
the standard Dirichlet problem, we find that $u_n\leq u_{n+1}$ in
$Q_T$. By monotonicity we can therefore define the pointwise limit
$u(x,t)=\lim\limits_{n\to\infty}u_{n}(x,t)$. Moreover, $u_n$
satisfies the local bounds for the gradient, Theorem
\ref{main-grad}, since any weak solution is in particular a local
weak solution. Using the energy estimates of Theorem
\ref{main-grad}, it is then easy to check that the sequence
$\{|\nabla u_{n}|\}$ is uniformly bounded in $L^{p}_{loc}(Q_T)$,
independently on $n$, hence it converges weakly in this space to a
function $v$. By standard arguments $v=\nabla u$. Next, we write the
local weak formulation for $u_n$, on any compact
$K\times[t_1,t_2]\subset Q_T$:
\begin{equation*}
\int_{K}u_{n}(t_2)\varphi(t_2)\,\dx-\int_{K}u_{n}(t_1)\varphi(t_1)\,
\dx=-\int_{t_1}^{t_2}(u_{n}\varphi_t+|\nabla
u_{n}|^{p-2}\nabla u_n\cdot\nabla\varphi)\,\dx\,\dt,
\end{equation*}
for any test function as in Definition \ref{local.weak}. We can pass
to the limit as $n\to\infty$ by the previous observations and the
monotone convergence theorem, so that the limit $u$ satisfies the
local weak formulation \eqref{weak}. From our local smoothing effect
and Dini's Theorem, we deduce that $u_n\to u$ locally uniformly.

Moreover, $u(x,t)\to+\infty$ as $x\to\partial\Omega$; the fact that
the boundary data is taken in the continuous sense follows from
comparison with the solution of the same problem with initial data
0, which has the separate variables form and takes boundary data in
the continuous sense, cf. \cite{DL}. The last condition is that
$u(x,t)<+\infty$ in $Q_T$; but this follows directly from Theorem
\ref{mainupper0} by our assumptions. Hence, $u$ is a H\"older
continuous large solution for the $p$-Laplacian equation.
\end{proof}

\noindent {\bf Remark.} Large solutions are a typical feature of
fast diffusion equations. We recall that in the case of the fast
diffusion equation $u_t=\Delta u^m$ with $0<m<1$, the theory of
large solutions can be developed as a particular case of the theory
of solutions with general Borel measures as initial data constructed
by Chasseigne and V\'azquez in \cite{ChV} with the name of {\sl
extended continuous solutions}. The existence and uniqueness of
large solutions has been completely settled in that paper for
$m_c=(n-2)/n<m<1$. For $m<m_c$, a general uniqueness result of such
solutions is still open. A similar approach can be applied to the
fast $p$-Laplacian equation considered in this paper, but the
detailed presentation entails modifications that deserve a careful
presentation.

\medskip

We next establish a sharp space-time asymptotic estimate, which also
gives the blow-up rate of large solutions near the boundary or for
large times.

\begin{theorem}\label{asympt.large}
Let $u$ be a continuous large solution with initial datum $u_0$, in
the conditions of Theorem \ref{large-parabolic}. We have the
following bounds:
\begin{equation}\label{asympt.large.eq}
\frac{C_0\,t^{\frac{1}{2-p}}}{{\rm
dist}(x,\partial\Omega)^{\frac{p}{2-p}}}\leq
u(x,t)\leq\frac{C_1\,t^{\frac{1}{2-p}}}{{\rm
dist}(x,\partial\Omega)^{\frac{p}{2-p}}}+C_2,
\end{equation}
for some positive constants $C_0$, $C_1$ and $C_2$. In particular
$u=O({\rm dist}(x,\partial\Omega)^{\frac{p}{p-2}})$ as
$x\to\partial\Omega$.
\end{theorem}
\begin{proof}
The upper bound comes from a direct application of the Local
Smoothing Effect, Theorem \ref{mainupper0}. For the lower bound, we
compare with the continuous large solution with initial datum
$u_0\equiv0$. We look for a separate variable solution of the form
$u(x,t)=\phi(x)t^{1/(2-p)}$, hence $\phi$ is a large solution of the
elliptic problem:
\begin{equation*}
\left\{\begin{array}{ll}\Delta_{p}\phi=\lambda\phi, & \quad \hbox{in} \
\Omega\\ \phi=+\infty &\quad \hbox{on} \ \partial\Omega.\end{array}\right.
\end{equation*}
Analyzing this problem for a ball $\Omega=B_R$, we find that there
exists a unique radial large solution, namely
\begin{equation*}
u(x,t)=k(p)\,\frac{t^{\frac{1}{2-p}}}{d(x)^{\frac{p}{2-p}}}, \qquad
k(p)^{2-p}=\frac{2(p-1)p^{p-1}}{(2-p)^p}\,,
\end{equation*}
where $d(x)=R-|x|$. This precise expression does not depend on the
radius of the ball, and it is in fact true to first approximation
for the large solution of the elliptic problem in any bounded domain
with a $C^1$ boundary, cf. \cite{DL}.
\end{proof}

The existence and properties of large solutions will be used to
conclude the proof of Theorem \ref{mainupper}. Such conclusion
consists in passing from a bounded local strong solution to a
general local strong solution. This will be done essentially by
showing that any local strong solution can be bounded above by a
large solution in a small ball around the point under consideration,
with the same local initial trace $u_0$. The difficult technical
problem is that we have to take into account the boundary data in
the comparison. The way out of this difficulty is a modification of
the construction of large solutions that leads to the concept of
``extended large solutions''. Such ideas are originated in
\cite{ChV} for the fast diffusion equation.

\medskip

\noindent {\bf Extended large solutions.} We now present an
alternative approach to the construction of continuous large
solutions that will be needed in the sequel to establish some
technical results. We will only need the construction on a ball.
Take $0<R<R_1$, let $B_R\subset B_{R_1}\subset\Omega$ and
$A=B_{R_1}\setminus B_R$, and consider the following family of
Dirichlet problems
\begin{equation*}
(\mathbb{D}_n)\quad\left\{\begin{array}{ll}
\partial_t v_n=\Delta_{p}v_n, & \mbox{in} \ B_{R_1}\times(0,T),\\
v_n(x,t)=n, & \mbox{on} \ \partial B_{R_1}\times(0,T),\\
v_n(x,0)=\left\{\begin{array}{ll}
u_{0}(x),\\
n, \end{array}\right.
&\hspace{-2mm}
\begin{array}{ll}
\mbox{in} \ B_R,\\
\mbox{in} \ A.
\end{array}
\end{array}\right.
\end{equation*}
Let $v_n(x,t)$ be the unique, continuous local strong solution to
the above Dirichlet problem, corresponding to the initial datum
$u_0\in L^r_{loc}(B_{R})$. Such solutions exist for all $0<t<\infty$
and form a family of locally bounded solutions that satisfy the
local smoothing effect of Theorem \ref{mainupper}, since they are
continuous. We stress that the initial datum $v_0$ need not to have
the gradient well defined on $B_{R}$, but in the annulus $A$ we have
$\nabla v_0\equiv 0$. As in the proof of Theorem
\ref{large-parabolic}, we see that the sequence $\{v_n\}$ is
monotone increasing, $v_n(x,t)\le v_{n+1}(x,t)$ for a.e. $(x,t)$,
and converges pointwise to a function $V$ which is a solution of the
fast $p$-Laplacian equation in $B_{R}\subset B_{R_1}$, and that we
will call \textit{extended large solution}. We next investigate the
behaviour of the extended large solution $V$ in the annulus
$A=B_{R_1}\setminus B_R$.
\begin{proposition}
Under the running assumptions on $v_n$ and $V$, The extended large solution satisfies
\begin{itemize}
\item[(i)] The restriction of V to $B_R$ is a continuous large solution in the sense specified at the beginning of this section, and of Theorem \ref{large-parabolic}.
\item[(ii)] V is ``large'' when extended to the annulus $A$, in the
sense that
\[
V(x,t)=\lim_{n\to\infty}v_n(x,t)=+\infty\qquad\mbox{for all \ }(x,t)\in A\times(0,T)
\]
and the divergence is uniform.
\item[(iii)] The initial trace $V_0:=\lim\limits_{t\to 0^+}V(t,\cdot)=u_0$ in $B_R$,
while $V_0=+\infty$ in $A$.
\end{itemize}
\end{proposition}

\noindent\textbf{Remark. }  The above result somehow proves the
sharpness of Theorem \ref{large-parabolic} and motivates the
terminology ``extended large solution''. Obviously, $V_0$ is not in
$L^r_{loc}$, and the smoothing effect cannot hold in $A$.

\noindent {\it Proof.~}  We only need to prove (i) and (ii), since
(iii) easily follows by construction. Parts (i) and (ii) follow from
local $L^1$ estimates together with a comparison with suitable
radially symmetric subsolutions.

\noindent\textsc{Radially symmetric subsolutions.} We define a
special class of subsolutions $\widetilde{v_n}$: consider the
problem $(\mathbb{D}_n)$, repeat the same construction made for
$v_n$, but now we choose $u_0=0$ in $B_R$. Obviously,
$\widetilde{v_n}\le v_n$ in $B_{R_1}$, $\widetilde{v_n}\le
\widetilde{v_{n+1}}$, and they are all radially symmetric. Moreover,
by the maximum principle we know that each function
$\widetilde{v_n}$ is nondecreasing along the radii, thus
\begin{equation}\label{radial}
\int_{B_{r}(x_0)}\widetilde{v_n}(x,t)\varphi(x)\dx \leq
    \widetilde{v_n}(\overline{x},t)\int_{B_{r}(x_0)}\varphi(x)\dx,
\end{equation}
where $\overline{x}$ is the point of $\overline{B_r(x_0)}$ with
maximum modulus, since $\widetilde{v_n}$ is radially symmetric in
the bigger ball $B_{R_1}$ and $\varphi\ge 0$ is a suitable test
function that will be chosen later.

\medskip

\noindent\textsc{$L^1$ estimates.}  These estimates are possible
thanks to the local $L^p$ bounds \eqref{gradientmain} valid for the
gradient of the solution $\widetilde{v_n}$ to the Dirichlet problems
$\mathbb{D}_n$, namely, for any small ball
$B_{r+\varepsilon}(x_0)\subset A$ and
\begin{equation}\label{gradientmain.n}
\int_{B_{r}(x_0)}|\nabla
\widetilde{v_n}(x,t)|^{p}\,\dx \leq c_0 \int_{B_{r+\varepsilon}(x_0)}|\nabla
\widetilde{v_n}(x,0)|^{p}\,\dx + c_1t^{\frac{p}{2-p}}=c_1t^{\frac{p}{2-p}}\,,
\end{equation}
the last equality holds since by definition the gradient of the
initial data is  zero in $A$.

We now fix a time $t\in(0,T]$, a point $x_0\in A$ and a ball
$B_{r+\varepsilon}(x_0)\subset A$. We choose a suitable test
function $\varphi$ supported in $B_{r}(x_0)$, and we calculate
\begin{equation}\begin{split}
\left|\frac{\rd}{\dt}\int_{B_{r}(x_0)}\widetilde{v_n}(x,t)\varphi(x)\dx\right|
    &= \left|\int_{B_{r}(x_0)}\Delta_p\big(\widetilde{v_n}(x,t)\big)\varphi(x)\dx\right|\\
    &= \left|-\int_{B_{r}(x_0)}\big|\nabla \widetilde{v_n}(x,t)\big|^{p-2}\nabla
        \widetilde{v_n}(x,t)\cdot\nabla\varphi(x)\dx\right|\\
    &\le \int_{B_{r}(x_0)}\big|\nabla \widetilde{v_n}(x,t)\big|^{p-1}\,\big|\nabla\varphi(x)\big|\dx\\
    &\le \left[\int_{B_{r}(x_0)}\big|\nabla\varphi(x)\big|^p\dx\right]^{\frac{1}{p}}
         \left[\int_{B_{r}(x_0)}\big|\nabla \widetilde{v_n}(x,t)\big|^p\dx\right]^{\frac{p-1}{p}}\\
    &\le C_\varphi\,c_1 t^{\frac{p-1}{2-p}}:=\,C\,t^{\frac{p-1}{2-p}},
\end{split}
\end{equation}
where in the second line we performed an integration by parts that
can be justified in view of the H\"older regularity of the solution
and by Corollary \ref{CorEnergy}. In the fourth line we have used
H\"older inequality, and in the last step the inequality
\eqref{gradientmain.n} and the fact that the integral of the test
function is bounded. We integrate such differential inequality over
$(0,t)$ to get
\begin{equation}
\left|\int_{B_{r}(x_0)}\widetilde{v_n}(x,t)\varphi(x)\dx-\int_{B_{r}(x_0)}\widetilde{v_n}(x,0)\varphi(x)\dx\right|\le \,C \,t^{\frac{1}{2-p}}.
\end{equation}
Taking into account that $\widetilde{v_n}(x,0)=n$ and \eqref{radial}, we obtain
\begin{equation}
\widetilde{v_n}(\overline{x},t)\int_{B_{r}(x_0)}\varphi(x)\dx
    \ge \int_{B_{r}(x_0)}\widetilde{v_n}(x,t)\varphi(x)\dx
    \ge n\int_{B_{r}(x_0)}\varphi(x)\dx - C \,t^{\frac{1}{2-p}}\,,
\end{equation}
hence $\widetilde{v_n}(\overline{x},t)\to \infty$ as $n\to \infty$,
since $\overline{x}$ does not depend on $n$. Since $\widetilde{v_n}$
is radially symmetric, we have proved that $\widetilde{v_n}(x,t)\to
\infty$ as $n\to \infty$ for any $|x|=|\overline{x}|$. We can repeat
the argument for any small ball $B_r(x_0)\subset A$, and we obtain
that $\widetilde{v_n}(x,t)\to \infty$ for any $x\in A$ and $t>0$,
but not for $|x|=R$. This result extends to $v_n\ge \widetilde{v_n}$
by comparison.

\medskip

\noindent\textsc{Behaviour of $V$ in $\overline{B_R}$.}
Let $0<R<R'<R_1$ and let $L_{R'}$ be the continuous large solution
in $B_{R'}$ whose initial trace is $0$ in $B_{R'}$. Since $L_{R'}$
satisfies the local smoothing effect \eqref{mainsmooth}, we can
compare it on a smaller ball say $B_{R'-\varepsilon}$, with a
suitably chosen $\widetilde{v_n}$, namely
\[
\forall \varepsilon>0 \qquad \exists n_\varepsilon\quad\mbox{such
that}\quad \widetilde{v_{n_\varepsilon}}(x,t)\ge
L_{R'}(x,t)\quad\mbox{for any }(x,t)\in
B_{R'-\varepsilon}\times(0,T]\,.
\]
This implies that
$\widetilde{V}:=\lim\limits_{n\to\infty}\widetilde{v_n}\ge L_{R'}$
in $B_{R'-\varepsilon}\times (0,T]$ for any $\varepsilon>0$. Letting
now $\varepsilon\to 0$, we obtain that $\widetilde{V}\ge L_{R'}$ in
$B_{R'}\times (0,T]$ and this holds for any $R'\in(R,R_1)$.

By scaling we can identify different continuous large solutions in
different balls, namely let $L_R$ and $L_{R'}$ be the large
solutions corresponding to the balls $B_R\subset B_{R'}$, and
\begin{equation*}
L_{R'}(x,t)=L_{R,\lambda}(x,t):=\lambda^{\frac{p}{2-p}}L_{R}\big(\lambda
x,t\big),\qquad\mbox{with }\lambda=\frac{R}{R'}<1.
\end{equation*}
It is then clear that $L_{R'}\to L_{R}$ when $R'\to R$ at least
pointwise in $\overline{B_{R}}\times(0,T]$, and this implies also
that $\widetilde{V}\ge L_{R}$ in $\overline{B_{R}}\times(0,T]$ and
in particular
\[
\lim_{x\to \partial\!B_R}\widetilde{V}(x,t)\ge\lim_{x\to
\partial\!B_R}L_R(x,t)=+\infty \qquad\mbox{in the continuous sense.}
\]
By comparison, we see that $V\ge \widetilde{V}$, hence
$\lim\limits_{x\to
\partial\!B_R}\widetilde{V}(x,t)=+\infty$ in the continuous sense.
The initial trace of $V$ in $B_R$ is $u_0\in L^r_{loc}$, thus the
local smoothing effect applies and implies, as usual, that $V$ is
locally bounded in $B_R$, therefore it is continuous. The proof is
concluded since we have proved that $V$ is an extended large
solution, in the above sense.\qed

 The uniqueness of the extended large solution is a delicate matter in general.
It is easy to show uniqueness of such solutions in a ball, but a complete
result is not known. We will not tackle this problem here.

\medskip

\section{Local boundedness for general strong solutions. End of proof of Theorem \ref{mainupper}}\label{EOP}

Let us now conclude the proof of Theorem \ref{mainupper}. The
last step in the proof consists in comparing a general (non necessarily
bounded) local strong solution $u$ with the extended large solution
$V$ that is known to satisfy the smoothing effect
\eqref{mainsmooth}.

Let $u$ be the local strong solution, $u_0\in L^r_{loc}$ be its
initial trace, as in the assumption of Theorem \ref{mainupper}. The
comparison $u\le V$ will be proved through an approximated $L^1$
contraction principle, which uses the approximating sequence $v_n$
defined above. We borrow some ideas from Proposition 9.1 of
\cite{PME}. Let us introduce a function $P\in C^1(\real)\cap
L^{\infty}(\real)$, such that $P(s)=0$ for $s\leq 0$, $P^{'}(s)>0$
for $s>0$ which is a smooth approximation of the positive sign
function
\begin{equation*}
\hbox{sgn}^{+}(s)=1 \ \hbox{if} \ s>0, \qquad \hbox{sgn}^{+}(s)=0 \
\hbox{if} \ s\leq0.
\end{equation*}
The primitive $Q(s)=\int\limits_{0}^{s}P(t)\,\dt$, is an
approximation of the positive part: $Q(s)\sim [s]^{+}$.
\begin{proposition}
Under the running notations and assumptions, the following
``approximate $L^1$ contraction principle'' holds:
\begin{equation}\label{L1.contraction}
\int\limits_{B_R}[u(x,t)-v_n(x,t)]_{+}\,\dx\leq\int\limits_{B_{R_1}}[u(x,s)-v_n(x,s)]_{+}\,\dx+C_n,
\end{equation}
where $C_n\to 0$ as $n\to\infty$.
\end{proposition}
\begin{proof} We choose a function $\varphi\in C_{0}^{\infty}(\Omega)$
such that $\varphi\equiv1$ in $B_{R+\e}\subset B_{R_1}$,
$\hbox{supp}\,\varphi\subset B_{R_1}$ and $0\leq\varphi\leq1$. We
calculate:
\begin{equation*}
\begin{split}
\frac{\rd}{\dt}\int\limits_{B_{R_1}}&Q(u-v_n)\varphi\,\dx=\int\limits_{B_{R_1}}Q^{'}(u-v_n)(\Delta_{p}u-\Delta_{p}v_n)\varphi\,\dx\\
&=\int\limits_{B_{R_1}}P(u-v_n)\,\hbox{div}(|\nabla u|^{p-2}\nabla
u-|\nabla v_n|^{p-2}\nabla v_n)\varphi\,\dx\\
&=-\int\limits_{B_{R_1}}P^{'}(u-v_n)(\nabla u-\nabla
v_n)\cdot(|\nabla u|^{p-2}\nabla u-|\nabla v_n|^{p-2}\nabla
v_n)\varphi\,\dx\\&-\int\limits_{B_{R_1}}P(u-v_n)(|\nabla
u|^{p-2}\nabla u-|\nabla v_n|^{p-2}\nabla
v_n)\cdot\nabla\varphi\,\dx= I_1+I_2,
\end{split}
\end{equation*}
where the calculations are allowed since $u$ and $v_n$ are both local
strong solutions. Taking into account the monotonicity of the
$p$-Laplace operator and the fact that $P^{'}\geq0$, we obtain that $I_1\le 0$ and
\begin{equation}\label{sect5.eq1}
\frac{\rd}{\dt}\int\limits_{B_{R_1}}Q(u-v_n)\varphi\,\dx\leq\int\limits_{A_{\e}}P(u-v_n)(|\nabla
u|^{p-1}+|\nabla v_n|^{p-1})\,|\nabla\varphi|\,\dx=I_3+I_4,
\end{equation}
since $\hbox{supp}\,\nabla\varphi\subset A_{\e}:=B_{R_1}\setminus
B_{R+\e}$. We then have:
\begin{equation*}
I_3:=\int\limits_{A_{\e}}P(u-v_n)|\nabla
u|^{p-1}\,|\nabla\varphi|\,\dx\to0, \quad \hbox{as} \
n\to\infty,
\end{equation*}
since $u(t)\in W^{1,p}_{loc}(\Omega)$ for any $t>0$, and $P(u-v_n)\to 0$ by construction. Moreover
\begin{equation*}
I_4:=\int\limits_{A_{\e}}P(u-v_n)|\nabla
v_n|^{p-1}\,|\nabla\varphi|\,\dx\leq\ck(R)\left(\int\limits_{A_{\e}}P(u-v_n)^p\,\dx\right)^{\frac{1}{p}}\left(\int\limits_{A_{\e}}|\nabla
v_n|^p\,\dx\right)^{\frac{p-1}{p}}.
\end{equation*}
Using the gradient inequality \eqref{gradientmain.n} and the fact that $P(u-v_n)\to0$ a.\,e.,
we obtain that $I_4\to 0$ as $n\to\infty$. It follows that
\begin{equation*}
\frac{\rd}{\dt}\int\limits_{B_{R_1}}Q(u-v_n)\varphi\,\dx\leq\e_n,
\end{equation*}
where $\e_n\to0$ as $n\to\infty$. An integration of the above differential inequality on $(s,t)$,
gives
\[
\int\limits_{B_{R_1}}Q\big(u(x,t)-v_n(x,t)\big)\varphi(x)\,\dx-
\int\limits_{B_{R_1}}Q\big(u(x,t)-v_n(x,t)\big)\varphi(x)\,\dx\le\e_n(t-s)\,.
\]
Letting $P$ tend to $\hbox{sgn}^{+}$ and $Q$ to $[s]^{+}$, and
taking into account the special choice of $\varphi$, we obtain
\begin{equation}\label{L1.contraction2}
\int\limits_{B_R}[u(x,t)-v_n(x,t)]_{+}\,\dx\leq\int\limits_{B_{R_1}}[u(x,s)-v_n(x,s)]_{+}\,\dx+\e_n(t-s),
\end{equation}
for any $0\leq s\leq t<T$. Since $\e_n(t-s)\leq T\e_n$, we have proved
\eqref{L1.contraction} with $C_n=T\e_n$.
\end{proof}
We put $s=0$ in \eqref{L1.contraction}, recalling that $v_n(x,0)=u_0$, and
we pass to the limit as $n\to\infty$ in the left-hand side of \eqref{L1.contraction}, to find
\begin{equation}
\int\limits_{B_R}[u(x,t)-V(x,t)]_{+}\,\dx\leq0,
\end{equation}
hence $u(x,t)\leq V(x,t)$ for a.\,e. $(x,t)\in B_R$.

Since $V$ is
locally bounded in $B_R$, satisfies the local smoothing effect
\eqref{mainsmooth} in $B_R$, and $V_0=u_0$ in $B_R$. The smoothing effect
\eqref{mainsmooth} then holds for any local strong solution $u$ with
initial trace $u_0\in L^r_{loc}$. This concludes the proof of Theorem \ref{mainupper}.\qed

\medskip

\noindent\textbf{Remarks. } (i) A posteriori, we can ``close the circle'' by proving that indeed
\textit{any local strong solution $u$ with initial trace $u_0\in L^r_{loc}$, is H\"older continuous}
(cf. Appendix A2), since it is locally bounded via the local smoothing effect of Theorem \ref{mainupper}.

(ii) The same proof applies to nonnegative strong subsolutions as in
Definition \ref{local.weak}, hence the upper bound
\eqref{mainsmooth} holds for initial traces with any sign, not only
for nonnegative. This can be done by repeating the whole proof,
replacing the local strong solution $u$ and its initial trace $u_0$
with  the nonnegative strong subsolution $u^+$ and its trace $u_0^+$
respectively.


\section{Positivity for a minimal Dirichlet problem}\label{sec.MDP.pos}

We follow the strategy introduced in \cite{BV} for the
fast-diffusion equation to prove quantitative lower bounds for a
suitable Dirichlet problem. More specifically, we will consider what
we call ``minimal Dirichlet problem'', MDP in the sequel, whose
nonnegative solutions lie below any nonnegative continuous local
weak solution. As a by-product of the concept of local weak
solution, the estimates can be extended to continuous weak solutions
to any other problem, such as Neumann, Dirichlet (even non
homogeneous or large), Robin, Cauchy, or any other initial-boundary
problem on any (even unbounded) domain $\Omega$ containing
$B_{R_0}(x_0)$. Let us introduce the \textit{Minimal Dirichlet
Problem}
\begin{equation}\label{MDP}
\mbox{(MDP)}\quad\left\{\begin{array}{ll}u_{t}=\Delta_{p}u, \ \hbox{in} \
B_{R_{0}}\times(0,T), \\u(x,0)=u_{0}(x), \ \hbox{in} \ B_{R_0},\quad
\hbox{supp}(u_{0})\subseteq B_{R}(x_0)
\\u(x,t)=0, \ \hbox{for} \ t>0 \ \hbox{and} \ x\in\partial B_{R_0},
\end{array}\right.
\end{equation}
where $B_{R_{0}}=B_{R_{0}}(x_0)\subset\real^{n}$, and $0<2R<R_0$.
The properties of existence and uniqueness for this problem are
well-known, in particular, for any initial data $u_{0}\in
L^{2}(B_{R_0})$, the problem admits a unique weak solution $u\in
C\big( [0,\infty): L^{2}(B_{R_0})\big) \cap
L^p\big((0,\infty):W_0^{1,p}(B_{R_0})\big)$, cf. \cite{DiB}.

In the range $1<p<2$ any such solution of (\ref{MDP}) extinguishes in finite
time; we denote the finite extinction time by $T=T(u_0)$. In general
it is not possible to have an explicit expression for $T(u_0)$ in
terms of the data, but we have lower and upper estimates for $T$,
cf. \eqref{low.FET} and Subsection \ref{FET} below.

Let $u_D$ be the solution to the MDP posed on a ball $B_{R_0}\subset\Omega$, and let $T_D$
be its finite extinction time. A priori we can not compare $u_D$ with
any local weak solution $u\ge 0$, because the
parabolic boundary data can be discontinuous. We therefore restrict
$u$ to the class of bounded (hence continuous) local weak solutions
and we can compare $u$ with $u_D$, to conclude that \textsl{any
solution of the MDP lies below any nonnegative and continuous local
weak solution, with the same initial trace on the smaller ball
$B_R$}. As a by-product of this comparison, if the local weak
solution also have an extinction time $T$, then we have $T_D\le T$,
for this reason we have called $T_D$ \textit{minimal life time} for
the general local weak solution.

\subsection{The Flux Lemma}

In the previous MDP  all the initial mass is concentrated in a
smaller ball $B_R$. The next result explains in a quantitative way
how in this situation the  mass is  transferred to the annulus $B_{R_0}\setminus
B_{R}$ across the internal boundary $\partial B_{R}$. Throughout
this subsection we will set $A_1:=B_{R_0}\setminus B_{R}$ and we
will consider a cutoff function $\varphi$ supported in $B_{R_0}$ and
taking the value $1$ in $B_{R}\subset B_{R_0}$.

\begin{lemma}\label{fluxlemma}
Let $u$ be a continuous local weak solution to the MDP (\ref{MDP})
and let $\varphi$ be a suitable cutoff function as above. Then the
following equality holds:
\begin{equation}\label{flux}
\int_{B_{R_{0}}}u(x,s)\varphi(x)\,\dx=\int_{s}^{T}\int_{A_{1}}|\nabla
u(x,\tau)|^{p-2}\nabla u(x,\tau)\cdot\nabla\varphi(x)\,\dx\,d\tau,
\end{equation}
for any $s\in[0,T]$. In particular, eliminating the dependence on
$\varphi$, we obtain the following estimate:
\begin{equation}\label{fluxineq}
\int_{B_{R}}u(x,s)\,\dx\leq\frac{k}{R_0-R}\int_{s}^{T}\int_{A_1}|\nabla
u(x,\tau)|^{p-1}\,\dx\,d\tau,
\end{equation}
for a suitable constant $k=k(n)$ and for any $s\in[0,T]$.
\end{lemma}

\begin{proof}
Let $0\leq s\leq t\leq T$. We begin by calculating
\begin{equation*}
\begin{split}
\int_{s}^{t}\int_{A_1}u_{t}\varphi\,\dx\,\rd\tau&=\int_{s}^{t}\int_{A_1}\hbox{div}(|\nabla
u|^{p-2}\nabla u)\varphi\,\dx\,\rd\tau\\&=-\int_{s}^{t}\int_{A_1}|\nabla
u|^{p-2}\nabla
u\cdot\nabla\varphi\,\dx\,\rd\tau+\int_{s}^{t}\int_{\partial
B_{R}}|\nabla
u|^{p-2}\left(\partial_{\nu}u\right)\,\varphi\,d\sigma\,\rd\tau\\&+\int_{s}^{t}\int_{\partial
B_{R_0}}|\nabla u|^{p-2}\left(\partial_{\nu}u\right)\,
\varphi\,d\sigma\,\rd\tau,
\end{split}
\end{equation*}
where $\nu$ is the outward normal vector to the boundary of the
annulus $A_1$. Since $\varphi=0$ on $\partial B_{R_0}$, the last
integral above vanishes. By integrating the left-hand side and
taking into account that $\varphi=1$ on $\partial B_{R}$, we obtain:
\begin{equation*}
\int_{A_{1}}u(t)\varphi\,\dx-\int_{A_1}u(s)\varphi\,\dx=-\int_{s}^{t}\int_{A_1}|\nabla
u|^{p-2}\nabla
u\cdot\nabla\varphi\,\dx\,\rd\tau+\int_{s}^{t}\int_{\partial
B_{R}}|\nabla u|^{p-2}\partial_{\nu}u\,d\sigma\,\rd\tau.
\end{equation*}
We put in this equality $t=T$, the finite extinction time of the
solution of (\ref{MDP}), hence we have:
\begin{equation}\label{flux1}
\int_{A_1}u(s)\varphi\,\dx=\int_{s}^{T}\int_{A_1}|\nabla
u|^{p-2}\nabla
u\cdot\nabla\varphi\,\dx\,\rd\tau-\int_{s}^{T}\int_{\partial
B_{R}}|\nabla u|^{p-2}\partial_{\nu}u\,d\sigma\,\rd\tau.
\end{equation}
On the other hand, we calculate the same quantity inside the small
ball $B_{R}$. Since $\varphi\equiv1$ in $B_{R}$, we can omit the
test function here. We obtain:
\begin{equation*}
\int_{B_{R}}\big[u(t)-u(s)\big]\,\dx=\int_{s}^{t}\int_{B_R}\hbox{div}(|\nabla
u|^{p-2}\nabla u)\,\dx\,\rd\tau=\int_{s}^{t}\int_{B_R}|\nabla
u|^{p-2}\partial_{\nu^{*}}u\,\dx\,\rd\tau,
\end{equation*}
where we denote by $\nu^{*}$ the outward normal vector to the
boundary of the ball $B_R$. Then $\nu^{*}=-\nu$, hence
$\partial_{\nu}u=-\partial_{\nu^{*}}u$. Letting again $t=T$, we get
\begin{equation}\label{flux2}
\int_{B_{R}}u(x,s)\,dx=\int_{s}^{T}\int_{\partial B_{R}}|\nabla
u|^{p-2}\partial_{\nu}u\,\dx\,\rd\tau\,.
\end{equation}
Joining relations (\ref{flux1}) and (\ref{flux2}), we see that
the terms on the boundary compensate, the flux going out of the ball
$B_R$ across its boundary equals the flux entering $A_1$. By
canceling these flux terms, we obtain exactly the identity
(\ref{flux}). In order to get the estimate (\ref{fluxineq}), it
suffices now to remark that, since $\hbox{supp}\varphi\subset
B_{R_0}$ and $\varphi\equiv1$ in $B_R$, then there exists a choice
of $\varphi$ and an universal constant $k=k(n)$, depending only on
the dimension, such that
\begin{equation*}
|\nabla\varphi(x)|\leq\frac{k(n)}{R_0-R},
\end{equation*}
for any $x\in A_1$. This concludes the proof.
\end{proof}

\noindent \textbf{Remark}: Note that the undesired boundary
term is eliminated only by the fact that $\varphi=0$ on $\partial
B_{R_0}$, independently of $u$. Hence, the same estimates
(\ref{flux}) and (\ref{fluxineq}) are true in any balls
$B_{R}\subset B_{r_1}\subset B_{r_2}\subset B_{R_0}$, the only
difference in the proof being the choice of $\varphi$.

\smallskip

\noindent \textbf{A local Aleksandrov reflection principle}. Here we
state the Aleksandrov reflection principle in the version adapted
for the minimal Dirichlet problem (\ref{MDP}). That is:
\begin{proposition}\label{Aleks}
Let $u$ be a continuous local weak solution to the MDP {\rm
\eqref{MDP}}. Then, for any $t>0$, we have $u(x_0,t)\geq u(x,t)$,
for any $t>0$ and $x\in A_2:=B_{R_0}(x_0)\setminus B_{2R}(x_0)$. In
particular, this implies the following mean-value inequality:
\begin{equation}\label{aleks2}
u(x_0,t)\geq\frac{1}{|A_2|}\int_{A_2}u(x,t)\,dx.
\end{equation}
\end{proposition}
In other words, this inequality says that the mean value
of the solution of (\ref{MDP}) in an annulus is less than the
value at the center of the ball where the whole mass was
concentrated at the initial time. The proof is a straightforward
adaptation of the proof of the corresponding local Aleksandrov
principle for the fast diffusion equation, given by two of the
authors in \cite{BV2}. Indeed, the unique property of the equation
involved in the proof is the comparison principle, which both the
fast diffusion equation and the $p$-Laplacian equation enjoy.

\subsection{A lower bound for the finite extinction
time}\label{lower.T}

A first application of the Flux Lemma is a lower bound for the
finite extinction time.
\begin{lemma}
Under the assumptions of Lemma {\rm \ref{fluxlemma}} and in the
running notations, assuming moreover that $0<R<2R<R_0$, we have the
following lower bound for the FET:
\begin{equation}\label{lowerbd}
T\geq
\mathcal{K}R(R_0-2R)^{p-1}\left[\frac{1}{|B_{R_0}|}\int_{B_{R}}u_{0}(x)\,\dx\right]^{2-p},
\end{equation}
where $\mathcal{K}$ is a constant depending only on $n$ and $p$. In
particular, we obtain the lower bound for $T$ in Theorem
\ref{posit.goodFDE}.
\end{lemma}
\begin{proof}
In order to derive this lower bound, we apply (\ref{fluxineq}) to
the annulus $A_{0}:=B_{2R}\setminus B_{R}$:
\begin{equation}\label{fluxineq2}
\int_{B_{2R}}u(x,s)\,\dx\leq\frac{k}{R}\int_{s}^{T}\int_{A_0}|\nabla
u(x,\tau)|^{p-1}\,\dx\,d\tau,
\end{equation}
We are going to use the following estimate for the gradient
due to DiBenedetto and Herrero, cf. formula (0.8) in \cite{DBH90},
that reads
\begin{equation}\label{DBHestimate}
\begin{split}
\int_{s}^{T}\int_{B_{2R}}|\nabla
u|^{p-1}\,\dx\,d\tau&\leq\gamma(n,p)\left[1+\frac{T-s}{\e^{2-p}(R_0-2R)^{p}}\right]^{\frac{p-1}{p}}\\
&\times\int_{s}^{T}\int_{B_{R_0}}(T-\tau)^{\frac{1-p}{p}}(u+\e)^{\frac{2(p-1)}{p}}\,\dx\,d\tau\\
&\leq\gamma(n,p)\left[1+\frac{T-s}{\e^{2-p}(R_0-2R)^{p}}\right]^{\frac{p-1}{p}}(T-s)^{\frac{1-p}{p}}\\
&\times\int_{s}^{T}\int_{B_{R_0}}(u+\e)^{\frac{2(p-1)}{p}}\,\dx\,d\tau,
\end{split}
\end{equation}
when applied to any ball $B_{2R}\subset B_{R_0}$, for any $0<s<T$ and for any $\e>0$. The constant $\gamma(n,p)$ depends only on $n$ and $p$. We join
(\ref{fluxineq2}) and (\ref{DBHestimate}) and we let
\begin{equation*}
D(s)=\left(1+\frac{T-s}{\e^{2-p}(R_0-2R)^{p}}\right)^{\frac{p-1}{p}},
\end{equation*}
to obtain
\begin{equation*}
\int_{B_{2R}}u(x,s)\,\dx\leq\frac{k(n,p)}{R}D(s)(T-s)^{\frac{1-p}{p}}\int_{s}^{T}\int_{B_{R_0}}(u+\e)^{\frac{2(p-1)}{p}}\,\dx\,d\tau.
\end{equation*}
Then there exists $\bar{s}\in(s,T)$ such that we have
\begin{equation*}
\begin{split}
\int_{B_{2R}}u(x,s)\,\dx
&\leq\frac{k(n,p)}{R}D(s)(T-s)^{\frac{1-p}{p}}(T-s)
    \int_{B_{R_0}}\big(u(x,\bar{s})+\e\big)^{\frac{2(p-1)}{p}}\,\dx\\
&\leq\frac{k(n,p)}{R}D(s)(T-s)^{\frac{1}{p}}|B_{R_0}|
\left[\frac{1}{|B_{R_0}|}\int_{B_{R_0}}\big(u(x,\bar{s})+\e\big)\,\dx\right]^{\frac{2(p-1)}{p}}\\
&=\frac{k(n,p)}{R}D(s)(T-s)^{\frac{1}{p}}|B_{R_0}|^{\frac{2-p}{p}}
    \left[\int_{B_{R_0}}\big(u(x,\bar{s})+\e\big)\,\dx\right]^{\frac{2(p-1)}{p}}.
\end{split}
\end{equation*}
where in the first step we have used the mean-value theorem for the time
integral in the right-hand side, and in the second step the H\"older
inequality. Using now the contractivity of the $L^{1}$ norm, we
obtain
\begin{equation}\label{est.3}
\int_{B_{2R}}u(x,s)\,\dx\leq\frac{k(n,p)}{R}D(s)(T-s)^{\frac{1}{p}}|B_{R_0}|^{\frac{2-p}{p}}
\left[\int_{B_{R_0}}\big(u(x,s)+\e\big)\,\dx\right]^{\frac{2(p-1)}{p}}.
\end{equation}
We put now $s=0$. On the other hand, we take $\e>0$ such that the
following condition holds true:
\begin{equation*}
\e|B_{R_0}|=\int_{B_R}u_{0}(x)\,\dx.
\end{equation*}
This condition implies that
\begin{equation*}\begin{split}
D(0)&\leq\frac{C(n,p)}{\e^{\frac{(2-p)(p-1)}{p}}(R_0-2R)^{p-1}}T^{\frac{p-1}{p}},\\[3mm]
\int_{B_{R_0}}(u_{0}+\e)\,\dx&=2\int_{B_{R_0}}u_{0}(x)\,\dx=2\int_{B_{R}}u_{0}(x)\,\dx,\\
\end{split}
\end{equation*}
the last equality being justified by the fact that
$\hbox{supp}(u_0)\subset B_{R}$. Coming back to \eqref{est.3},
letting there $s=0$, replacing the precise value of $\e$ and taking
into account the previous remarks, we obtain:
\begin{equation*}
\begin{split}
\int_{B_R}u_{0}(x)\,\dx
&\leq\frac{K(n,p)}{\e^{\frac{(2-p)(p-1)}{p}}R(R_0-2R)^{p-1}}T|B_{R_{0}}|^{\frac{2-p}{p}}
\left(\int_{B_R}u_{0}(x)\,\dx\right)^{\frac{2(p-1)}{p}}\\[2mm]
&\leq\frac{K(n,p)}{R(R_0-2R)^{p-1}}T|B_{R_0}|^{2-p}\left(\int_{B_R}u_{0}(x)\,\dx\right)^{p-1},
\end{split}
\end{equation*}
where $K(n,p)=2^{2(p-1)/p}C(n,p)k\gamma(n,p)$, $k$ being the constant
in (\ref{fluxineq}). It follows that:
\begin{equation*}
\left(\int_{B_R}u_{0}(x)\,\dx\right)^{2-p}\leq\frac{K(n,p)}{R(R_0-2R)^{p-1}}T|B_{R_{0}}|^{2-p},
\end{equation*}
hence the lower bound follows in the stated form, once we let
$\mathcal{K}=K(n,p)$.
\end{proof}

\subsection{Positivity for the minimal Dirichlet problem}\label{subs.7.3}

The result of the Flux Lemma \ref{fluxlemma} can be interpreted as
the transformation of the positivity information coming from the initial mass
into positivity information in terms of energy. Our next goal is to transfer the
positivity information for the energy obtained so far, to positivity for the
solution itself in an annulus. To this end we will use again the above mentioned
gradient estimate of \cite{DBH90}, formula (0.8). We split the proof of the positivity
estimate into several steps.

\noindent \textbf{Step 1. Reversed space-time Sobolev inequalities along the
flow}. Let  $u$ be the solution of the MDP (\ref{MDP}), in the
assumption that $R_0>3R$. We begin by writing the estimate
(\ref{fluxineq}) in the ball of radius $7R/3$:
\begin{equation}\label{est.4}
\int_{B_{7R/3}}u(x,s)\,\ds\leq\frac{k}{R}\int_{s}^{T}\int_{B_{8R/3}\setminus
B_{7R/3}}|\nabla u(x,\tau)|^{p-1}\,\dx\,d\tau\,.
\end{equation}
We now want to estimate the right-hand side in terms of a suitable
mean value of $u$. The estimate we would like to have is quite uncommon,
indeed it can be interpreted as a reversed Sobolev inequality on an annulus
$A_1$, along the $p$-Laplacian flow. In general this kind of reversed
inequalities tend to be false.

\noindent To this end, we cover the annulus $B_{8R/3}\setminus
B_{7R/3}$ by smaller balls, of ``good'' radius, then we consider a
covering with larger balls and we apply the estimate
\eqref{DBHestimate} for $|\nabla u|^{p-1}$ . More precisely, we
consider a family of balls $\{B_{i}\}_{i=1,N}$ with radius $R_i$,
satisfying the following two conditions: that $B_{8R/3}\setminus
B_{7R/3}\subset\bigcup_{i=1}^{N}B_i$ and that $R/6<R_i<R/3$. For any
ball $B_i$, we consider a larger, concentric ball $B_i^{'}$ with
radius $R_i^{'}$, such that $R_i<R_i^{'}<R/3$. From this
construction, we deduce that
\begin{equation*}
B_{8R/3}\setminus
B_{7R/3}\subset\bigcup_{i=1}^{N}B_i\subset\bigcup_{i=1}^{N}B_{i}^{'}\subset
B_{3R}\setminus B_{2R}\subset B_{R_0}\setminus B_{2R},
\end{equation*}
which is useful, since we remain in a region where the Aleksandrov
principle applies. We apply the estimate from \cite{DBH90} for any
of the pairs $(B_i,B_i^{'})$ and we sum up to finally obtain the
desired form for the reversed space-time Sobolev inequality:
\begin{equation}
\int_{s}^{T}\int_{B_{8R/3}\setminus B_{7R/3}}|\nabla
u(x,\tau)|^{p-1}\,\dx\,d\tau \le
\frac{N\gamma(n,p)}{(T-s)^{\frac{p-1}{p}}}D(s,\e)\int_{s}^{T}\int_{B_{R_0}\setminus
B_{2R}}(u+\e)^{\frac{2(p-1)}{p}}\,\dx\,d\tau,
\end{equation}
Joining this with \eqref{est.4} we get
\begin{equation}\label{posit1}
\int_{B_{7R/3}}u(x,s)\,\ds\leq\frac{Nk(n)\gamma(n,p)}{R}D(s,\e)(T-s)^{\frac{1-p}{p}}\int_{s}^{T}\int_{B_{R_0}\setminus
B_{2R}}(u+\e)^{\frac{2(p-1)}{p}}\,\dx\,d\tau,
\end{equation}
which holds for any $s\in[0,T]$ and $\e>0$, where we have used the
following notations
\begin{equation}\label{est.K}
D(s,\e):=\left(1+\frac{T-s}{\e^{2-p}K^{p}}\right)^{\frac{p-1}{p}},
\qquad\quad K:=\min_{i=1,N}(R_{i}^{'}-R_i).
\end{equation}
\noindent \textbf{Remark.} In the estimates above, the condition
$B_{3R}\subset B_{R_0}$ can be replaced by $B_{2R+\e}\subset
B_{R_0}$, for any $\e>0$ fixed, with the same proof. That is why,
the condition $R_0>2R$ is sufficient for the result to hold.

\noindent \textbf{Step 2. Estimating time integrals}. We are going
to estimate the time integral in the right-hand side of
(\ref{posit1}) by splitting it in two parts. For any $0\leq s\leq
t\leq T$ we have
\begin{equation*}
\begin{split}
\int_{s}^{t}\int_{B_{R_0}\setminus
B_{2R}}(u+\e)^{\frac{2(p-1)}{p}}\,\dx &\leq|B_{R_0}\setminus
B_{2R}|^{\frac{2-p}{p}}\int_{s}^{t}
    \left[\int_{B_{R_0}\setminus B_{2R}}(u+\e)\,\dx\right]^{\frac{2(p-1)}{p}}\,d\tau\\
&\leq|B_{R_0}\setminus B_{2R}|^{\frac{2-p}{p}}\int_{s}^{t}
    \left[\int_{B_{R_0}}(u+\e)\,\dx\right]^{\frac{2(p-1)}{p}}\,d\tau\\
&\leq|B_{R_0}\setminus B_{2R}|^{\frac{2-p}{p}}\int_{s}^{t}
    \left[\int_{B_{R_0}}u_0 \,\dx  +\e|B_{R_0}| \right]^{\frac{2(p-1)}{p}}\,d\tau\\
&=(t-s)|B_{R_0}\setminus B_{2R}|^{\frac{2-p}{p}}
    \left[\int_{B_{R}}u_0 \,\dx  +\e|B_{R_0}| \right]^{\frac{2(p-1)}{p}}\,,
\end{split}
\end{equation*}
where we have used  H\"older inequality in the first step, and then
the $L^1(B_{R_0})$-contractivity for the MDP in the third step,
while in the last step we take into account that
$\hbox{supp}\,u_0\subset B_{R}$. We rescale $\e$ in such a way that
$\e=\a\int_{B_{R}}u_0\,\dx/|B_{R_0}|$, leaving $\a>0$ as a free
parameter that will be chosen later on. The final result of this
step reads
\begin{equation}\label{est.9}
\int_{s}^{t}\int_{B_{R_0}\setminus
B_{2R}}(u+\e)^{\frac{2(p-1)}{p}}\,\dx\,d\tau\leq(1+\a)^{\frac{2(p-1)}{p}}(t-s)|B_{R_0}\setminus
B_R|^{\frac{2-p}{p}}\left[\int_{B_{R}}u_0\,\dx\right]^{\frac{2(p-1)}{p}}.
\end{equation}

\noindent \textbf{Step 3. The critical time}. Let us come back to
(\ref{posit1}) and put $s=0$, so that
\begin{equation*}
\begin{split}
\int_{B_R}u_0(x)\,\dx&\leq\frac{Nk(n)\gamma(n,p)}{R}D(0,\e)T^{\frac{1-p}{p}}\int_{0}^{T}\int_{B_{R_0}\setminus
B_{2R}}(u+\e)^{\frac{2(p-1)}{p}}\,\dx\,d\tau\\
&=\frac{Nk(n)\gamma(n,p)}{R}D(0,\e)T^{\frac{1-p}{p}}\left[\int_{0}^{t^{*}}\int_{B_{R_0}\setminus
B_{2R}}(u+\e)^{\frac{2(p-1)}{p}}\,\dx\,d\tau\right.\\&+\left.\int_{t^{*}}^{T}\int_{B_{R_0}\setminus
B_{2R}}(u+\e)^{\frac{2(p-1)}{p}}\,\dx\,d\tau\right]\\
&\leq\frac{Nk(n)\gamma(n,p)}{R}D(0,\e)T^{\frac{1-p}{p}}\Bigg[(1+\a)^{\frac{2(p-1)}{p}}t^{*}|B_{R_0}\setminus
B_{2R}|^{\frac{2-p}{p}}\left(\int_{B_R}u_0\,\dx\right)^{\frac{2(p-1)}{p}}\\&\left.+\int_{t^{*}}^{T}\int_{B_{R_{0}}\setminus
B_{R}}(u+\e)^{\frac{2(p-1)}{p}}\,\dx\,d\tau\right],
\end{split}
\end{equation*}
where in the last step we have used \eqref{est.9} to estimate the
first integral. Here $t^{*}$ is a particular time that will be
chosen later. We estimate now $D(0,\e)$, with our choice of $\e$,
starting from the numeric inequality $(1+y)^{(p-1)/p}\leq
(2y)^{(p-1)/p}:= \kappa\, y^{(p-1)/p}$, which holds for any $y>1$,
\begin{equation*}
D(0,\e)=\left(1+\frac{T}{\e^{2-p}K^{p}}\right)^{\frac{p-1}{p}}\leq
\frac{\kappa
T^{\frac{p-1}{p}}}{\a^{\frac{(2-p)(p-1)}{p}}K^{p-1}}\left[\frac{1}{|B_{R_0}|}\int_{B_R}u_0(x)\,\dx\right]^{-\frac{(2-p)(p-1)}{p}},
\end{equation*}
where we have chosen $y= T/\big(\e^{2-p}K^{p}\big)>1$. The
condition, in terms of $K$ (defined in \eqref{est.K}), becomes
\begin{equation}\label{Cond.1}
K^p:=\left[\min_{i=1,N}(R_{i}^{'}-R_i)\right]^p<
T\,\varepsilon^{p-2}= T\,
\left[\a\int_{B_{R}}u_0\,\frac{\dx}{|B_{R_0}|}\right]^{p-2}.
\end{equation}
We will check the compatibility of this condition after our choice
of $\varepsilon$. Joining the above two estimates, we get
\begin{equation}\label{est.10}
\begin{split}
\left(\int_{B_R}u_0\,\dx\right)^{1+\frac{(2-p)(p-1)}{p}}
&\leq\frac{k_0|B_{R_0}|^{\frac{(2-p)(p-1)}{p}}}{RK^{p-1}\a^{\frac{(2-p)(p-1)}{p}}}\Bigg[
(1+\a)^{\frac{2(p-1)}{p}}t^{*}|B_{R_0}|^{\frac{2-p}{p}}\left(\int_{B_R}u_0\,\dx\right)^{\frac{2(p-1)}{p}}\\
&+\left.\int_{t^{*}}^{T}\int_{B_{R_0}\setminus
B_{2R}}(u+\e)^{\frac{2(p-1)}{p}}\,\dx\,d\tau\right],
\end{split}
\end{equation}
where we have used that $|B_{R_0}\setminus B_{2R}|<|B_{R_0}|$, and
we have defined $k_0:=Nk(n)\gamma(n,p)\kappa$. We choose now the
critical time $t^{*}$ as
\begin{equation}\label{crittime}
t^{*}=\frac{R}{2k_0}\left(\frac{K}{\alpha}\right)^{p-1}
\left(\frac{\alpha}{1+\alpha}\right)^{\frac{2(p-1)}{p}}
\left(\frac{1}{|B_{R_0}|}\int_{B_R}u_0\,\dx\right)^{2-p}.
\end{equation}
It remains to check that $t^*\le T$, and this will be done after we
fix the values of $\alpha$ and $K$.

\noindent \textbf{Step 4. The mean-value theorem}. First we
substitute the value \eqref{crittime} of $t^*$ in \eqref{est.10}
\begin{equation*}
\left(\int_{B_{R}}u_{0}\,\dx\right)^{1+\frac{(2-p)(p-1)}{p}}
\leq\frac{2k_0\,|B_{R_0}|^{\frac{(2-p)(p-1)}{p}}}{RK^{p-1}\a^{\frac{(2-p)(p-1)}{p}}}
\int_{t^{*}}^{T}\int_{B_{R_0}\setminus
B_{2R}}(u+\e)^{\frac{2(p-1)}{p}}\,\dx\,d\tau,
\end{equation*}
then we apply the mean-value theorem to the time integral in the
right-hand side and we obtain that there exists $t_1\in[t^{*},T]$
such that
\begin{equation}\label{est.11}
\frac{RK^{p-1}}{2k_0(T-t^{*})}\left(\int_{B_{R}}u_{0}\,\dx\right)^{1+\frac{(2-p)(p-1)}{p}}
\leq\left[\frac{|B_{R_0}|}{\a}\right]^{\frac{(2-p)(p-1)}{p}}
\int_{B_{R_0}\setminus B_{2R}}\big(u(x,t_1)+\e\big)^{\frac{2(p-1)}{p}}\,\dx.
\end{equation}

\noindent \textbf{Step 5. Application of the Aleksandrov reflection
principle}. We are now in position to apply Proposition \ref{Aleks},
in the form \eqref{aleks2}, to the right-hand side of the above
estimate
\begin{equation}\label{est.12}
\int_{B_{R_0}\setminus
B_{2R}}\big(u(x,t_1)+\e\big)^{\frac{2(p-1)}{p}}\,\dx\leq|B_{R_0}|\big(u(x_0,t_1)+\e\big)^{\frac{2(p-1)}{p}},
\end{equation}
note that the presence of $\varepsilon$ does not affect the
estimate. Joining \eqref{est.11} and \eqref{est.12}, and recalling
that we have rescaled $\e=\a\int_{B_{R}}u_0\,\dx/|B_{R_0}|$ we get
\begin{equation*}
\left[u(x_0,t_1)+\frac{\a}{|B_{R_0}|}\int_{B_R}u_0\,\dx\right]^{\frac{2(p-1)}{p}}
\geq\frac{\a^{\frac{(2-p)(p-1)}{p}}RK^{p-1}}{2k_0 T}
\left(\frac{1}{|B_{R_0}|}\int_{B_R}u_0\,\dx\right)^{1+\frac{(2-p)(p-1)}{p}},
\end{equation*}
or, equivalently,
\begin{equation*}
u(x_0,t_1)\geq\a^{\frac{2-p}{2}}\left(\frac{RK^{p-1}}{2k_0T}\right)^{\frac{p}{2(p-1)}}
\left(\frac{1}{|B_{R_0}|}\int_{B_R}u_0\,\dx\right)^{1+\frac{p(2-p)}{2(p-1)}}
-\frac{\a}{|B_{R_0}|}\int_{B_R}u_0\,\dx= \mathcal{H}(\alpha),
\end{equation*}
which holds for any $\a>0$. Immediately we see that
$\mathcal{H}(0)=0$ and in the limit $\alpha\to +\infty$ we get
$\mathcal{H}(\alpha)\to -\infty$, since $1<p<2$. An optimization of
$\mathcal{H}$ in $\alpha$ shows that it achieves its maximum value
at the point
\begin{equation}\label{alpha.max}
\overline{\alpha}= \left(\frac{2-p}{2}\right)^{\frac{2}{p}}
\frac{K}{\big[2k_0\big]^{\frac{1}{p-1}}}
\left(\frac{R}{T}\right)^{\frac{1}{p-1}}
\left(\frac{1}{|B_{R_0}|}\int_{B_R}u_0(x)\,\dx\right)^{\frac{2-p}{p-1}}.
\end{equation}
The value of the function $\mathcal{H}(\overline{\alpha})$ is
strictly positive and takes the form
\begin{equation}\label{posit2}
u(x_0,t_1)\geq \mathcal{H}(\overline{\alpha})=
\frac{p}{2-p}\left[\frac{2-p}{2}\right]^{\frac{2}{p}}
\frac{K}{\big[2k_0\big]^{\frac{1}{p-1}}}
\left[\frac{R}{T}\right]^{\frac{1}{p-1}}
\left[\frac{1}{|B_{R_0}|}\int_{B_{R}}u_0\,\dx\right]^{\frac{1}{p-1}},
\end{equation}
which finally gives our first positivity estimate at the point
$t_1$, once we check that all the choices of the parameters are
compatible. Indeed, we first have to check the compatibility between
\eqref{Cond.1} and \eqref{alpha.max}, that is
\begin{equation}\label{KKK}
K^2:=\left[\min_{i=1,N}\big\{R_{i}^{'}-R_i\big\}\right]^2:=\rho^2 R^2 <
\frac{2^{\frac{2}{p}}\big(2k_0\big)^{\frac{1}{p-1}}}{\big(2-p\big)^{\frac{2}{p}}}
\left[\frac{T}{R^{2-p}}\right]^{\frac{1}{p-1}}
\left[\frac{1}{|B_{R_0}|}\int_{B_{R}}u_0\,\dx\right]^{\frac{p-2}{p-1}},
\end{equation}
which is nothing but a restriction on the choice of the radii $R_i$
and $R_{i}^{'}$ in terms of the data of the MDP, and allow to fix a
value of $\rho$ in terms of the data. It only remains to check that
substituting the value $\overline{\alpha}$ in the expression
\eqref{crittime} of $t^*$, we have $t^*\le T$, where $T$ is the
finite extinction time. From \eqref{crittime} and \eqref{alpha.max}
we obtain
\begin{equation}\label{crittime2b}
\begin{split}
t^{*}&=\frac{R}{2k_0}\left(\frac{K}{\overline{\alpha}}\right)^{p-1}
\left(\frac{\overline{\alpha}}{1+\overline{\alpha}}\right)^{\frac{2(p-1)}{p}}
\left(\frac{1}{|B_{R_0}|}\int_{B_R}u_0\,\dx\right)^{2-p}\\
&=
\left[\frac{2\,\overline{\alpha}}{(2-p)(1+\overline{\alpha})}\right]^{\frac{2(p-1)}{p}}
:= k T\,,
\end{split}\end{equation}
where $k\le 1$ if and only if
\begin{equation}\label{alpha.bar}
\overline{\alpha}= \left(\frac{2-p}{2}\right)^{\frac{2}{p}}
\frac{K}{\big(2k_0\big)^{\frac{1}{p-1}}}
\left(\frac{R}{T}\right)^{\frac{1}{p-1}}
\left(\frac{1}{|B_{R_0}|}\int_{B_R}u_0(x)\,\dx\right)^{\frac{2-p}{p-1}}
\le\frac{\left(\frac{2-p}{2}\right)^{\frac{2}{p}}}{1-\left(\frac{2-p}{2}\right)^{\frac{2}{p}}}\,,
\end{equation}
and this condition is satisfied, since $K$ is bounded as in \eqref{KKK},
but the constant $k_0$ can be chosen arbitrarily large, since it
comes from the upper bound \eqref{est.10}.

\noindent \textsl{Removing the dependence on $T$ in the expression \eqref{alpha.bar} of $\overline{\alpha}$.} Let us note that formula \eqref{crittime2b} expresses $t^{*}$ as an
increasing function of $\overline{\alpha}$ whenever
\begin{equation*} \overline{\alpha}\le\frac{\left(\frac{2-p}{2}\right)^{\frac{2}{p}}}{1-\left(\frac{2-p}{2}\right)^{\frac{2}{p}}}.
\end{equation*}
Letting equality in the above expression we can remove $T$  from the
expression of $\overline{\alpha}$ and a posteriori we can conclude
that $t^{*}$ given by \eqref{crittime2b}, does not depend on $T$. A
convenient expression for $t^*$ is given by
\begin{equation}\label{crittime2}
t^{*}=k^{*} R^{p-n(2-p)}\|u_0\|_{L^{1}(B_R(x_0))}^{2-p},
\end{equation}
where the constant $k^{*}$ depends only on $n,p$.

\noindent \textbf{Step 6. Positivity backward in time}. In this step
we recover positivity for any time $0<t<t_1$, using an extension of
the celebrated  Benilan-Crandall estimates, cf. \cite{BC}. Indeed,
the Benilan-Crandall estimate for the MDP  reads
\begin{equation}
u_{t}(x,t)\leq\frac{u(x,t)}{(2-p)t},
\end{equation}
hence the function $u(x,t)t^{-1/(2-p)}$ is non-increasing in time.
It follows that for any time $t\in(0,t_1)$, we have:
\begin{equation*}
u(x,t_1)\leq t^{-\frac{1}{2-p}}t_{1}^{\frac{1}{2-p}}u(x,t)\leq
t^{-\frac{1}{2-p}}T^{\frac{1}{2-p}}u(x,t).
\end{equation*}
We join this last inequality with \eqref{posit2} and we obtain our
main positivity result for solutions to MDP:
\begin{equation}
\left(\frac{p}{2-p}\right)^{p-1}\left(\frac{2-p}{2}\right)^{\frac{2(p-1)}{p}}
\frac{\rho^{p-1}R^p}{2k_0T}\frac{1}{|B_{R_0}|}\int_{B_{R}}u_0\,\dx
\leq t^{-\frac{p-1}{2-p}}T^{\frac{p-1}{2-p}}u(x_0,t)^{p-1}.
\end{equation}
We conclude by letting
\begin{equation*}
k(n,p)=2k_0\,\rho^{p-1}\,\frac{2-p}{p}\left(\frac{2}{2-p}\right)^{\frac{2}{p}}.
\end{equation*}
We thus proved the following positivity theorem for solutions to
MDP.
\begin{theorem}\label{thm.MDP.pos}
Let $1<p<2$, let $u$ be the solution to the Minimal Dirichlet
Problem {\rm (\ref{MDP})} and let $T$ be its finite extinction time.
Then $T>t^{*}$ and the following inequality holds true for any
$t\in(0,t^{*}]$:
\begin{equation}\label{posit3}
u(x_{0},t)^{p-1} \geq k(n,p)
t^{\frac{p-1}{2-p}}T^{-\frac{1}{2-p}}\frac{R^p}{|B_{R_0}|}\int_{B_{R}}u_0\,\dx.
\end{equation}
In particular, the estimate {\rm (\ref{posit3})} establishes the
positivity of $u$ in the interior ball of the annulus up to the
critical time $t^{*}$ expressed by {\rm (\ref{crittime2})}.
\end{theorem}

\subsection{Aronson-Caffarelli type estimates}

We have obtained positivity estimates for initial times, namely
$t\in(0,t^*)$ and now we want to see whether it is possible to
extend such positivity estimates globally in time, i.\,e.,  for any
$t\in(0,T)$. This can be done and leads to some kind of inequalities
in the form of the celebrated Aronson-Caffarelli estimates valid for
the degenerate/slow diffusions, cf. \cite{AC83}. As a precedent two
of the authors proved in \cite{BV} some kind of Aronson-Caffarelli
estimates for the fast diffusion equation.

We begin by rewriting the positivity estimates in the form of the
following alternative: either $t>t^{*}$, or
\begin{equation*}
u(x_{0},t)^{p-1} \geq k(n,p)
t^{\frac{p-1}{2-p}}T^{-\frac{1}{2-p}}\frac{R^p}{|B_{R_0}|}\int_{B_{R}}u_0\,\dx,
\end{equation*}
We recall now the expression of $t^{*}$ given in \eqref{crittime2}
\[
t^{*}=k_* R^{p-n(2-p)}\|u_0\|_{L^{1}(B_R(x_0))}^{2-p}.
\]
The above inequalities can be summarized in the following equivalent alternative: either
\begin{equation*}
\frac{1}{|B_{R_0}|}\int_{B_R}u_0(x)\,\dx\leq
C_{1}(n,p)t^{\frac{1}{2-p}}R^{-\frac{p}{2-p}},
\end{equation*}
or
\begin{equation*}
\frac{1}{|B_{R_0}|}\int_{B_R}u_0(x)\,\dx\leq
k(n,p)t^{-\frac{p-1}{2-p}}T^{\frac{1}{2-p}}\,R^{-p}u(x_0,t)^{p-1}.
\end{equation*}
Summing up the above estimates, we obtain, for any $t\in(0,T)$,
\begin{equation}\label{AC.estimates}
R^{-n}\int_{B_R}u_0(x)\,\dx\leq
C_{1}t^{\frac{1}{2-p}}R^{-\frac{p}{2-p}}+C_{2}t^{-\frac{p-1}{2-p}}T^{\frac{1}{2-p}}R^{-p}u(x_0,t)^{p-1},
\end{equation}
where $C_1$ and $C_2$ are constants depending only on $n$ and $p$.

As already mentioned, the above Aronson-Caffarelli type estimates
are global in time, but they provide quantitative lower bounds only
for $0<t<t^*$. As far as we know, this kind of lower parabolic
Harnack inequalities are new for the $p$-Laplacian.

\noindent \textbf{Remark.} Let us notice that, even
working with initial data $u_0\in L^2(B_R)$, we never use the $L^2$
norm of the initial datum in a quantitative way, but only its $L^1$
norm. This observation allows for the approximation argument
described in the next section.

\section{Positivity for continuous local weak solutions}\label{sec.mainposit}

Throughout this section, $u$ will be a non-negative and continuous
local weak solution, cf. Definition \ref{local.weak}, defined in
$Q_T=\Omega\times(0,T)$, taking initial data $u_0\in
L^{1}_{loc}(\Omega)$. We recall that $B_{R_0}(x_0)\subset\Omega$ and
assume in all this section that $R_{0}>5R$, in order to compare $u$
and the solution $u_D$ of a suitable Minimal Dirichlet Problem. We
never use the modulus of continuity of $u$.

\subsection{Proof of Theorems \ref{mainposit} and \ref{main.AC}}

Fix a time $t\in(0,T_1)$ and a point
$x_1\in\overline{B_{R}(x_0)}$, so that $B_{R}(x_0)\subset
B_{2R}(x_1)\subset B_{(4+\e)R}(x_1)\subset B_{R_0}(x_0)$, for some
$\e>0$ sufficiently small (more precisely, $\e>0$ should satisfy
$R_0>(5+\e)R$). Since $u_{0}\chi_{B_{R}(x_0)}\in L^1(B_{R}(x_0))$,
we can approximate it with functions $u_{0,j}\in L^2(B_{R}(x_0))$,
such that $u_{0,j}\to u_0\chi_{B_{R}(x_0)}$ as $j\to\infty$ in the
space $L^{1}(B_{R}(x_0))$. We consider now the following sequence of
minimal Dirichlet problems in a ball centered at $x_1$:
\begin{equation*}
\left\{\begin{array}{ll}u_t=\Delta_{p}u, \ \ &
\hbox{in} \ B_{(4+\e)R}(x_1)\times(0,T),\\
u(x,0)=u_{0,j}(x)\chi_{B_{R}(x_0)}(x), \ \  &\hbox{in} \ B_{(4+\e)R}(x_1),\\
u(x,t)=0, \ \ & \hbox{for} \
t>0 \ \hbox{and} \ x\in\partial B_{(4+\e)R}(x_1),\end{array}\right.
\end{equation*}
which, by standard theory (see \cite{DiB}), admits a unique
continuous weak solution $u_{D,j}$, for which Theorem
\ref{thm.MDP.pos} applies. We then compare $u_{D,j}$ with the
continuous solution to the problem $(\mathbb{D})$, which is our
local weak solution $u$ restricted to $B_{(4+\e)R}(x_1)\times(0,T)$.
It follows that
\begin{equation*}
u(x,t)\ge u_{D,j}(x,t), \qquad\mbox{and}\qquad T\geq T_{m,j},
\end{equation*}
where $T_{m,j}$ is the finite extinction time for $u_{D,j}$. We then
apply Theorem \ref{thm.MDP.pos} to $u_{D,j}$ to obtain
\begin{equation*}
\begin{split}
u_{D,j}(x_1,t)^{p-1}&\geq
c\,R^{p}\,t^{\frac{p-1}{2-p}}T_{m,j}^{-\frac{1}{2-p}}\frac{1}{|B_{R_0}(x_1)|}\int_{B_{(4+\e)R}(x_1)}u_{0,j}(x)\chi_{B_R(x_0)}(x)\,\dx\\&\geq
c(n,p)R^{p-n}t^{\frac{p-1}{2-p}}T_{m,j}^{-\frac{1}{2-p}}\int_{B_{R}(x_0)}u_{0,j}(x)\,\dx,
\end{split}
\end{equation*}
provided that $t<t^{*}_{j}$, with $t^{*}_{j}$ as in the previous
section (but applied to $u_{0,j}$). Taking into account that
$u_{D,j}(x_1,t)\leq u(x_1,t)$ and that, in the previous estimates,
$t^{*}_{j}$ and $T_{m,j}$ depend only on the $L^1$ norm of
$u_{0,j}$, we can pass to the limit in order to find that
\begin{equation*}
u(x_1,t)^{p-1}\geq
c(n,p)R^{p-n}t^{\frac{p-1}{2-p}}T_{m}^{-\frac{1}{2-p}}\int_{B_{R}(x_0)}u_{0}(x)\,\dx\,,
\end{equation*}
where $T_m=T_m(u_0)=\lim\limits_{j\to\infty}T_{m,j}$, provided that
$t<t^{*}=\lim\limits_{j\to\infty}t^{*}_{j}$, as in the previous
section. Moreover, $t^{*}$ and $T_m$ do not depend on
the choice of the point $x_1\in\overline{B_{R}(x_0)}$, but only on
the support of the initial data which is fixed, we can take
$x_1=x_1(t)$ as the point where
\begin{equation*}
u(x_1,t)=\inf\limits_{x\in B_{R}(x_0)}u(x,t).
\end{equation*}
Thus, we arrive to the desired inequality \eqref{ineq.mainpos}.
Moreover, by the same comparison we get the Aronson-Caffarelli type
estimates \eqref{AC.estimates.main} for any continuous local weak
solution. \qed

\noindent \textbf{Remark.} The fact that $T(u)\ge T_m=T_m(u_0)$ for
any continuous local weak solution $u$ justifies the name of
\textit{minimal life time} that we give to $T_m$ in the
Introduction.

\subsection{Proof of Theorem \ref{posit.goodFDE}}

Let $p_c<p<2$. We divide the proof of Theorem \ref{posit.goodFDE}
into several steps, following the lines of the similar result in
\cite{BV}.

\noindent \textbf{Step 1. Scaling.} Let $u_R$ be the solution of the
homogeneous Dirichlet problem in the ball $B_R(x_0)$, with initial
datum $u_0\in L^{1}(B_R)$ and with extinction time
$T(u_0,R)<\infty$. Then the rescaled function
\begin{equation*}
u(x,t)=\frac{M}{R^n}\,\overline{u}\left(\frac{x-x_0}{R},\frac{t}{R^{np-2n+p}M^{2-p}}\right),
\quad M=\int\limits_{B_R}u_0\,\dx\,,
\end{equation*}
solves the homogeneous Dirichlet problem in $B(0,1)$, with initial
datum $\overline{u}_{0}$ of mass 1 and with extinction time
$\overline{T}$ such that $T(u_0,R)=R^{np-2n+p}M^{2-p}\overline{T}$.
Therefore, we can work in the unit ball and with rescaled solutions.

\noindent \textbf{Step 2. Barenblatt-type solutions.} Consider the
solution $\cb$ of the homogeneous Dirichlet problem in the unit ball
$B(0,1)$, with initial trace the Dirac mass, $\cb(0)=\delta_0$. By
comparison with the Barenblatt solutions of the Cauchy problem (that
exist precisely for $p_c<p<2$), we find that
\begin{equation*}
\cb(x,t)\leq C(n,p)t^{-n\vartheta_1}, \quad \mbox{for} \ \mbox{any}
\ (x,t)\in B(0,1)\times[0,\infty).
\end{equation*}
By the concentration-comparison principle (see \cite{PME},
\cite{VazquezSmoothing}), it follows that the solution $\cb$
extinguishes at the later time among all the solutions with initial
datum of mass 1, call $T(\cb)$ its extinction time. We have to prove
that $T(\cb)<\infty$, that will be done by comparison with another
solution, described below.

\noindent \textbf{Step 3. Separate variable solution.} Let us
consider the solution
\begin{equation*}
U_{r}(x,t)=(T_1-t)^{\frac{1}{2-p}}X(x), \quad \mbox{in} \ B_r, \
r>1,
\end{equation*}
with extinction time $T_1$ to be chosen later. Then, $X$ is a
solution of the elliptic equation $\Delta_{p}X+X/(2-p)=0$ in
$B_{R_0}$, hence it can be chosen radially symmetric and bounded
from above and from below by the distance to the boundary. On the
other hand, fix $t_0>0$ and let $T_1$ be given by
$X(1)(T_1-t_0)^{1/(2-p)}=C(p,n)t_0^{-n\vartheta_1}$.

\noindent \textbf{Step 4. Comparison and end of proof.} We compare
the solutions $\cb$ and $U_r$ constructed above in the cylinder
$Q_1=B_1(0)\times[t_0,T_1)$. The comparison on the boundary is
trivial and the initial data (at $t=t_0$) are ordered by the choice
of $t_0$. It follows that $\cb(x,t)\leq U_{r}(x,t)$ in $Q_1$, hence
their extinction times are ordered: $T(\cb)\leq T_1<\infty$.
Moreover, it is easy to check (by optimizing in $t_0$) that $T_1$
depends only on $p$ and $n$, hence $\overline{T}\leq T(\cb)\leq
K(n,p)$, for any solution of the homogeneous Dirichlet problem in
$B_1$ with extinction time $\overline{T}$. Coming back to the
original variables, we find that
\begin{equation*}
T(u_0,R)\leq K(n,p)R^{np-2n+p}\|u_0\|_{L^{1}(B_R)}^{2-p},
\end{equation*}
which is the upper bound of Theorem \ref{posit.goodFDE}. The lower
bound has been obtained in Subsection \ref{lower.T}. The lower
Harnack inequality \eqref{lowerH.goodFDE} follows immediately from
estimate \eqref{ineq.mainpos}. \qed

\subsection{Upper bounds for the extinction time and proof of Theorem \ref{posit.VFDE}}\label{FET}

In this subsection we prove universal upper estimates for the finite
extinction time $T$, in the range $1<p<p_c$, in terms of suitable
norms of the initial datum $u_0$, and we subsequently prove Theorem
\ref{posit.VFDE}. Throughout this subsection, $u$ is a solution to a
global homogeneous Dirichlet or Cauchy problem in
$\Omega\subseteq\real^n$, with initial datum $u_0$, whose regularity
will be treated below.

\medskip

\noindent \textbf{Bounds in terms of the $L^{r_c}$ norm.} Following
the ideas of Benilan and Crandall \cite{BC2}, we begin by
differentiating in time the global $L^r$ norm of the solution $u(t)$
to a global (Cauchy or Dirichlet) problem:
\begin{equation}\label{est.5}
\begin{split}
\frac{\rd}{\dt}\int\limits_{\Omega}u^r\,\dx&=-r(r-1)\int\limits_{\Omega}u^{r-2}|\nabla
u|^p\,\dx=-\frac{r(r-1)p^p}{(r+p-2)^p}\int\limits_{\Omega}\left|\nabla
u^{\frac{r+p-2}{p}}\right|^p\,\dx\\&\leq-\frac{r(r-1)p^p\cs_{p}^p}{(r+p-2)^p}
\left[\int\limits_{\Omega}u^{\frac{(r+p-2)p^*}{p}}\,\dx\right]^{\frac{p}{p^*}},
\end{split}
\end{equation}
where in the last step we used the Sobolev inequality; here,
$p^*=np/(n-p)$ and $\cs_p$ is the Sobolev constant. Note that
$(r+p-2)p^*/p=r$ if and only if $r=r_c$. If $p>p_c$, then $r_c<1$,
hence the global $L^{r_c}$ norm increases, originating a
\textsl{Backward Effect} (see \cite{VazquezSmoothing}).

\noindent We thus restrict ourselves to $p<p_c$, in which case the
constant $r_c(r_c-1)p^p/(r_c+p-2)^p$ is positive. Then,
\eqref{est.5} implies the following closed differential inequality
\begin{equation*}
\frac{\rd}{\dt}\|u(t)\|_{r_c}^{r_c}\leq-\frac{r_c(r_c-1)p^p\cs_{p}^p}{(r_c+p-2)^p}
\|u(t)\|_{r_c}^{\frac{pr_c}{p^*}},
\end{equation*}
whose integration leads to
\begin{equation}\label{est.BC}
\|u(t)\|_{r_c}^{2-p}\leq\|u(s)\|_{r_c}^{2-p}-K(t-s), \qquad
K=\frac{r_c(r_c-1)p^{p+1}\cs_{p}^p}{n(r+p-2)^p},
\end{equation}
which holds for any $0\leq s\leq t\leq T$ and for any $p<p_c$.
Letting now $s=0$ and $t=T$ in \eqref{est.BC}, we obtain the
following universal upper bound for the extinction time:
\begin{equation}\label{upper.T1}
T\leq K^{-1}\|u_0\|_{r_c}^{2-p}.
\end{equation}
In particular, if the initial datum $u_0\in L^{r_c}(\Omega)$, then
the solution $u$ extinguishes in finite time.

\medskip

\noindent \textbf{Bounds in terms of other $L^r$ norms.} As we have
seen, the condition $u_0\in L^{r_c}(\Omega)$ does not allow for the
Local Smoothing Effect to hold. That is why, in this part we obtain
upper bounds for the extinction time $T$ in terms of other global
$L^r$ norms, with the expected condition $r>r_c$, but only in
bounded domains $\Omega$. Following ideas from \cite{BGV} and
\cite{BV}, we consider  a function $f\in
W^{1,p}_{0}(\Omega)$, and we apply the Poincar\'e, Sobolev and
H\"older inequalities as follows:
\begin{equation}\label{est.6}
\|f\|_{q}\leq\|f\|_{p}^{\vartheta}\|f\|_{p^*}^{1-\vartheta}\leq\cp_{\Omega}^{\vartheta}\cs_{p}^{1-\vartheta}\|\nabla
f\|_{p},
\end{equation}
for any $q\in(p,p^*)$, where $\vartheta\in(0,1)$, $\cp_{\Omega}$ is
the Poincar\'e constant of the domain $\Omega$ and $\cs_{p}$ is the
Sobolev constant. We let in \eqref{est.6}
\begin{equation*}
f=u^{\frac{r+p-2}{p}}, \qquad q=\frac{pr}{r+p-2}, \qquad
\vartheta=\frac{r-r_c}{r},
\end{equation*}
which are in the range where this inequality applies, since $q>p$
for any $p<2$ and $q<p^*$ if and only if $r>r_c$. We then restrict
ourselves to the case $r>r_c$ and, replacing in \eqref{est.6}, we
obtain
\begin{equation}\label{est.7}
\|u\|_{r}^{\frac{r+p-2}{rp}}\leq\cp_{\Omega}^{1-\frac{r_c}{r}}\cs_{p}^{\frac{r_c}{r}}\left\|\nabla
u^{\frac{r+p-2}{p}}\right\|_{p}.
\end{equation}
We elevate \eqref{est.7} at power $p$ and join it then with the
inequality \eqref{est.5} for the derivative of the global $L^r$
norm. It follows that
\begin{equation*}
\frac{\rd}{\dt}\|u(t)\|_{r}^{r}=-\frac{r(r-1)p^p}{(r+p-2)^p}\left\|\nabla
u^{\frac{r+p-2}{p}}\right\|_p^{p}\leq
K_0\|u(t)\|_{r}^{\frac{r+p-2}{r}},
\end{equation*}
where
\begin{equation*}
K_0:=\frac{r(r-1)p^p\,\cs_{p}^{\frac{r}{p\,r_c}}\cp_{\Omega}^{\frac{p(r-r_c)}{r}}}{(r+p-2)^p}\,.
\end{equation*}
By integration over $[s,t]\subseteq[0,T]$, we obtain that
\begin{equation}\label{est.8}
\|u(t)\|_{r}^{2-p}\leq\|u(s)\|_{r}^{2-p}-K_0(t-s),
\end{equation}
for any $0\leq s\leq t\leq T$ and for any $r>r_c$. We let now $s=0$,
$t=T$ in \eqref{est.8} and we obtain an upper bound for the
extinction time:
\begin{equation}\label{upper.T2}
T^{\frac{1}{2-p}}\leq K_0^{-\frac{1}{2-p}}\|u_0\|_{r}
=\left[\frac{r(r-1)p^p\,\cs_{p}^{\frac{r}{p\,r_c}}
\cp_{\Omega}^{\frac{p(r-r_c)}{r}}}{(r+p-2)^p}\right]^{-\frac{1}{2-p}}\|u_0\|_{r}
=c_1R^{-\frac{rp+n(p-2)}{r(2-p)}} \|u_0\|_{r}\,,
\end{equation}
since the Poincar\'e constant $\mathcal{P}_{\Omega}\sim R$ and where
$c_1$ only depends on $p,r,n$ and goes to zero as $r\to 1$. In
particular, any solution $u$ of a homogeneous Dirichlet problem in
$\Omega$, with $u_0\in L^r$, $r>r_c$, extinguishes in finite time.

\noindent \textbf{Remarks.} (i) The above results prove that a
global Sobolev and Poincar\'e inequality implies that the solution
extinguishes in finite time and gives quantitative upper bounds for
the extinction time $T$.

\noindent (ii) Direct applications of these bounds in the estimates
\eqref{ineq.mainpos} and \eqref{AC.estimates.main} prove Theorem
\ref{posit.VFDE}.

\section{Harnack inequalities}\label{sec.Harnack}

By joining the local lower and upper bounds obtained
in the previous parts of the paper, we obtain various forms of
Harnack inequalities. These are expressions
relating the maximum and minimum of a solution inside certain
parabolic cylinders. In the well known linear case one has
\begin{equation}
\sup_{Q_1} u(x,t)\le C \,\inf_{Q_2} u(x,t).
\end{equation}
The main idea is that the formula applies for a large class of
solutions and the constant $C$ that enters the relation does not
depend on the particular solution, but only on the data like $p,n$
 and the  size of the cylinder. The
cylinders in the standard case are supposed to be ordered,
$Q_1=B_{R_1}(x_0)\times[t_1,t_2]$,
$Q_2=B_{R_2}(x_0)\times[t_3,t_4]$, with $t_1\le t_2<t_3\le t_4$ and
$R_1<R_2$.

It is well-known that in the degenerate nonlinear elliptic or
parabolic problems a plain form of the inequality does not hold. In
the  work of DiBenedetto and collaborators, see the book \cite{DiB}
or the recent work \cite{DGV}, versions are obtained where some
information of the solution is used to define so-called intrinsic
sizes, like the size of the parabolic cylinder(s), that usually
depends on $u(x_0,t_0)$. They are called {\sl intrinsic Harnack
inequalities}.

The Harnack Inequalities of \cite{DiB,DGV}, in the supercritical
range then read: \textit{There exist positive constants
$\overline{c}$ and $\overline{\delta}$ depending only on $p,n$, such
that for all $(x_0,t_0)\in \Omega\times(0,T)$ and all cylinders of
the type
\begin{equation}\label{Intr.cyl.Dib}
B_R(x_0)\times\left(t_0-c\,u(x_0,t_0)^{2-p}(8R)^p,
t_0+c\,u(x_0,t_0)^{2-p}(8R)^p\right)\subset\Omega\times(0,T)\,,
\end{equation}
we have
\[
\overline{c}\,u(x_0,t_0)\le \inf_{x\in B_R(x_0)}u(x,t)\,,
\]
for all times $t_0-\overline{\delta}\,u(x_0,t_0)^{2-p}\,R^p<t<t_0
+\overline{\delta}\,u(x_0,t_0)^{2-p}\,R^p$. The constants
$\overline{\delta}$ and $\overline{c}$ tend to zero as $p\to 2$ or
as $p\to p_c$\,.}

They also give a counter-example in the lower range $p<p_c$, by
producing an explicit local solution that does not satisfy any kind
of Harnack inequality (neither of the types called intrinsic,
elliptic, forward,  backward) if one fixes ``a priori'' the
constant $c$\,. At this point a natural question is posed:

\textit{What form may take the Harnack estimate, if any,  when $p$ is
in the subcritical range $1<p \le p_c$?}

\noindent We will give an answer to this question.

\medskip

If one wants to apply the above result to a local weak solution
defined on $\Omega\times[0,T]$, where $T$ is possibly the extinction
time, one should care about the size of the intrinsic cylinder,
namely the intrinsic hypothesis \eqref{Intr.cyl.Dib} reads
\begin{equation}\label{Intr.hyp}
\overline{c}\,u(x_0,t_0) \le
\left[\frac{\min\{t_0,T-t_0\}}{(8R)^p}\right]^{\frac{1}{2-p}}\qquad\mbox{and
}\qquad \dist(x_0,\partial\Omega)<\frac{R}{8}\,,
\end{equation}
This hypothesis is guaranteed in the good range by the fact that solutions with initial data in $\LL^1_{\rm loc}$ are bounded, while in the very fast diffusion range hypothesis
\eqref{Intr.hyp} fails, and should be replaced by :
\[
u(x_0,t)
    \le \frac{c_{p,n} }{\varepsilon^{\frac{2r\vartheta_r}{2-p}}} \left[\frac{\|u(t_0)\|_{\LL^r({B_R})}\,R^d}{\|u(t_0)\|_{\LL^1({B_R})}R^{\frac{n}{r}}}
            \right]^{2r\vartheta_r}\;\left[\frac{t_0}{R^p}\right]^{\frac{1}{2-p}}\,.
\]
This local upper bound can be derived by our local smoothing effect
of Theorem \ref{mainupper0}, whenever $t_0+\varepsilon
t^{*}(t_0)<t<t_0+t^{*}(t_0)$, where the critical time is defined by
a translation in formula  \eqref{crittime2} as follows
\begin{equation}\label{t*t0}
t^{*}(t_0)=k^{*} R^{p-n(p-2)}\|u(t_0)\|_{L^{1}(B_R(x_0))}^{2-p}\,,
\end{equation}
full details are given below, in the proof of Theorem \ref{FBH}. In this new intrinsic geometry
we obtain the plain form of intrinsic Harnack inequalities of Theorem \ref{FBH}, namely

\textit{There exists constants $h_1\,,h_2$
depending only on $d,p,r$, such that, for any $\varepsilon\in[0,1]$  the following inequality holds
\begin{equation*}
\inf\limits_{x\in B_{R}(x_0)}u(x,t\pm\theta)\geq
h_1\e^{\frac{rp\vartheta_r}{2-p}}\left[\frac{\|u(t_0)\|_{L^1(B_R)}R^{\frac{n}{r}}}
    {\|u(t_0)\|_{L^r(B_R)}R^{n}}\right]^{rp\vartheta_r+\frac{1}{2-p}}u(x_0,t),
\end{equation*}
for any $t_0+\e t^*(t_0)< t \pm \theta<t_0+t^*(t_0).$}

\medskip

We have obtained various forms of Harnack inequalities, namely

\noindent\textsl{Forward Harnack inequalities.} These inequalities compare
the supremum at a time $t_0$ with the infimum of the solution at a later time
$t_0+\vartheta$. These kind of Harnack inequalities hold
for the linear heat equation as well: we recover the classical result
just by letting $p\to 2$.

\noindent \textsl{Elliptic-type Harnack inequalities.} These
inequalities are typical of the fast diffusion range, indeed they
compare the infimum and the supremum of the solution at the same
time, namely consider $\theta=0$ above. It is false for the Heat
equation and for the degenerate $p$-Laplacian, as one can easily
check by plugging respectively the gaussian heat kernel or the
Barenblatt solutions. This kind of inequalities are true for the
fast diffusion processes, as noticed by two of the authors in
\cite{BV, BV2} and by DiBenedetto et al. in \cite{DGV, DUV} in the
supercritical range.

\noindent \textsl{Backward Harnack inequalities}. These inequalities compare
the supremum at a time $t_0$ with the infimum of the solution at a previous time
$t_0-\vartheta$. This backward inequality is a typical feature of the fast diffusion
processes, that somehow takes into account the phenomena of extinction in finite time,
as already mentioned in Subsection 2.4.

In the very fast diffusion range $1<p\le p_c$ our intrinsic Harnack inequality represents the first and only known result. In the good range, $p_c<p<1$ we can take $r=1$, so that the ratio of $L^r$ norms simplifies and we recover the result of \cite{DiB, DGV} with a different proof.

Throughout this section $T_m$ will denote
the finite extinction time for the minimal Dirichlet problem
\eqref{MDP}, i.\,e. the so-called minimal life time of any
continuous local weak solution.

\subsection{Intrinsic Harnack inequalities. Proof of Theorem \ref{FBH}}\label{subsect.Harnack.FBH}

Let $u$ be a nonnegative, continuous local weak solution of the fast
$p$-Laplacian equation in a cylinder $Q=\Omega\times (0,T)$, with
$1<p<2$, taking an initial datum $u_0\in L^r_{\rm loc}(\Omega)$,
with $r\geq\max\{1,r_c\}$. Let $x_0\in\Omega$ be a fixed point, such
that $\dist(x_0,\partial\Omega)>5R$. We recall the notation
$T_m$ for the minimal life time associated to the initial data $u_0$
and the ball $B_R(x_0)$, and we denote the critical time
$$
t^{*}(s)=k^{*} R^{p-n(p-2)}\|u(s)\|_{L^{1}(B_R(x_0))}^{2-p}, \quad
t^{*}=t^{*}(0),
$$
which is a shift in time of the expression \eqref{crittime2}.

With these notations and assumptions, we first prove a generalized
form of the Harnack inequality, that holds for initial times, or equivalently for small intrinsic cylinders,  and in which we allow the constants to depend also on $T_m$.
\begin{theorem}\label{Har4}
For any $t_0\in(0,t^{*}]$, and any $\theta\in[0,t_0/2]$ such that
$t_0+\theta\leq t^{*}$, the following forward/backward/elliptic
Harnack inequality holds true:
\begin{equation}\label{Harnack3}
\inf_{x\in B_{R}(x_0)}u(x,t_0\pm\theta)\geq Hu(x_0,t_0),
\end{equation}
where
\begin{equation*}
H=CR^{\frac{np-2n+p}{(p-1)(2-p)}}\left[\frac{\|u_0\|_{L^{1}(B_R)}}{T_m^{\frac{1}{2-p}}}\right]^{\frac{1}{p-1}}
\left[R^{\frac{p}{2-p}}\frac{\big\|u_0\big\|_{L^{r}(B_{R})}^{rp\vartheta_r}}
{t_{0}^{\frac{rp\vartheta_r}{2-p}}}+1\right]^{-1}.
\end{equation*}
and $C$ depends only on $r$, $p$, $n$.  $H$ goes to 0 as
$t_0\to 0$.
\end{theorem}
\begin{proof}
Let us recall first that, from Theorem \ref{mainposit},
$u(x_0,t_0)>0$ for $t_0<t^{*}$. Let us fix $t_0\in(0,t^{*})$ and
choose $\theta>0$ sufficiently small such that $t_0+\theta\leq
t^{*}$ and $t_0\pm\theta\ge t_0/3$. We plug these quantities into
the lower estimate \eqref{mainposit} to get:
\begin{equation*}
\begin{split}
\inf_{x\in B_R}u(x,t_0\pm\theta)&\geq
C(t_0\pm\theta)^{\frac{1}{2-p}}R^{\frac{p-n}{p-1}}T_m^{-\frac{1}{(2-p)(p-1)}}\|u_0\|_{L^{1}(B_R)}^{\frac{1}{p-1}}\\&\geq
C\left(\frac{t_0}{3R_0^p}\right)^{\frac{1}{2-p}}R^{\frac{np-2n+p}{(p-1)(2-p)}}T_m^{-\frac{1}{(2-p)(p-1)}}\|u_0\|_{L^{1}(B_R)}^{\frac{1}{p-1}}.
\end{split}
\end{equation*}
On the other hand, we use the local upper bound
\eqref{ineq.mainpos}, in the following way:
\begin{equation*}
u(x_0,t_0)\leq
C_{3}\left[R^{\frac{p}{2-p}}\frac{\big\|u_0\big\|_{L^{r}(B_{2R})}^{rp\vartheta_r}}
{t_{0}^{\frac{rp\vartheta_r}{2-p}}}+1\right]\left(\frac{t_{0}}{R^p}\right)^{\frac{1}{2-p}}.
\end{equation*}
Joining the two previous estimates, we obtain the desired form of
the inequality.
\end{proof}
We are now ready to prove Theorem \ref{FBH}, which is our main
intrinsic Harnack inequality.

\noindent \textbf{Proof of Theorem \ref{FBH}} We may assume that
$t_0=0$, hence $t^{*}(t_0)=t^{*}$; the general result follows by
translation in time. We use again the local smoothing effect of
Theorem \ref{mainupper0} as before and we estimate:
\begin{equation}\label{Har.upper}
\begin{split}
u(x_0,t_0)&\leq
C_3\left[1+\frac{\|u_0\|_{L^{r}(B_R)}^{rp\vartheta_r}}{t_0^{\frac{rp\vartheta_r}{2-p}}}R^{\frac{p}{2-p}}\right]\left[\frac{t_0}{R^{2}}\right]^{\frac{1}{2-p}}
\leq C_4\left[\frac{\|u_0\|_{L^{r}(B_R)}^{rp\vartheta_r}}{(\e
t^{*})^{\frac{rp\vartheta_r}{2-p}}}R^{\frac{p}{2-p}}\right]\left[\frac{t_0}{R^{2}}\right]^{\frac{1}{2-p}}\\
&\leq\frac{C_5}{\e^{\frac{rp\vartheta_r}{2-p}}}\left[\frac{\|u_0\|_{L^r(B_R)}R^{n}}{\|u_0\|_{L^1(B_R)}R^{\frac{n}{r}}}\right]^{rp\vartheta_r}\left[\frac{t_0}{R^{2}}\right]^{\frac{1}{2-p}}\,,
\end{split}
\end{equation}
where the second step in the inequality above follows from the
assumption that $t_0\ge\e t^{*}$. On the other hand, we can remove the dependence on $T_m$
in the lower estimate of Theorem \ref{Har4}, using the results in Subsection \ref{FET}, namely:
\begin{equation*}
T_m^{\frac{1}{2-p}}\leq
C(r,p,n)R^{\frac{p}{2-p}-\frac{n}{r}}\|u_0\|_{L^r(B_R)}, \quad
r\geq\max\{1,r_c\},
\end{equation*}
hence the lower estimate becomes
\begin{equation}\label{Har.lower}
\inf\limits_{x\in B_{R}(x_0)}u(x,t\pm\theta)\geq
C_6\left[\frac{\|u_0\|_{L^1(B_R)}R^{\frac{n}{r}}}{\|u_0\|_{L^r(B_R)}R^{n}}\right]^{-\frac{1}{p-1}}\left[\frac{t_0}{R^{2}}\right]^{\frac{1}{2-p}}.
\end{equation}
Joining the inequalities \eqref{Har.upper} and \eqref{Har.lower}, we
obtain the estimate \eqref{intr.Harnack} as stated. We pass from
$[0,t^{*}]$ to any interval $[t_0,t_0+t^{*}(t_0)]$ by translation in
time. \qed

\noindent \textbf{Alternative form of the Harnack inequality}. The
following alternative form of the Harnack inequality is
given avoiding the intrinsic geometry and the waiting time $\e\in [0,1]$.
An analogous version, for the degenerate diffusion of $p$-Laplacian type, can be found
in \cite{DGV-alt}.

\begin{theorem}\label{Har1}
Under the running assumptions, there exists $C_1$, $C_2>0$,
depending only on $r$, $n$, $p$, such that the following inequality
holds true:
\begin{equation}\label{Harnack1}
\sup_{x\in B_R}u(x,t)\leq
C_1\frac{\|u(t_0)\|_{L^{r}(B_{2R})}^{rp\vartheta_r}}{t^{n\vartheta_r}}+
C_2\left[\frac{\|u(t_0)\|_{L^r(B_R)}R^{n}}{\|u(t_0)\|_{L^1(B_R)}R^{\frac{n}{r}}}\right]^{\frac{1}{p-1}}
\inf_{x\in B_{R}}u(x,t\pm\theta),
\end{equation}
for any $0\leq t_0<t\pm\theta<t_0+t^{*}(t_0)<T$.
\end{theorem}
The proof is very easy and it consists only in joining the upper
estimate \eqref{mainsmooth0} with the lower estimate
\eqref{Har.lower} above. We leave the details to the interested
reader.

\noindent \textbf{Remark.} In the good fast-diffusion range $p>p_c$,
we can let $r=1$ and obtain
\begin{equation*}
\sup_{x\in B_R}u(x,t)\leq
C_1\frac{\|u(t_0)\|_{L^{1}(B_{2R})}^{p\vartheta_1}}{t^{n\vartheta_1}}+C_2\inf_{x\in
B_{R}}u(x,t\pm\theta).
\end{equation*}


\section{Special Energy Inequality. Rigorous Proof of Theorem \ref{inequality}}\label{Sec.Energy}

We devote this section to the proof of Theorem \ref{inequality}, and
to further generalizations and applications of it. Throughout this
section, by admissible test function we mean $\varphi\in
C^{2}_c(\Omega)$ as specified in the statement of Theorem
\ref{inequality}.

We have presented in the Introduction the basic, formal calculation
leading to inequality \eqref{mainid}. Our task here will be to give
a detailed justification of this formal proof. To this end we state
and prove in full detail an auxiliary result.

\begin{proposition}\label{Ineq.Phi}
Let $\Phi:\real\to\real$ be a strictly positive smooth function, let
$\varphi\ge 0$ be a nonnegative admissible test function. Define the associated
$\Phi$-Laplacian operator
\begin{equation}
\Delta_{\Phi}u:={\rm div}\left[\Phi'\left(|\nabla u|^2\right)\nabla
u\right].
\end{equation}
Then the following inequality holds true for continuous weak
solutions to the $\Phi$-Laplacian evolution equation
\begin{equation}\label{mainidPhi}
\frac{\rd}{\dt}\int_{\Omega}\Phi\left(|\nabla
u|^2\right)\varphi\,\dx+
\frac{2}{n}\int_{\Omega}(\Delta_{\Phi}u)^{2}\varphi\,\dx\leq
\int_{\Omega}\left[\Phi'\left(|\nabla
u|^2\right)\right]^{2}\left(|\nabla u|^2\right)\Delta\varphi\,\dx.
\end{equation}
\end{proposition}

\noindent\textbf{Remark. } Let us remark that the $p$-Laplacian is
obtained by taking $\Phi(w)=\frac{2}{p}w^{p/2}$, but we stress the
fact that this choice of $\Phi$ falls out the smoothness requirement
of the above proposition.

\noindent {\sl Proof.~} This proof is a straightforward
generalization of the above formal proof of Theorem
\ref{inequality}. Denote $w=|\nabla u|^2$. Take a test function
$\varphi\ge 0$ as in the assumptions. We perform a time derivation
of the energy associated to the  $\Phi$-Laplacian
\begin{equation*}
\begin{split}
\frac{\rd}{\dt}\int_{\Omega}\Phi(w)\varphi\,\dx&=-2\int_{\Omega}\hbox{div}\big[\Phi'(w)(\nabla
u)\,\varphi\big]\Delta_{\Phi}u\,\dx\\&=-2\int_{\Omega}(\Delta_{\Phi}u)^{2}\varphi\,\dx-
2\int_{\Omega}(\Delta_{\Phi}u)\Phi'(w) (\nabla
u\cdot\nabla\varphi)\,\dx.
\end{split}
\end{equation*}
We then apply identity (\ref{formula}) and inequality (\ref{ineq})
for the vector field $F=\Phi(w)|\nabla u|$ and finally obtain
\eqref{mainidPhi}.\qed

\medskip

\noindent The rest of the argument is based on suitable
approximations of the $p$-Laplacian equation by the
$\Phi$-Laplacians introduced above; it will be divided into several
steps.

\medskip

\noindent \textsc{Step 1. Approximating problems}. We now let
$\Phi_{\e}(w)=\frac{2}{p}(w+\e^2)^{p/2}$, which is our approximation
for the $p$-Laplacian nonlinearity. We also consider a fixed
sub-cylinder $Q'\subset Q_T$ of the form  $Q'=B_{R}\times (T_1,T_2)$
where $B_{R}\subset\Omega$ is a small ball and $0<T_1<T_2<T$.
 Choose moreover $T_1$ such that $\|\nabla
u(T_1)\|_{L^p(B_R)}=K<\infty$, which is true for a.\,e. time

We introduce the following approximating Dirichlet problem in $Q'$:
\begin{equation}\label{Eps.Problem}
(P_{\e})
\left\{\begin{array}{ll}u_{\e,t}=\Delta_{\Phi_{\e}}u_{\e}:=\hbox{div}\left[\left(|\nabla
u_{\e}|^2+\e^2\right)^{(p-2)/2}\nabla u_{\e}\right], \ \hbox{in} \
Q',\\u_{\e}(x,T_1)=u(x,T_1), \ \hbox{for} \ \hbox{any} \ x\in
B_R,\\u_{\e}(x,t)=u(x,t), \ \hbox{for} \ x\in\partial B_{R}, \
t\in(T_1,T_2).\end{array}\right.
\end{equation}
Since the equation in this problem is neither degenerate, nor
singular, and the boundary data are continuous by our assumptions,
the solution $u_{\e}$ of $(P_{\e})$ is unique and belongs to
$C^{\infty}(Q')$ (see \cite{LSU} for the standard parabolic theory),
hence the result of Proposition \ref{Ineq.Phi} holds true for
$u_\varepsilon$. Moreover, $u_{\e}$ satisfies the following weak
formulation:
\begin{equation}\label{local.eps}
\begin{split}
\int_{B_R}u_{\e}(x,t_2)&\varphi(x,t_2)\,\dx-\int_{B_R}u_{\e}(x,t_1)\varphi(x,t_1)\,\dx\\&+
\int_{t_1}^{t_2}\int_{B_R}\left[-u_{\e}(x,s)\varphi_{t}(x,s)+\left(|\nabla
u_{\e}|^2+\e^2\right)^{\frac{p-2}{2}}\nabla
u_{\e}(x,s)\cdot\nabla\varphi(x,s)\right]\,\dx\,\ds=0,
\end{split}
\end{equation}
for any times $T_1\leq t_1<t_2\leq T_2$ and for any test function
$\varphi\in W^{1,2}\big(T_1,T_2;L^{2}(B_R)\big)\cap
L^{p}\big(T_1,T_2;W_{0}^{1,p}(B_R)\big)$. Conversely, if a function
$v\in C^{\infty}(Q')$ satisfies the weak formulation
\eqref{local.eps} and takes as boundary values $u$ in the continuous
sense, then by uniqueness of the Dirichlet problem, we can conclude
$v=u_{\e}$.

\medskip

\noindent \textsc{Step 2: Uniform local energy estimates for
$u_{\e}$}. In the next steps, we are going to establish uniform
estimates (i.\,e. independent of $\e$) for some suitable norms of
the solution $u_{\e}$ to $(P_{\e})$. In the first part, we deal with
the local $L^p$ norm of the gradient of the solution. Starting from
\eqref{mainidPhi}, we have:
\begin{equation*}
\begin{split}
\frac{\rd}{\dt}\int_{B_R}\left(\e^2+|\nabla
u_{\e}|^{2}\right)^{\frac{p}{2}}&\varphi\,\dx\leq\frac{p}{2}\int_{B_R}\left(\e^2+|\nabla
u_{\e}|^{2}\right)^{p-1}\Delta\varphi\,\dx\\&\leq\frac{p}{2}\left[\int_{B_R}\left(\e^2+|\nabla
u_{\e}|^{2}\right)^{\frac{p}{2}}\varphi\,\dx\right]^{\frac{2(p-1)}{p}}\left[\int_{B_R}\varphi^{-\frac{2(p-1)}{2-p}}
(\Delta\varphi)^{\frac{p}{2-p}}\,\dx\right]^{\frac{2-p}{p}}\\&=C(\varphi)\left[\int_{B_R}\left(\e^2+|\nabla
u_{\e}|^{2}\right)^{\frac{p}{2}}\varphi\,\dx\right]^{\frac{2(p-1)}{p}},
\end{split}
\end{equation*}
where in the last inequality we applied H\"older inequality with the
exponents $p/(2-p)$ and $p/2(p-1)$, and we have set
\begin{equation}\label{c.phi}
C(\varphi)=\frac{p}{2}\left[\int_{B_R}\varphi^{-\frac{2(p-1)}{2-p}}
(\Delta\varphi)^{\frac{p}{2-p}}\,\dx\right]^{\frac{2-p}{p}}.
\end{equation}
We assume for the moment that $C(\varphi)<\infty$. We then arrive to
the following closed differential inequality:
\begin{equation*}
\frac{\rd}{\dt}Y_{\e}(t)\leq C(\varphi)Y_{\e}(t)^{\frac{2(p-1)}{p}},
\end{equation*}
where
\begin{equation*}
Y_{\e}(t)=\int_{B_R}\left(\e^2+|\nabla
u_{\e}(x,t)|^{2}\right)^{\frac{p}{2}}\varphi(x)\,\dx.
\end{equation*}
An integration over $(t_0,t_1)$ gives
\begin{equation*}
Y_{\e}(t_1)^{\frac{2-p}{p}}-Y_{\e}(t_0)^{\frac{2-p}{p}}\leq
C(\varphi)(t_1-t_0),
\end{equation*}
for any $T_1\leq t_0<t_1\leq T_2$. Letting $t_0=T_1$ and observing
that $t:=t_1-t_0<T$, we find:
\begin{equation*}
\left[\int_{B_R}\left(\e^2+|\nabla
u_{\e}(t)|^{2}\right)^{\frac{p}{2}}\varphi\,\dx\right]^{\frac{2-p}{p}}\leq
C(\varphi)T+\left[|B_R|+\|\nabla
u(T_1)\|_{L^{p}(B_R)}^p\right]^{\frac{2-p}{p}},
\end{equation*}
where in the last step we have used the numerical inequality
$(a+b)^{p/2}\leq a^{p/2}+b^{p/2}$, valid for any $a$, $b>0$ and
$p<2$. On the other hand, we see that
\begin{equation}\label{step2.ineq}
\int_{B_R}|\nabla
u_{\e}|^{p}\varphi\,\dx\leq\int_{B_R}\left(\e^2+|\nabla
u_{\e}|^2\right)^{\frac{p}{2}}\varphi\,\dx\leq \left[|B_R|+\|\nabla
u(T_1)\|_{L^{p}(B_R)}^p\right]^{\frac{2-p}{p}}+C(\varphi)T\,.
\end{equation}
From the choice of $T_1$ such that $\|\nabla
u(T_1)\|_{L^{p}(B_R)}<\infty$, it follows that the right-hand side
is uniformly bounded. Hence the family $\{|\nabla u_{\e}|\}$ has a
uniform bound in $L^{\infty}([T_1,T_2];L^{p}_{loc}(B_R))$, which
does not depend on $\e$. The choice of $\varphi$ such that
$C(\varphi)<\infty$ follows from Lemma \ref{choice.varphi}, part
(b), applied for $\b=p/(2-p)$.

 Finally, from standard results in measure theory we
know that the set of times $t\in(0,T)$ such that $\|\nabla
u(t)\|_{L^p(B_R)}<\infty$ is a dense set. Hence, for any $t_0\in
(0,T)$ given, there exists $T_1<t_0$ with the above property, and,
consequently, a generic parabolic cylinder $B_R\times[t_0,T_2]$ can
be considered as part of a bigger cylinder $B_R\times[T_1,T_2]$ with
$T_1$ as above, for which our approximation process applies.

\medskip

\noindent \textsc{Step 3. A uniform H\"older estimate for
$\{u_{\e}\}$}. We prove that the family $\{u_{\e}\}$ admits a
uniform H\"older regularity up to the boundary. We will use Theorem
1.2, Chapter 4 of \cite{DiB}, and to this end we change the
notations to $a(x,t,u,\nabla u)=\left(|\nabla
u|^2+\e^2\right)^{\frac{p-2}{2}}\nabla u$ and we prove the following
inequalities:

\noindent (a) Since $(2-p)/2<1$, we have that $\left(|\nabla
u|^2+\e^2\right)^{\frac{2-p}{2}}\leq|\nabla u|^{2-p}+\e^{2-p}$, and
\begin{equation*}
a(x,t,u,\nabla u)\cdot\nabla u=\frac{|\nabla u|^2}{(\e^2+|\nabla
u|^2)^{(2-p)/2}}\geq\frac{|\nabla u|^2}{\e^{2-p}+|\nabla u|^{2-p}}.
\end{equation*}
In order to apply the above mentioned result of \cite{DiB}, we have to find a constant
$C_0>0$ and a nonnegative function $\varphi_{0}$ such that
\begin{equation*}
\frac{|\nabla u|^2}{\e^{2-p}+|\nabla u|^{2-p}}\geq C_0|\nabla
u|^p-\varphi_{0}(x,t),
\end{equation*}
or equivalently
\begin{equation*}
\varphi_{0}(x,t)\geq\frac{C_0\e^{2-p}|\nabla u|^p-(1-C_0)|\nabla
u|^2}{\e^{2-p}+|\nabla u|^{2-p}}=\frac{1}{2}\frac{\e^{2-p}|\nabla
u|^p-|\nabla u|^2}{\e^{2-p}+|\nabla u|^{2-p}},
\end{equation*}
by taking $C_0=1/2$. If $|\nabla u|\geq\e$, then the right-hand side
in the previous inequality is nonpositive and the existence of
$\varphi_0$ is trivial. If $|\nabla u|<\e$, we can write:
\begin{equation*}
\begin{split}
\frac{\e^{2-p}|\nabla u|^p-|\nabla u|^2}{\e^{2-p}+|\nabla
u|^{2-p}}&\leq\frac{\e^{2-p}|\nabla u|^p-|\nabla u|^2}{2|\nabla
u|^{2-p}}=\frac{\e^{2-p}-|\nabla u|^{2-p}}{2|\nabla u|^{2(1-p)}}\\&=
\frac{1}{2}\left(\e^{2-p}-|\nabla u|^{2-p}\right)|\nabla
u|^{2(p-1)}\leq\frac{\e^p}{2},
\end{split}
\end{equation*}
hence we can take $\varphi_0\equiv1$.

\noindent (b) Since $p-2<0$, it follows that $\left(|\nabla
u|^2+\e^2\right)^{(p-2)/2}\leq|\nabla u|^{p-2}$, hence
\begin{equation*}
\big|a(x,t,u,\nabla u)\big|=\left(|\nabla u|^2+\e^2\right)^{(p-2)/2}|\nabla
u|\leq|\nabla u|^{p-1}.
\end{equation*}
Joining the inequalities in (a) and (b) and taking into account that
$u$ is H\"older continuous (cf. \cite{DUV} and Appendix A2), the
family of Dirichlet problems $(P_{\e})$ that we consider satisfies
the assumptions of Theorem 1.2, Chapter 4 of \cite{DiB} in an
uniform way, independent on $\e$, since the boundary and initial
data are H\"older continuous with the same exponent as $u$. We
conclude then that the family $\{u_{\e}\}$ is uniformly H\"older
continuous up to the boundary in $\overline{Q'}$. By the
Arzel\`a-Ascoli Theorem, we obtain that, eventually passing to a
subsequence, $u_{\e}\to\tilde{u}$ uniformly in $\overline{Q'}$.

\medskip

\noindent \textsc{Step 4. Passing to the limit in $(P_{\e})$}. The
strategy will be the following: we pass to the limit $\varepsilon\to
0$ in the weak formulation \eqref{local.eps} for $(P_{\e})$, in
order to get the local weak formulation \eqref{weak} for the
original problem. We can pass to the limit in the terms without
gradients using the uniform convergence proved in the previous step.
On the other hand, we recall that $\{|\nabla u_{\e}|\,:\e>0\}$ is
uniformly bounded in $L^{\infty}\big(T_1,T_2;
L^{p}_{loc}(B_R)\big)$, by Step 2, then, up to subsequences, there
exists $v$ such that, $\nabla u_{\e}\to v$ weakly in
$L^q\big(T_1,T_2 ; L^{p}_{loc}(B_R)\big)$, for any $1\le q<+\infty$.
Next, we can identify $v=\nabla\tilde{u}$, which gives that
$u_{\e}\to\tilde{u}$ in
$L^{\infty}\big(T_1,T_2;W^{1,p}_{loc}(B_R)\big)$. From this, we can
pass to the limit also in the term containing gradients in the local
weak formulation of $(P_{\e})$.

From the uniform convergence in $\overline{Q^{'}}$ (cf. Step 3) and
the considerations above, we deduce that the limit $\tilde{u}$ is
actually a continuous weak solution of the following Dirichlet
problem
\begin{equation}\label{DP}
(DP) \left\{\begin{array}{ll}v_{t}=\Delta_{p}v, \ \hbox{in} \
Q',\\v(x,T_1)=u(x,T_1), \ \hbox{for} \ \hbox{any} \ x\in
B_R,\\v(x,t)=u(x,t), \ \hbox{for} \ x\in\partial B_{R}, \
t\in(T_1,T_2).\end{array}\right.
\end{equation}
On the other hand, the continuous local weak solution $u$ is a
solution of the same Dirichlet problem. By comparison (that holds,
since both solutions are continuous up to the boundary), it follows
that $u=\tilde{u}$. We have thus proved that our approximation
converges to the continuous solutions of the $p$-Laplacian equation.

\medskip

\noindent \textsc{Step 5: Convergence in measure of the gradients}.
In this step, we will improve the convergence of $\nabla u_{\e}$ to
$\nabla u$. More precisely, we prove that the gradients converge in
measure, which is stronger than the weak $L^p$ convergence
established in the previous steps. We follow ideas from the paper
\cite{V6}, having as starting point the following inequality for
vectors $a$, $b\in\real^n$
\begin{equation}\label{ineq.vect2}
(a-b)\cdot(|a|^{p-2}a-|b|^{p-2}b)\ge
c_p\frac{|a-b|^2}{|a|^{2-p}+|b|^{2-p}},
\end{equation}
for some $c_p>0$ for all $1<p<2$. This inequality is proved in
Appendix A3 with optimal constant $c_p=\min\{1,2(p-1)\}$. To prove
the convergence in measure, take $\lambda>0$ and decompose as in
\cite{V6}
\begin{equation*}
\begin{split}
\{|\nabla u_{\e_1}-\nabla u_{\e_2}|>\lambda\}&\subset\Big\{\,\{|\nabla
u_{\e_1}|>A\}\cup\{|\nabla
u_{\e_2}|>A\}\cup\{|u_{\e_1}-u_{\e_2}|>B\}\Big\}\\
&\cup\Big\{|\nabla u_{\e_1}|\le A, \ |\nabla u_{\e_2}|\le A, \
|\nabla u_{\e_1}-\nabla u_{\e_2}|>\lambda, \ |u_{\e_1}-u_{\e_2}|\le
B\Big\}:=S_1\cup S_2,
\end{split}
\end{equation*}
for any $\e_1$, $\e_2>0$ and for any $A>0$, $B>0$ and $\lambda>0$;
we will choose $A$ and $B$ later. Since $\{\nabla u_{\e}\,:\e>0\}$
is uniformly bounded in $L^p(B_R)$, for $t$ fixed, and that
$\{u_{\e}\}$ is Cauchy in the uniform norm, for any $\delta>0$
given, we can choose $A=A(\delta)>0$ sufficiently large and
$B=B(\delta)>0$ such that $|S_1|<\delta$. On the other hand, in
order to estimate $|S_2|$, we observe that
\begin{equation*}
S_2\subset\Big\{|u_{\e_1}-u_{\e_2}|\le B, \ (\nabla u_{\e_1}-\nabla
u_{\e_2})\cdot(|\nabla u_{\e_1}|^{p-2}\nabla u_{\e_1}-|\nabla
u_{\e_2}|^{p-2}\nabla u_{\e_2})\ge
\frac{C\lambda^2}{2A^{2-p}}\Big\},
\end{equation*}
where we have used the definition of $S_2$ and the inequality
\eqref{ineq.vect2}. Letting $\mu=C\lambda^2/2A^{2-p}$ and estimating
further, we obtain
\begin{equation*}
\begin{split}
|S_2|&\le\frac{1}{\mu}\iint\limits_{\{|u_{\e_1}-u_{\e_2}|\le
B\}}(\nabla u_{\e_1}-\nabla u_{\e_2})\cdot(|\nabla
u_{\e_1}|^{p-2}\nabla u_{\e_1}-|\nabla u_{\e_2}|^{p-2}\nabla
u_{\e_2})\,\dx\,\dt\\
&\leq\frac{1}{\mu}\int_{T_1}^{T_2}\int_{B_R}(u_{\e_1}-u_{\e_2})(\Delta_{p}u_{\e_1}-\Delta_{p}u_{\e_2})\,\dx\,\dt,
\end{split}
\end{equation*}
where the integration by parts does not give boundary integrals,
since $u_{\e_1}=u_{\e_2}=u$ on the parabolic boundary of the
cylinder $Q^{'}$. From the previous steps, we can replace
$\Delta_pu_{\e_i}$ by
$\Delta_{\Phi_{\e_i}}u_{\e_i}=\partial_tu_{\e_i}$, $i=1,2$ without
losing too much (less than $\delta/3$ for $\e_1$, $\e_2$
sufficiently small), and the last estimate becomes
\begin{equation*}
|S_2|\le\frac{1}{2\mu}\int_{B_R}\int_{T_1}^{T_2}
\left[\frac{\rd}{\dt}(u_{\e_1}-u_{\e_2})^{2}\right]\,\dx\,\dt+\frac{2\delta}{3}\le\delta,
\end{equation*}
for $\mu$ sufficiently large (or, equivalently, for $\lambda>0$
sufficiently large) and for $\e_1$, $\e_2<\e=\e(\delta)$
sufficiently small. This proves that for any $\delta>0$, there exist
$\lambda=\lambda(\delta)>0$ and $\e=\e(\delta)>0$ such that
\begin{equation*}
\big|\{|\nabla u_{\e_1}|-|\nabla u_{\e_2}|>\lambda\}\big|\le\delta,
\quad \forall\, \e_1, \e_2<\e(\delta), \ \lambda>\lambda(\delta),
\end{equation*}
that is, the family $\{\nabla u_{\e}\}$ is Cauchy in measure, hence
convergent in measure. The limit coincides with the already
established weak limit, which is $\nabla u$.

\medskip

\noindent \textsc{Step 6: Passing to the limit in the inequality}.
We have already proved that the weak solution $u_{\e}$ of $(P_{\e})$
satisfies the inequality
\begin{equation}\label{ineq.eps}
\frac{\rd}{\dt}\int_{B_R}\left(\e^2+|\nabla
u_{\e}|^{2}\right)^{\frac{p}{2}}\varphi\,\dx
+\frac{p}{n}\int_{B_R}(\Delta_{\Phi_{\e}}u)^{2}\varphi\,\dx\leq\frac{p}{2}\int_{B_R}\left(\e^2+|\nabla
u_{\e}|^{2}\right)^{p-1}\Delta\varphi\,\dx,
\end{equation}
where $\Phi_{\e}(w)=\frac{2}{p}\left(w+\e^2\right)^{\frac{p}{2}}$.
From the previous step we know that $\nabla u_{\e}\to\nabla u$ in
measure, hence, by passing to a suitable subsequence if necessary,
the convergence is also true a.\,e. in $Q^{'}$. From this fact, we
obtain that
\begin{equation}\label{in1}
\frac{\rd}{\dt}\int_{B_R}\left(\e^2+|\nabla
u_{\e}|^{2}\right)^{\frac{p}{2}}\varphi\,\dx\to\frac{\rd}{\dt}\int_{B_R}|\nabla
u|^{p}\varphi\,\dx\,,
\end{equation}
as $\e\to0$, in distributional sense in $\cd^{'}(T_1,T_2)$, for any
suitable test function $\varphi$. On the other hand, the continuous
embedding $L^{p}(B_R)\subset L^{2(p-1)}(B_R)$, valid since
$2(p-1)<p$ whenever $1<p<2$, implies
\begin{equation*}
\int_{B_R}|\nabla u_{\e}|^{2(p-1)}\varphi\,\dx\leq
C\int_{B_R}|\nabla u_{\e}|^{p}\varphi\,\dx,
\end{equation*}
with a positive constant $C$ independent of $u$, and for any
suitable test function $\varphi$.   We can easily see that the sequence
$|\nabla u_{\e}|$ is weakly convergent in $\LL^p(B_R)$, since
\begin{equation}\label{ineq.1}
\int_{B_R}\left(\e^2+|\nabla u_{\e}|^2\right)^{p-1}\dx
    \le \int_{B_R}\big(1+|\nabla u_{\e}|^{2(p-1)}\big)\dx
    \le C\left(\int_{B_R}1+|\nabla u_{\e}|^p\right)\dx\le K<+\infty\,,
\end{equation}
where in the last step we have used inequality
\eqref{step2.ineq} of Step 2, and $K$ does not depend on $\varepsilon>0$.
It is a well known fact that if a sequence is uniformly bounded in $\LL^p(B_R)$ and converges
in measure, then it converges strongly in any $\LL^q(B_R)$, for any $1\le q<p$, and
in particular for $q=2(p-1)<p$, whenever $p<2$.
The same holds for $\left(\e^2+|\nabla
u_{\e}|^2\right)^{p-1}$, by inequality \eqref{ineq.1}. Summing up, we have proved that
\begin{equation}\label{in2}
\int_{B_R}\left(\e^2+|\nabla
u_{\e}|^{2}\right)^{p-1}\Delta\varphi\,\dx\to\int_{B_R}|\nabla
u|^{2(p-1)}\Delta\varphi\,\dx\,.
\end{equation}
It remains to analyze the second term in \eqref{ineq.eps}, which is
bounded as a difference of the other two terms, and this implies
that $u_{\e,t}$ is uniformly bounded in
$L^2([T_1,T_2];L^{2}_{loc}(B_R))$. Up to subsequences, there exists
$v\in L^2(T_1,T_2;L^{2}_{loc}(B_R))$ such that $u_{\e,t}\to v$
weakly in $L^2(T_1,T_2;L^{2}_{loc}(B_R))$ and we can identify easily
$v=u_t$. Using the weak lower semicontinuity of the (local) $L^2$
norm, we obtain:
\begin{equation}\label{in3}
\liminf\limits_{\e\to0}\int_{B_R}(\Delta_{\Phi_{\e}}u)^{2}\varphi\,\dx=\liminf\limits_{\e\to0}\int_{B_R}u_{\e,t}^{2}\varphi\,\dx\ge
\int_{B_R}u_{t}^{2}\varphi\,\dx=\int_{B_R}(\Delta_{p}u)^{2}\varphi\,\dx,
\end{equation}
that finally implies inequality \eqref{mainid} for the solution $u$
in $B_R$. Since the ball $B_R$ and the time interval $[T_1,T_2]$
were arbitrarily chosen, we obtain \eqref{mainid} as in the
statement of the theorem.\qed

\noindent \textbf{Remarks}. (i) From \eqref{mainid}, we deduce
directly that $u_t\in L^2(0,T ; L^{2}_{loc}(\Omega))$, which is an
improvement with respect to the $L^1_{loc}$ regularity.

\noindent(ii) A closer inspection of the proof reveals that with
minor modifications we can prove the inequality \eqref{mainidPhi} of
Proposition \ref{Ineq.Phi} also for general nonnegative $\Phi$, thus
allowing degeneracies and singularities of the corresponding
$\Phi$-Laplacian equation.  More precisely, let us consider
nonnegative functions $\Phi$ satisfying the following inequalities:
\begin{equation*}
\Phi'(|\nabla u|^2)|\nabla u|^2\geq C_0|\nabla u|^p-\psi_{0}(x,t),
\end{equation*}
and
\begin{equation*}
|\Phi'(|\nabla u|^2)||\nabla u|\geq C_1|\nabla
u|^{p-1}+\psi_{1}(x,t),
\end{equation*}
where $C_0$, $C_1>0$ and $\psi_0$, $\psi_1$ are nonnegative
functions such that $\psi_0\in L^{s}(0,T;L^{q}(\Omega))$ and
$\psi_1^{p/(p-1)}\in L^{s}(0,T;L^{q}(\Omega))$, where $1\leq
s,q\leq\infty$ and
\begin{equation*}
\frac{1}{s}+\frac{n}{p\,q}<1.
\end{equation*}
These technical hypothesis appear in DiBenedetto's book \cite{DiB}.

\subsection{Local upper bounds for the energy}

In this subsection we derive local upper energy estimates, as
an application of Theorem \ref{inequality}.

\begin{theorem}\label{main-grad}
Let $u$ be a continuous local weak solution of the fast
$p$-Laplacian equation, with $1<p<2$, as in Definition
\ref{local.weak}, corresponding to an initial datum $u_0\in L^r_{\rm
loc}(\Omega)$, where $\Omega\subseteq\real^n$ is any open domain
containing the ball $B_{R_0}(x_0)$. Then, for any $0\leq s\leq t$,
and any $0<R<R_0$ and any $x_0\in \Omega$ such that
$B_{R_0}(x_0)\subset\Omega$\,, the following inequality holds true:
\begin{equation}\label{gradientmain}
\left[\int_{B_{R}(x_0)}|\nabla
u(x,t)|^{p}\,\dx\right]^{(2-p)/p}\leq\left[\int_{B_{R_{0}}(x_0)}|\nabla
u(x,s)|^{p}\,\dx\right]^{(2-p)/p}+K(t-s),
\end{equation}
where the positive constant $K$ has the form
\begin{equation}
K=\frac{C_{p,n}}{(R_{0}-R)^{2}}|B_{R_{0}}\setminus B_{R}|^{(2-p)/p},
\end{equation}
and where $C_{p,n}$ is a positive constant depending only on $p$ and
$n$.
\end{theorem}

\begin{proof}
We begin with inequality \eqref{mainid} and we drop the first term
in the right-hand side, which is nonpositive:
\begin{equation*}
\frac{\rd}{\dt}\int_{\Omega}|\nabla
u|^{p}\varphi\,\dx\leq\frac{p}{2}\int_{\Omega}|\nabla
u|^{2(p-1)}\Delta\varphi\,\dx.
\end{equation*}
An application of H\"older inequality, with conjugate exponents
$p/2(p-1)$ and $p/(2-p)$, leads to
\begin{equation}\label{gradinterm}
\frac{\rd}{\dt}\int_{\Omega}|\nabla u|^{p}\varphi\,\dx\leq
C(\varphi)\left[\int_{\Omega}|\nabla
u|^{p}\varphi\,\dx\right]^{2(p-1)/p},
\end{equation}
where
\begin{equation*}
C(\varphi)
=\frac{p}{2}\left[\int_{\Omega}|\Delta\varphi|^{\frac{p}{2-p}}
\varphi^{-\frac{2(p-1)}{2-p}}\,\dx\right]^{(2-p)/p}<+\infty\,,
\end{equation*}
since has the same expression as in \eqref{c.phi}. An integration
over $(s,t)$ gives
\begin{equation*}
\left[\int_{\Omega}|\nabla
u(x,t)|^{p}\varphi(x)\,\dx\right]^{(2-p)/p}\leq\left[\int_{\Omega}|\nabla
u(x,s)|^{p}\varphi(x)\,\dx\right]^{(2-p)/p}+\frac{(2-p)}{p}C(\varphi)(t-s).
\end{equation*}
We conclude by observing that the constant $C(\varphi)$ is exactly
the same as \eqref{c.phi} and thus we can repeat the same
observation made there to express it in the desired form.\end{proof}

\noindent\textbf{Remarks.} (i) It is essential in the above
inequality that $p<2$, since the constant explodes in the limit
$p\to2$. Indeed such kind of estimates are false for the heat
equation, that is for $p=2$.

\noindent (ii) The constant also explodes when $R/R_{0}\to 1$.
Indeed,
\begin{equation*}
K\sim C\frac{(R_{0}^{n}-R^{n})^{(2-p)/p}}{(R_{0}-R)^{2}}\sim
C(R_{0}-R)^{(2-3p)/p}.
\end{equation*}

\subsection{Lower bounds for the $\LL^1_{\rm loc}-$norm}

In this subsection, we establish local lower bounds for the mass, in
the following form:
\begin{theorem}\label{main-norm1}
Let $u$ be a local weak solution of the fast $p$-Laplacian equation,
with $1<p<2$, as in Definition \ref{local.weak}, corresponding to an
initial datum $u_0\,,\,\left|\nabla u_0\right|^p\in
L^1_{loc}(\Omega)$, where $\Omega\subseteq\real^n$ is any open
domain containing the ball $B_{R_0}(x_0)$. Then, for any $0\leq
t\leq s$ and for any $0<R<R_0$\,, the following inequality holds
true:
\begin{equation}\label{loc.norm.1}
\int\limits_{B_{R}(x_0)}u(x,t)\,\dx
\leq\int\limits_{B_{R_0}(x_0)}u(x,s)\,\dx+C\left[\left(\int_{\Omega}|\nabla
u(x,s)|^{p}\varphi(x)\,\dx\right)^{1/p}+K^{\frac{1}{2-p}}\,|t-s|^{\frac{1}{2-p}}\right],
\end{equation}
where
\begin{equation}
C=\overline{C}_{p,n}(R_{0}-R)|B_{R_0}\setminus
B_{R}|^{\frac{p-1}{p}},\qquad
K=\frac{C_{p,n}}{(R_{0}-R)^{2}}|B_{R_{0}}\setminus B_{R}|^{(2-p)/p},
\end{equation}
with $\overline{C}_{p,n}$ and $C_{p,n}$ depending only on $p$ and $n$.
\end{theorem}

\noindent \textbf{Remarks.} (i) The limits as $R\to+\infty$ give
mass conservation for the Cauchy problem, when $p_c<p<2$, while in
the subcritical range $1<p<p_c$ it indicates how much mass is lost
at infinity. This estimates are new as far as we know.

\noindent (ii) The estimate \eqref{loc.norm.1} is different from
that of Theorem \ref{main-normr}, since it applies for $0\leq t\leq
s$, and provides a local lower bound for the mass. Moreover, it
allows us to obtain lower bounds for the finite extinction time $T$,
in the cases it occurs, just by letting $s=T$ and $t=0$ in
\eqref{loc.norm.1}, as follows:
\begin{equation}\label{low.FET}
C^{p-2}\,K^{-1}\left[\int_{B_{R}(x_0)}u_0(x)\right]^{2-p}\,\dx \leq
\,T\,.
\end{equation}

\begin{proof}
We begin by performing a time derivative of the local mass
\begin{equation*}\begin{split}
\frac{\rd}{\dt}\int_{\Omega}u(x,t)\varphi(x)\,\dx
&=\int_{\Omega}\hbox{div}\left(|\nabla u(x,t)|^{p-2}\nabla
u(x,t)\right)\varphi(x)\,\dx\\
&\geq-\int_{\Omega}|\nabla u(x,t)|^{p-1}|\nabla\varphi(x)|\,\dx.
\end{split}
\end{equation*}
We then estimate the right-hand side by applying H\"older inequality
with exponents $(p-1)/p$ and $1/p$, then using Lemma
\ref{choice.varphi} with $\a=p$. Integrating over the time interval
$(t,s)$, we obtain the desired estimate. We leave the
straightforward details to the reader.
\end{proof}

\section{Panorama, open problems and existing literature}
\label{sec.panorama}

We recall here the values of $p_c=2n/(n+1)$ and of the critical line $r_c=\max\{n(2-p)/p\,,\,1\}$.

\begin{figure}[ht!]
\centering
\includegraphics[width=11cm,height=7cm]{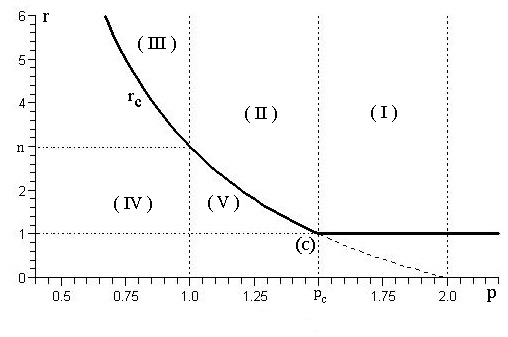}
\caption{The critical line for the Smoothing Effect}\label{Fig}
\end{figure}

\begin{itemize}
\item[(I)] \textsl{Good Fast Diffusion Range:}
$p\in(p_c,2)$ and $r\ge 1$. In this range the local smoothing effect
holds, cf. Theorem  \ref{mainupper0}, as well as the positivity
estimates of Theorem \ref{mainposit} and the Aronson-Caffarelli type
estimates of Theorem \ref{main.AC}. The intrinsic forward/backward/elliptic Harnack
inequality Theorem \ref {FBH} holds in this range. This is the only range in
which there are some other works on Harnack inequalities. Indeed in
the pioneering work of DiBenedetto and Kwong \cite{DiBK92} there
appeared for the first time the intrinsic Harnack inequalities for
fast diffusion processes related to the $p$-Laplacian, now
classified as forward Harnack inequalities. See also \cite{DUV} for
an excellent survey on these topics. In a recent paper DiBenedetto,
Gianazza and Vespri \cite{DGV} improve on the previous work by
proving elliptic, forward and backward Harnack inequalities for more
general operators of $p$-Laplacian type, but always in this ``good
range''.

\item[(II)] \textsl{Very Fast Diffusion
Range:} $p\in(1,p_c)$ and $r\ge r_c>1$. In this range the local
smoothing effect holds, cf. Theorem  \ref{mainupper0}, as well as
the positivity estimates of Theorem \ref{mainposit} and the
Aronson-Caffarelli type estimates of Theorem \ref{main.AC}.
The intrinsic forward/backward/elliptic Harnack
inequality Theorem \ref {FBH} holds in this range as well,
showing that if one allow the constants to depend on
the initial data, then the form of Harnack inequalities is the same.
No other kind of positivity, smoothing or Harnack estimates are
known in this range, and our results represent a breakthrough in the
theory of the singular $p$-Laplacian, indeed in \cite{DGV}
there is an explicit counterexample that shows that Harnack inequalities of
backward, forward or elliptic type, are not true in general in this
range, if the constants depend only on $p$ and $n$.

The open question is now: If one wants absolute constants, what is
the relation between the supremum and the infimum, if any?

\item[(c)] \textsl{Critical Case:} $p=p_c$ and $r>r_c=1$. The local upper
and lower estimates of zone (II) apply, as well as the Harnack
inequalities. As previously remarked, all of our results are stable
and consistent when $p=p_c$.

\item[(III)] and (IV) \textsl{Very Singular Range:} $0<p\le 1$ with $r>r_c$ or
$0<p\le 1$ with $r<r_c$. In the range $p<1$ the multidimensional $p$-Laplacian formula
does not produce a parabolic equation. A theory in one dimension has been started
in \cite{BarVa, RodVa}, while radial self-similar solutions in several dimensions are classified
in \cite{IS}. For reference to $p=1$, the so-called total variation flow, cf. \cite{ACM, CBN}.

\item[(V)] \textsl{Very Fast Diffusion Range:}
$1<p<p_c$ and $r\in[1,r_c]$. It is
well known that the smoothing effect is not true in general, since
initial data are not in $\LL^p$ with $p>p_c$, cf. \cite{VazquezSmoothing}.
Lower estimates are as in (II). In general, Harnack
inequalities are not possible in this range since solution may no be
(neither locally) bounded.
\end{itemize}

\subsection{Some general remarks}
\label{ss.sgr}

\begin{itemize}
\item We stress  the fact that our results are
completely local, and they apply to any kind of initial-boundary
value problem, in any Euclidean domain: Dirichlet, Neumann, Cauchy,
or problem for large solutions, namely  when $u=+\infty$ on the
boundary, etc. Natural extensions are fast diffusion problems for more general $p$-Laplacian operators and fast diffusion problems on manifolds.

\item We calculate (almost) explicitly
all the constants, through all the paper.

\item We have not entered either into the derivation of H\"older
continuity and further regularity from the Harnack inequalities.
This is a subject extensively treated in the works of DiBenedetto et
al., see \cite{DUV, DiB, DGV} and references therein.

\item Summing up, no other results but ours are known in the lower range $p\le p_c$,
and essentially one is known in the good range, and it refers to a
different point of view.

\item A combination of the techniques developed in this paper and in \cite{BV},
allow to extend the local smoothing effects,
or the positivity estimates as well as the intrinsic Harnack inequalities
to the \textsl{doubly nonlinear equation}
\begin{equation*}
\partial_t u=\Delta_{p}u^{m},
\end{equation*}
for which the fast diffusion range is understood as the set of
exponents $m>0$ and $p>1$ such that $m(p-1)\in(0,1)$. Basic
existence, uniqueness and regularity results on this equation, that
allow for extensions of our results, appear in
\cite{EstebanVazquez1} and in \cite{IU}. We will not enter into
the analysis of the extension in this paper.
\end{itemize}

\section*{Appendix}
\setcounter{section}{10}

\subsection*{A1. Choice of particular test functions}

In this appendix we show how we choose special test functions
$\varphi$ in various steps of the proof of our Local Smoothing
Effect. We express these technical results in the form of the
following
\begin{lemma}\label{choice.varphi}
(a) For any open set $\Omega\subset\real^n$, for any two balls
$B_{R}\subset B_{R_0}\subset\Omega$, and for any $\a>0$, there
exists a test function $\varphi\in C_{c}^{\infty}(\Omega)$ such that
\begin{equation}\label{prop.test}
0\leq\varphi\leq1, \qquad \varphi\equiv1 \ {\rm in} \ B_{R}, \qquad
\varphi\equiv0 \ {\rm outside} \ B_{R_0}\,,
\end{equation}
and
\begin{equation}\label{Const.Phi.Append.1}
\int\limits_{\Omega}|\nabla\varphi|^{\a}\varphi^{1-\a}\,\dx<\frac{C_1}{(R_0-R)^{\a}}|A|<\infty,
\end{equation}
where $C_1(n,\alpha)$ is a positive constant and $A=B_{R_0}\setminus
B_R$.

\noindent (b) In the same conditions as in part (a), for any $\b>0$,
there exists a test function $\varphi\in C_{c}^{\infty}(\Omega)$
satisfying \eqref{prop.test} and such that
\begin{equation*}
\int\limits_{\Omega}|\Delta\varphi|^{\b}\varphi^{1-\b}\,\dx<\frac{C_2}{(R_0-R)^{2\b}}|A|<\infty,
\end{equation*}
where $C_2(n,\beta)$ is a positive constant.
\end{lemma}
\begin{proof}
Let $\psi$ be a radially symmetric $C_{c}^{\infty}$ function which
satisfies \eqref{prop.test}. It is easy to find $\psi$ (see also
\cite{BV}) satisfying the following estimates:
\begin{equation*}
|\nabla\psi(x)|\leq\frac{K_1}{(R_0-R)}, \qquad
|\Delta\psi(x)|\leq\frac{K_2}{(R_0-R)^2},
\end{equation*}
where $K_1$ and $K_2$ are positive constants depending only on
 $n$. Take $\varphi=\psi^{\gamma}$, where $\gamma>0$ will be chosen later. It
is clear that $\varphi$ satisfies \eqref{prop.test}. We calculate:
\begin{equation*}
|\nabla\varphi|=\gamma\psi^{\gamma-1}|\nabla\psi|, \qquad
\Delta\varphi=\gamma\psi^{\gamma-1}\Delta\psi+\gamma(\gamma-1)\psi^{\gamma-2}|\nabla\psi|^2.
\end{equation*}
In order to prove part (a), we take  $\gamma\ge
\max\{1,\a\}$ and we remark that $\nabla\varphi$ is
supported in the annulus $A$ to estimate:
\begin{equation*}
\int\limits_{\Omega}|\nabla\varphi|^{\a}\varphi^{1-\a}\,\dx\leq
\gamma^{\a}\int\limits_{A}\psi^{\gamma-\a}\frac{K_1^{\a}}{(R_0-R)^{\a}}
\,\dx<C_1(n)\frac{(K_1\gamma)^{\alpha}}{(R_0-R)^{\a}}|A|\,.
\end{equation*}

In order to prove part (b), we estimate:
\begin{equation*}
|\Delta\varphi|^{\b}\varphi^{1-\b}\leq
c\left[\gamma(\gamma-1)\right]^{\b}\psi^{(\gamma-2)\b+\gamma(1-\b)}(|\Delta\psi|+|\nabla\psi|^2)^{\b}.
\end{equation*}
Thus, choosing $\gamma>\max\{1,2\b\}$ and taking into account that
$\Delta\varphi$ is supported in the annulus $A$, we obtain
\begin{equation*}
\int\limits_{\Omega}|\Delta\varphi|^{\b}\varphi^{1-\b}\,\dx\leq\frac{C_2}{(R_0-R)^{2\b}}|A|\,,
\end{equation*}
where $C_2=C_2(p,n,\beta,\gamma)$ is a positive constant.
\end{proof}

\subsection*{A2. Boundedness, regularity and local comparison}

Let us recall now some well known regularity results for local weak
solutions as introduced in Definition \ref{local.weak}, given in
Theorem 2.25 of \cite{DUV}:
\begin{theorem*}
If $u$ is a bounded local weak solution of {\rm \eqref{PLE}} in
$Q_T$, then $u$ is locally H\"older continuous in $Q_T$. More
precisely, there exist constants $\a\in(0,1)$ and $\gamma>0$ such
that, for every compact subset $K\subset Q_T$, and for every points
$(x_1,t_1)$, $(x_2,t_2)\in K$, we have:
\begin{equation*}
\big|u(x_1,t_1)-u(x_2,t_2)\big|\leq\gamma\|u\|_{L^{\infty}(Q_T)}
\left[\frac{|x_1-x_2|+\|u\|_{L^{\infty}(Q_T)}^{\frac{p-2}{p}}|t_1-t_2|^{\frac{1}{p}}}{{\rm
dist}(K,\partial Q_T)}\right]^{\a},
\end{equation*}
where
\[
{\rm dist}(K,\partial Q_T)=\inf_{(x,t)\in K,
(y,s)\in\partial\Omega}\left\{|x-y|,\|u\|_{L^{\infty}(Q_T)}^{(p-2)/p}|t-s|^{1/p}\right\}\,,
\]
and by $\partial Q_T$ we understand the parabolic boundary of $Q_T$.
The constants $\a$ and $\gamma$ depend only on $n$ and $p$.
\end{theorem*}

\noindent\textbf{Remark. } The above theorem holds whenever $u$ is a
locally bounded function of space and time. We have used this result
just in some technical steps: we begin with bounded local strong
solution, which thanks to the above result are H\"older continuos.
By the way we can prove the smoothing effect for any local strong
solution, independently of this continuity result, we thus obtain a
posteriori that any local strong solution is H\"older continuous.

\medskip

\subsection*{A3. A useful inequality related to the $p$-Laplacian}

We prove the following inequality, used in some technical steps of
the proof of Theorem \ref{inequality}.
\begin{lemma}
For any vectors $a,b\in\real^n$, and for $1< p\le 2$, we have:
\begin{equation}\label{ineq.vect}
(a-b)\cdot(|a|^{p-2}a-|b|^{p-2}b)
    \ge c_p\frac{|a-b|^2}{|a|^{2-p}+|b|^{2-p}},
\end{equation}
where the optimal constant is achieved when $a\cdot b=|a||b|$ and is given by $c_p=\min\{1,2(p-1)\}$, if $1<p<2$, and  $c_2=2$.
\end{lemma}

\begin{proof}
When $p=2$, the inequality becomes a trivial equality with $c_2=2$.
We next assume that $1<p<2$ and we rewrite inequality
\eqref{ineq.vect} as follows
\[
\left(|a|^{2-p}+|b|^{2-p}\right)\left[|a|^p+|b|^p-\left(|a|^{p-2}+|b|^{p-2}\right)a\cdot b\right]
    \ge c_p \left(|a|^{2}+|b|^{2}-2a\cdot b\right)
\]
or, equivalently, in the form
\[
(1-c_p)\left(|a|^{2}+|b|^{2}-2a\cdot b\right) + |a|^{2-p}|b|^p+|a|^p|b|^{2-p}
    -\left(\frac{|a|^{2-p}}{|b|^{2-p}}+\frac{|b|^{2-p}}{|a|^{2-p}}\right)a\cdot b
    \ge 0
\]
that can be reduced to
\[
\left(\frac{|a|^{2-p}}{|b|^{2-p}}+\frac{|b|^{2-p}}{|a|^{2-p}}+2(1-c_p)\right)a\cdot b
    \le |a|^{2-p}|b|^p+|a|^p|b|^{2-p}
     +(1-c_p)\left(|a|^2+|b|^2\right).
\]
Now it is clear that the worst case occurs when $a\cdot b=|a||b|$,
since we always have $a\cdot b \le |a||b|$. Hence, proving
inequality \eqref{ineq.vect} is equivalent to prove the numerical
inequality
\[
|a|^{2-p}|b|^p+|a|^p|b|^{2-p}
     +(1-c_p)\left(|a|-|b|\right)^2
     -\left(\frac{|a|^{2-p}}{|b|^{2-p}}+\frac{|b|^{2-p}}{|a|^{2-p}}\right)|a||b|\ge 0,
\]
when $|a|\ge |b|$. Dividing the above inequality by $|b|^2$ and letting $\lambda=|a|/|b|$, we get
\[
\Phi_p(\lambda)=\lambda^{2-p}+\lambda^p+(1-c_p)(\lambda-1)^2-\lambda^{3-p}-\lambda^{1-p}\ge 0\qquad\mbox{for any $1<p\le 2$ and $\lambda\ge 1$.}
\]
In the range $3/2<p<2$, we can always let $c_p=1$, since
$\lambda^{2-p}+\lambda^p\ge \lambda^{3-p}+\lambda^{1-p}$, and this
guarantees that $\Phi_p(\lambda)\ge 0$; again this constant is
optimal and achieved when $\lambda=1$, that is when $a=b$. When
$p=3/2$, we have $\Phi_{3/2}(\lambda)=(1-c_p)(\lambda-1)^2 \ge 0$,
so the inequality holds again with $c_p=1$. When $1<p<3/2$ we have
to work a bit more. We calculate
\[
\Phi''_p(\lambda)
    =-(2-p)(p-1)\lambda^{-p}+p(p-1)\lambda^{p-2}-(3-p)(2-p)\lambda^{1-p}+(2-p)(p-1)\lambda^{p-3}+2(1-c_p)
\]
and we observe that $\Phi''_p(1)=-6+4p+2(1-c_p)\ge 0$ if $c_p\le
2(p-1)$. Moreover, in the limit $\lambda\to \infty$,
$\Phi''_p(\lambda)\to 2(1-c_p)= 6-4p > 0$, when $1<p<3/2$. Then it
is easy to check that
\[\begin{split}
\Phi'''_p(\lambda)
    &=(p-1)(2-p)\left[p\lambda^{-p-1}-p\lambda^{p-3}+(3-p)\lambda^{-p}-(3-p)\lambda^{p-4}\right]\\
    &\ge p(p-1)(p-2)\left(\frac{1}{\lambda}+1\right)
        \left(\frac{1}{\lambda^{p}}-\frac{1}{\lambda^{3-p}}\right)\ge 0
\end{split}
\]
since $3-p>p$ when $p<3/2$ and $t\ge 1$. We have thus proved that
$\Phi''_p(\lambda)$ is a non-decreasing function of $\lambda$, which
is zero in $\lambda=1$ and  $\Phi_p(\lambda)\le
\Phi_p(\infty)=2(1-c_p)= 6-4p$. This implies that $\lambda=1$ is a
minimum for $\Phi_p$, since $\Phi_p(1)=0$, $\Phi'_p(1)=0$. As a
consequence $\Phi_p(\lambda)\ge 0$ for any $\lambda\ge 1$. Equality
is attained for $\lambda=1$ and $c_p=2(p-1)$, and this fact proves
optimality of the constant when $a=b$.
\end{proof}

\textsc{Acknowledgments.} The authors are funded by Project
MTM2005-08760-C02-01 (Spain). M. Bonforte and J. L. Vázquez are
partially supported by the European Science Foundation Programme
"Global and Geometric Aspects of Nonlinear Partial Differential
Equations".

\bibliographystyle{plain}

\end{document}